\documentclass[twoside]{article}

\usepackage[accepted]{aistats2025}
\usepackage[round]{natbib}

\usepackage[utf8]{inputenc} % allow utf-8 input
\usepackage[T1]{fontenc}    % use 8-bit T1 fonts
\usepackage[hidelinks]{hyperref}       % hyperlinks
\usepackage{url}            % simple URL typesetting
\usepackage{booktabs}       % professional-quality tables
\usepackage{amsfonts}       % blackboard math symbols
\usepackage{nicefrac}       % compact symbols for 1/2, etc.
\usepackage{microtype}      % microtypography
\usepackage[table]{xcolor}         % colors
\usepackage{acronym}
\usepackage{mathtools}
\usepackage{hyperref}
\usepackage{glossaries}
\usepackage{cancel}

\usepackage{algpseudocode}
\usepackage{algorithm}
\usepackage{multirow}

\usepackage{IEEEtrantools, amsthm, amsmath, amssymb}
\usepackage{cleveref}
\usepackage{pifont}% http://ctan.org/pkg/pifont

\PassOptionsToPackage{unicode}{hyperref}

\newcommand*{\R}{\mathbb{R}}
\newcommand*{\norm}[1]{\| #1 \|}
\newcommand*{\innerproduct}[2]{\left \langle #1, #2 \right \rangle}
\newcommand*{\lambdamax}[1]{\lambda_{\max}\left (#1\right )}
\newcommand*{\lambdamin}[1]{\lambda_{\min}\left (#1\right )}

\def\Cone{C1}
\def\Ctwo{C2}
\def\Cthree{C3}
\def\Aone{\textnormal{A1}}
\def\Atwo{\textnormal{A2}}

\def\mf{\mu_f}
\def\mg{\mu_g}
\def\mA{\mu_A}
\def\Lf{L_f}
\def\Lg{L_g}

\def\0{\boldsymbol{0}}
\def\J{\mathrm{J}}
\def\Id{\operatorname{Id}}

\def\gra{\operatorname{gra}}

\def\Fb{F_\textnormal{b}}
\def\Ff{F_\textnormal{f}}
\def\zet{ \zeta_{\tau,\sigma}}
\def\w{{\omega}}
\def\F{{\mathcal F}}
\def\nf{{\nabla f}}
\def\ng{{\nabla g}}
\def\Phiinv{{\Phi_{{\tau,\sigma}} ^{-1}}}
\def\P{{\Phi_{{\tau,\sigma}}}}
\def\diag{\operatorname{diag}}

\DeclareUnicodeCharacter{0342}{}

\DeclareMathOperator*{\dom}{dom}

\newtheorem{theorem}{Theorem}
\newtheorem{definition}{Definition}

\newtheorem{lemma}{Lemma}
\newtheorem{proposition}{Proposition}
\newtheorem{remark}{Remark}

\acrodef{lsc}[l.s.c.]{lower-semicontinuous}

\usepackage{tikz}

\usepackage{enumitem}

\newacronym{PDG}{PDG}{primal-dual gradient}

\makeatletter
\def\Cline#1#2{\@Cline#1#2\@nil}
\def\@Cline#1-#2#3\@nil{%
  \omit
  \@multicnt#1%
  \advance\@multispan\m@ne
  \ifnum\@multicnt=\@ne\@firstofone{&\omit}\fi
  \@multicnt#2%
  \advance\@multicnt-#1%
  \advance\@multispan\@ne
  \leaders\hrule\@height#3\hfill
  \cr}
\makeatother

\begin{document}



\runningauthor{Colin Dirren,
     Mattia Bianchi,
     Panagiotis D.~Grontas,
     John Lygeros,
     Florian D\"orfler}

\twocolumn[

\aistatstitle{Contractivity and linear convergence  in bilinear saddle-point problems: An operator-theoretic approach}

\aistatsauthor{
     Colin Dirren*\And
     Mattia Bianchi* \And
     Panagiotis D.~Grontas \And
     John Lygeros  \And
     Florian D\"orfler}

\aistatsaddress{ ETH Zürich \And ETH Zürich \And ETH Zürich \And ETH Zürich \And ETH Zürich}

%     ETH Zürich, Switzerland\\
%     \texttt{cdirren@.ethz.ch} \\
%     \And
%      \\
%     ETH Zürich, Switzerland\\
%     \texttt{mbianch@ethz.ch} \\
%     \AND
%     Panagiotis D.~Grontas \\
%     ETH Zürich, Switzerland\\
%     \texttt{pgrontas@ethz.ch} \\
%     \And
%     John Lygeros \\ 
%     ETH Zürich, Switzerland\\
%     \texttt{jlygeros@ethz.ch}
%     \And
%     Florian Dörfler \\
%     ETH Zürich, Switzerland \\
%     \texttt{doerfler@ethz.ch}\\\\
% }

% 
]

\begin{abstract}
  We study the convex-concave bilinear saddle-point problem $\min_x \max_y f(x) + y^\top Ax - g(y)$, where both, only one, or none of the functions $f$ and $g$ are strongly convex, and suitable rank conditions on the matrix $A$ hold. The solution of this problem is at the core of many machine learning tasks. By employing tools from monotone operator theory, we systematically prove the contractivity (in turn, the linear convergence) of several first-order primal-dual algorithms, including the Chambolle--Pock method. Our approach results in concise  proofs, and it yields new convergence guarantees and tighter bounds compared to known results. 
\end{abstract}

\section{INTRODUCTION}
\label{sec:introduction}
We consider the bilinear saddle-point problem
\begin{equation}
	\min_{x \in \R^n} \max_{y \in \R^m}  \ f(x) +y^\top Ax -g(y), 
	\label{eq:bilinearSaddlePointProblem}
\end{equation}
where $f: \R^n \rightarrow \R \cup \{\infty\}$ and $g: \R^m \rightarrow \R \cup \{\infty\}$ are proper, convex and closed functions, and $A\in \R^{m \times n}$ is a coupling matrix. The related primal problem is 
\begin{equation}
    \label{eq:primal}
	\min_{x \in \R^n} \ f(x) + g^{*}(Ax),
\end{equation}
where $g^*$ is the Fenchel conjugate of $g$. 
Although many problems can be directly solved using \eqref{eq:primal}, the saddle-point formulation \eqref{eq:bilinearSaddlePointProblem} has favorable properties that allow for efficient or parallel solution, e.g., in case of finite-sum  \citep{wangExploitingStrongConvexity2017} or sparsity \citep{leiDoublyGreedyPrimalDual2017} structure. For this reason, problem \eqref{eq:bilinearSaddlePointProblem} is widely employed, including in empirical risk minimization  \citep{zhangStochasticPrimalDualCoordinate2015, shalev-shwartzAcceleratedProximalStochastic2013}, unsupervised learning \citep{xuMaximumMarginClustering2004}, reinforcement learning \citep{duStochasticVarianceReduction2017}, robust optimization \citep{ben-talRobustOptimization2009}, inverse imaging tasks \citep{chambolleFirstOrderPrimalDualAlgorithm2011}, 
extensive-form games \citep{farinaCorrelationExtensiveFormGames2019} 
and compressed sensing \citep{bachConvexSparseMatrix2008, fanPrimalDualActive2014}.  

Here, we aim at solving \eqref{eq:bilinearSaddlePointProblem}
with linear iteration complexity, under suitable assumptions. More specifically, we are interested in algorithms that are \emph{contractive}. Contractivity is a highly desirable property, that not only implies Q-linear convergence, but also ensures strong robustness properties, with respect to  parameters and data, for instance in case of inexact updates (see \Cref{sec:preliminaries} below or Chapter~3 of \cite{FB-CTDS}). 
To this goal, we will consider either of the three following alternative conditions:
%  	 (where $\lambdamin{}$ and $\lambdamax{}$ denote minimum and maximum eigenvalues):
\begin{itemize} 
	\item[\Cone.] $f$ is $\mf$-strongly convex, $g$ is $\mg$-strongly convex, $\mf>0$, $\mg>0$.
	\item[\Ctwo.] $g$ is $\mg$-strongly convex and $\Lg$-smooth,  $\mg>0$, $\Lg>0$;  $\mA \coloneqq \lambdamin{A^\top A}>0$.
	\item[\Cthree.] $f$ is $\Lf$-smooth, $g$ is $\Lg$-smooth,  $\Lf\geq 0$, $\Lg \geq 0$; $n =m$ and $\mA \coloneqq  \lambdamin{A^\top A} = \lambdamin{A A^\top } >0$. 
\end{itemize}

 Each of these conditions ensures the strong convexity of the primal problem \eqref{eq:primal}, and also the existence of a unique solution $(x^\star,y^\star)$ to problem \eqref{eq:bilinearSaddlePointProblem} (see a proof in \Cref{sec:uniquenessofprimaldual}). Furthermore, each condition is tight, in the sense that uniqueness of a primal-dual solution to \eqref{eq:bilinearSaddlePointProblem} cannot be guaranteed if any of the subconditions is relaxed---in which case, contractivity is excluded. Conditions \Cone{}, \Ctwo{} and \Cthree{} are relevant for a variety of applications, as exemplified next.

    \paragraph{Imaging.}  
    A common refinement of the Rudin, Osher and Fatemi (ROF) model for image recovery, particularly useful in avoiding the staircasing effect, is the total variation Huber--ROF  model \citep{Nikolova2005,Heise2013PMHuber}. It is   given by $\min_{x \in \R^n } \ \ell^{\textrm{hub}}_\alpha (\Delta x) + \lambda \norm{x - \hat{x}}^2 $,
    	where $ \hat{x} \in \R^n $ is the (vectorized) noisy image, $ \lambda > 0 $ weighs the trade-off between fitting and regularization, $\Delta  \in \R^{m \times n} $ is the linear difference operator between adjacent pixels, and \( \ell^{\textrm{hub}}_\alpha\) is the Huber  regularization with parameter $\alpha>0$.  Its primal-dual reformulation is \citep[Eq.~71]{chambolleFirstOrderPrimalDualAlgorithm2011}
    	  \begin{equation}\label{eq:HuberPrimalDual}
    		\min_{x \in \R^n} \max_{y \in \R^m}  \ \frac{\lambda}{2}\|x-\hat x\|^2  - y^\top  \Delta x - \frac{\alpha}{2}\|y\|^2 -  \iota_P(y),
    	\end{equation} 
    where $\iota_P$ is the indicator function of the polar ball. Primal-dual proximal-based methods have been shown to be extremely efficient for the Huber-ROF model,
    thanks to the existence of closed-form expressions for the involved proximal operators.  Clearly, problem \eqref{eq:HuberPrimalDual} satisfies \Cone{}.

    \paragraph{Affinely constrained optimization.} 
    The constrained  convex optimization problem 
            \begin{equation}
                \label{eq:minimizationAffineConstraints}
                \min_{A^\top y  \geq  b} g(y) =  - \max_{A^\top y  \geq  b} -g(y),
            \end{equation} 
            arises in a variety of applications,  including support vector machines \citep{scholkopf2002learning}, constrained regression \citep{monfort1982likelihood}, model predictive control \citep{rawlings2017model}.
             % \citep{chambolleFirstOrderPrimalDualAlgorithm2011}, sketched-learning types applications \citep{kerivenSketchingLargescaleLearning2016}, network flow optimization \citep{zarghamAcceleratedDualDescent2014} and optimal transport \citep{peyreComputationalOptimalTransport2019} and has therefore various applications. 
             Its  saddle-point formulation is obtained by setting $f(x) = -x^\top  b + \iota_{\R^m_{\geq 0}}(x)$ in \eqref{eq:bilinearSaddlePointProblem} (note that $f$ is non-smooth, as it is constrained on the non-negative orthant).  Often, state-of-the-art algorithms solve \eqref{eq:bilinearSaddlePointProblem} instead of \eqref{eq:minimizationAffineConstraints}, since the primal-dual formulation allows for efficient distributed computation \citep{kovalevOptimalPracticalAlgorithms2020} and  avoids projecting onto $A^\top y  \geq b$ at each iteration \citep{salimOptimalAlgorithmStrongly2022}. Whenever $g$ is smooth and strongly convex and $A^\top$ is full-row rank,
%             (for instance,  if $A$ is a vector of ones, as typical for resource allocation problems), 
			\Ctwo{} is satisfied. This also holds for the case of equality constraints $A^\top y = b$ (e.g., in network flow problems \citep{Jiang2021}).
    \color{black}    
    
    \paragraph{Reinforcement Learning.}
         Policy evaluation is a  key task in reinforcement learning. It consists in approximating the value function  $V^{\pi}$, namely the expected cumulative reward achieved by a given policy \( \pi \), 
    \begin{equation*}
        V^{\pi}(s) = \mathbb{E} \Big [ \sum\nolimits_{t=0}^{\infty} \gamma^t R(s_t, a_t) \vert s_0 = s, a_t \sim  \pi(s_t) \Big ], 
    \end{equation*}
    where $\gamma \in (0, 1)$ is a discount factor and $ R(s_t, a_t)$ is the reward for taking action $a_t$ in state $s_t$. A common approach for policy evaluation is to use a linear function approximation  $V^{\pi}(s) \approx \phi(s)^\top x$, where $\phi$ is a feature map of the state and the weights $x\in\R^n$ can be estimated  by minimizing the empirical mean squared projected Bellman error, for a dataset of length $T$ \citep{duStochasticVarianceReduction2017}:
    \begin{equation*}
        \min_x \ \norm{Ax-b}_{C^{-1}}^2,
    \end{equation*}
    with $C = \sum_{t=0} ^T  \phi(s_t) \phi(s_t)^\top $, $b =  \sum_{t=0} ^T  r_t \phi(s_t)$, $A = C - \gamma \sum_{t=0} ^T \phi(s_t) \phi(s_{t+1})^\top$. Using gradient methods to solve this problem would require inverting $C$. The corresponding primal-dual formulation 
    \begin{equation}
        \label{eq:RLPrimalDual}
        \min_x\max_y \  - y^\top Ax - \frac{1}{2}\norm{y}_C^2-y^\top b,
    \end{equation}
    is often preferred, as computing gradients does not require matrix inversion and thanks to its finite-sum structure \citep{duStochasticVarianceReduction2017}. Problem \eqref{eq:RLPrimalDual} is in the form \eqref{eq:bilinearSaddlePointProblem}, and it satisfies \Cthree{}  if $A$ is invertible.

    %    \begin{table*}
    % 	\caption{Alternative assumptions for linear convergence used in this paper (``s.c.'' stands for  strongly convex), and some of the literature where they are employed.}
    % 	\label{tab:mainCases}
    % 	\centering
    % 	\begin{tabular}{c|c | c}
    % 		& $f$ {\footnotesize  smooth} &$f$ {\footnotesize s.c.}    &$g$ {\footnotesize  smooth}   &$g$ {\footnotesize s.c.} & $A$ & extra asm.   References (extra assumptions) \\
    % 		\hline \hline
    % 		\Cone & &\cmark & &\cmark & & \begin{tabular}{c}  		\citep{chambolleFirstOrderPrimalDualAlgorithm2011,chambolle2016ergodic}, \\
    %         \citep{Palappian2016NipS},\\ \citep{Mokhtari2020};
    %         \end{tabular}
    % 		\\
    % 		($f,g$ smooth): \cite{Arrow1958} \cite{Wang2020Nips}, \cite{kovalevAcceleratedPrimalDualGradient2022}, \cite{Korpelevich1976,Thekumparampil2022}
    % 		\end{tabular}
    % 		 \\
    % 		\hline
    % 		\Ctwo &  &  &\cmark &\cmark & {\footnotesize full col. rank} 
    % 		&  \begin{tabular}{c}
    % 			\citep{wangExploitingStrongConvexity2017,Abdurakhmon2022Neurips}
    % 			\\
    % 			($f$ smooth): \citep{duLinearConvergencePrimalDual2019,kovalevAcceleratedPrimalDualGradient2022,Zhang2022MLR}
    % 			\\
    % 			($f$ affine or zero): \citep{Qu2019,Alghunaim2020}
    % 		\end{tabular}
    % 		\\
    % 		\hline
    % 		\Cthree &\cmark &  &\cmark &   & {\footnotesize invertible} 
    % 		&  \begin{tabular}{c}
    % 			\citep{kovalevAcceleratedPrimalDualGradient2022};
    % 			\\ 
    % 			($f,g$ affine or zero):  \citep{daskalakis2018training,Mokhtari2020,Liang2019} 
    % 		\end{tabular}
    % 		\\
    % 		\hline
    % 	\end{tabular}
    % \end{table*}

 \begin{table*}\small 
    	\caption{The alternative assumptions for linear convergence \Cone{}, \Ctwo{}, \Cthree{}, in the literature. }
    	\label{tab:mainCases}
    	\centering \vspace{1em}
    	\begin{tabular}{ c | c | c | c}
    		\textbf{Asm.} & Extra asm. & References & Convergence type \\ 
    		\hline
            \multirow{6}{*}{\Cone{}}  & & 
            \begin{tabular}{c} \cite{chambolleFirstOrderPrimalDualAlgorithm2011,chambolle2016ergodic};  \cite{Mokhtari2020};
            \\
            \cite{JiangCaiHan_2023}
            \end{tabular} & R-linear
            \\ 
            \cline{3-4}
            & &  \begin{tabular}{c} \cite{Bredies_SIAM_2022,Balamurugan2016};
            \\ 
            \cite{OConnor2020}
            \end{tabular} & contractivity
            \\ \cline{2-4}
             & \multirow{2}{*}{ $f,g$ smooth} & 
            \begin{tabular}{c} \cite{Arrow1958};\cite{Wang2020Nips}; \cite{kovalevAcceleratedPrimalDualGradient2022}; \\ \cite{Korpelevich1976,Thekumparampil2022}
            \end{tabular} 
            & Q/R-linear
            \\ 
            \cline{3-4}
            & & \cite{chenConvergenceRatesForward1997,bauschkeConvexAnalysisMonotone2017} & contractivity
            \\ \Cline{2-4}{0.5pt}
            & \cellcolor{blue!25} & \cellcolor{blue!25} this paper &  \cellcolor{blue!25} contractivity
            \\ \hline 
            \multirow{6}{*}{\Ctwo{}}  & & 
            \begin{tabular}{c} \cite{wangExploitingStrongConvexity2017,Sadiev2022Neurips} \end{tabular} & Q/R-linear 
            \\ \cline{2-4}
             &  $f $ smooth & 
            \begin{tabular}{c} \cite{duLinearConvergencePrimalDual2019,kovalevAcceleratedPrimalDualGradient2022,Zhang2022MLR};
            \\ \cite{Qureshi2023}
            \end{tabular} 
            & R-linear 
            \\ \cline{2-4}
            &  $f=0$ ({\tiny optimization}) & 
            \begin{tabular}{c} \cite{Qu2019,Alghunaim2020};\\ \cite{Salim2022}     
            \end{tabular}
            & Q-linear 
            \\ 
             \cline{2-4}
            & 
            $f=0$ ({\tiny optimization}) & \cite{Velarde_TAC_2022,su2019contractionanalysisprimaldualgradient} & contractivity
            \\ \Cline{2-4}{0.5pt}
            
            & \cellcolor{blue!25} & \cellcolor{blue!25} this paper &  \cellcolor{blue!25} contractivity
            \\ \hline 
            \multirow{4}{*}{\Cthree{}}  & & 
            \begin{tabular}{c} \cite{kovalevAcceleratedPrimalDualGradient2022}\\ \end{tabular}   & R-linear 
            \\ \cline{2-4}
             &  $f,g$ affine or zero & 
            \begin{tabular}{c} \cite{daskalakis2018training,Mokhtari2020};  \\ \cite{Liang2019,Gidel2019}
            \end{tabular} 
            & Q/R-linear 
            \\ \Cline{2-4}{0.5pt}
            & \cellcolor{blue!25} & \cellcolor{blue!25} this paper &  \cellcolor{blue!25} contractivity
            \\ \hline 
    	\end{tabular}
    \end{table*}

    \subsection{Related work}
    A large body of research focuses on showing linear convergence of first-order primal-dual algorithms for solving \eqref{eq:bilinearSaddlePointProblem}, under various assumptions \citep{Zamani_2024,JiangUnifies2022,Zhang2022MLR,Jiang2023}, including \Cone{}, \Ctwo{} and \Cthree{}, see \Cref{tab:mainCases}. For instance, Condition \Cone{} is considered in their seminal paper by  \cite{chambolleFirstOrderPrimalDualAlgorithm2011} and in a number of extensions \citep{Lorenz2023}. The Chambolle--Pock algorithm is also studied under \Ctwo{} by \cite{wangExploitingStrongConvexity2017}. Most literature focuses instead on methods that do not rely on proximal mappings, where both $f$ and $g$ are required to be smooth \citep{Korpelevich1976,duLinearConvergencePrimalDual2019}: \cite{duLinearConvergencePrimalDual2019} first prove linear convergence of the primal-dual gradient method under \Ctwo{} and smoothness of $f$;  \Cthree{} is first exploited for the accelerated method of  \cite{kovalevAcceleratedPrimalDualGradient2022}. 
    In essence, most existing analysis methods are customized to particular problems. For example, a typical strategy to prove linear convergence is to find a merit function that decreases linearly along the algorithm iterates \citep{chambolleFirstOrderPrimalDualAlgorithm2011,wangExploitingStrongConvexity2017, kovalevAcceleratedPrimalDualGradient2022}; yet,  finding such a function is not easy, as the choice heavily depends on problem assumptions and algorithm. This approach makes a systematic treatment difficult and provides little insight for the design of new methods. Further, few works focus on contractivity properties, and those that do are limited to Condition C1 or to the special case of primal-dual methods for constrained optimization. For the general case, we are not aware of any contractivity result for problem~\eqref{eq:bilinearSaddlePointProblem} under \Ctwo{} or \Cthree{}.

\subsection{Contributions} 
 % We take a fresh look
  In this paper, we take a fresh look at linear convergence of first-order primal-dual algorithms through the lens of operator theory. We start by casting several algorithms as preconditioned forward-backward methods. This unified perspective provides a structured way to prove contractivity, and hence linear convergence to the solution of problem \eqref{eq:bilinearSaddlePointProblem}.  Our contributions are the following:
    \begin{itemize}[leftmargin=*]
        \item By leveraging tools from monotone operator theory, we propose a systematic and interpretable approach to derive and establish contractivity of first-order  primal-dual algorithms, both proximal and gradient-based, and under any of the conditions \Cone{}, \Ctwo{} or \Cthree{}. 
        
        \item  In our analysis, we put forward the notion of inverse Lipschitz operator, which to the best of our knowledge has never been considered in saddle-point problems. This allows us, among others,  to prove the first linear rate for the Chambolle--Pock algorithm under \Cthree{}.
        
        \item We improve on existing convergence results, by showing both stronger notions of linear convergence (i.e., contractivity rather than Q- or R-linear) and sharper rates. For instance, we improve on the known rate for the (vanilla) Chambolle--Pock algorithm under  \Cone{}, and we provide the first contractivity results under \Ctwo{} or \Cthree{}.
    \end{itemize}
    % \begin{itemize}
    %     \item  We prove linear convergence under various disjoint assumptions on $f$, $g$ and $A$, that guarantee the existence of a primal-dual solution of \eqref{eq:bilinearSaddlePointProblem}.
    %     \item For all cases, we use the same procedure for deriving the algorithms, i.e., forward-backward splitting. Besides, the splitting is chosen such that the assumptions are already taken into account, giving rise to algorithms which are tailored to the different problem formulations.
    %     \item We demonstrate how this systematic approach can be used to develop novel algorithms.
    %     \item We present a structured and concise way to derive linear convergence proofs for \eqref{eq:bilinearSaddlePointProblem}. This allows us to recover multiple existing results with short and easy-to-follow arguments.
    %      \end{itemize}

    The remainder of this paper is organized as follows. In \Cref{sec:preliminaries} we introduce some necessary notations and   operator theory basics. In \Cref{sec:mainIdea} we  present our operator-theoretic approach. \Cref{sec:linearConvergence} contains our main  contractivity results for primal-dual algorithms under  \Cone{}, \Ctwo, and \Cthree{}. We close the paper in \Cref{sec:conclusion} with some extensions and remarks for future work.

\subsection{Preliminaries}

    \label{sec:preliminaries}
    % Let $\norm{\cdot}$ and $\innerproduct{\cdot}{\cdot}$ denote the standard Euclidean norm and inner product.
     \paragraph{Notation.} For a symmetric and positive semidefinite matrix $\Phi \in \R^{q \times q} \succcurlyeq 0$,  $\innerproduct{\omega}{u}_{\Phi} = u^\top \Phi \omega$   and   $\norm{u}_{\Phi}^2 = \langle u, u \rangle_\Phi$; we omit the subscripts if $\Phi = I$.  The minimum and maximum eigenvalues of a symmetric matrix $C$ are denoted by $\lambdamin{C}$ and $\lambdamax{C}$.   
    % For a set  $S\subseteq \R^q$, $\iota_{S}:\R^q \rightarrow \{ 0, \infty \}$ is its indicator function, i.e., $\iota_{S}(x) = 0$ if $x \in S$, $\infty$ otherwise. 
    We use the convention $\frac{1}{0} = +\infty$.
    A scalar non-negative sequence $(v^k)_{k\in\mathbb N}$ converges  Q-linearly  (to zero)  if   $ v^{k+1} \leq \rho v^k $ for some   $0< \rho < 1$  and all  $k \in \mathbb N$;  R-linearly if   $ v^{k} \leq M  \rho^ k $ for some $0< \rho < 1$,  $M>0$, and for all $k \in \mathbb N$.  
%     We will extensively use the equivalence of norms, i.e., 
%     \begin{equation}
%        \label{eq:equivalenceOfNorms}
%        \frac{1}{\sqrt{\lambda_{\max}(\Phi)}} \norm{x}_{\Phi} \leq \norm{x} \leq \frac{1}{\sqrt{\lambda_{\min}(\Phi)}} \norm{x}_{\Phi}
%    \end{equation}

  \paragraph{Contractivity.} We recall the concept of contractivity, which is central to our work. 

    \begin{definition}[Contractivity]
    Let $\Phi \succ 0$. The operator $\mathcal{A}:\R^q\rightarrow \R^q$ is contractive in $\|\cdot\|_\Phi $ with parameter (or rate) $0\leq \rho <1$ if, for all $\omega, \omega' \in \R^q$,
    \begin{align}
        \| \mathcal{A}(\omega)- \mathcal{A}(\omega') \|_\Phi \leq \rho \|\omega - \omega'\|_\Phi,
    \end{align}
    namely, if $\mathcal{A}$ is $\rho$-Lipschitz in $\|\cdot\|_\Phi$. We also say that the iteration  $\omega^{k+1} = \mathcal{A}(\omega^k)$ is contractive if $\mathcal{A}$ is. 
    \end{definition}
    
    A contractive iteration has a unique fixed point $\omega^\star = \mathcal{A}(\omega^\star)$ \cite[Th.~6]{Banach1922}. Furthermore, it is easily seen that contractivity implies Q-linear (hence, R-linear) convergence with rate $\rho$ of the sequence $(\|\omega^k - \omega^\star\|_\Phi)_{k\in\mathbb{N}}$, but not vice versa.   Compared to Q-linear convergent iterations, contractive algorithms enjoy superior numerical stability and robustness (e.g., in case of perturbed updates),  tracking  properties (e.g., for problems with streaming data), and modularity (the composition of contractive operators is contractive). 
    We illustrate in details the  critical distinction between contractivity and Q-linear convergence in \Cref{sec:Propertiesofcontractivity}.

    \paragraph{Operator theory.} \citep{bauschkeConvexAnalysisMonotone2017}
    %For a closed set $S \subseteq \R^q$, the mapping $\proj_{S}:\R^q \rightarrow S$ denotes the projection onto $S$, i.e., $\proj_{S}(x) = \argmin_{y \in S} \left\| y - x\right\|$.
    A (set-valued)  operator $\mathcal{A}:\R^q\rightrightarrows \R^q$ is characterized by its graph
    $\gra (\mathcal{A})\coloneqq \{(\omega,u) \mid u\in \mathcal{A}(\omega)\}$. The inverse operator ${\mathcal{A}}^{-1}:\R^q\rightrightarrows \R^q$  is always well-defined  via $\gra ({\mathcal{A}}^{-1})=\{(u,\omega)\mid (\omega,u)\in \gra({\mathcal{A}})\}$.  
     % $\zer\left( \mathcal{F}\right) \coloneqq  \left\{ x \in \R^q \mid \0 \in {F}(x) \right\}$ is the  set of zeros of $F$. 
    %%
    The operator ${\mathcal{A}}$ is $\mu$-strongly monotone  if there is $\mu >0$ such that, for all $(\omega,u),(\omega',u') \in\gra(\mathcal{A})$,
    \begin{align}
        \langle u-u' , \omega-\omega'\rangle \geq  \mu\|\omega-\omega'\|^{2},
    \end{align}
    and monotone if this holds with $\mu =0$.
    % \langle u-u' \mid \omega-\omega'\rangle \geq 0$ ($\geq \mu\|\omega-\omega'\|^{2}$) for all $(\omega,u)$,$(\omega',u')\in\gra(\mathcal{A})$.
     Further, $\mathcal{A}$ is maximally monotone if it is monotone and there is no other monotone operator $\bar{{\mathcal{A}}}$ such that  $\gra(\mathcal{A})\subset \gra(\bar {\mathcal A})$.
    The identity operator is $\Id:\omega\mapsto \omega$.
    For an extended-value function $\psi: \R^q \rightarrow \overline{\R} \coloneqq \R \cup \{\infty\}$,  
    $\psi^*(y) = \sup \{ \langle y,x \rangle -\psi(x): x \in \R^q \}$ is its Fenchel conjugate;  and
    $\partial \psi: \dom(\psi) \rightrightarrows \R^q$ is its convex subdifferential,
    \begin{align}
        \partial \psi(x) = \{ v \in \R^q \mid \psi(z) \geq \psi(x) + \langle v , z-x \rangle , \forall  z  \}. \hspace{-0.5em}
    \end{align}
    Let $\psi$ be proper, closed, convex. Then $\partial \psi$ is maximally monotone. Furthermore, $\psi$ is $L$-smooth if and only if $\partial \psi$  is  $L$-Lipschitz continuous (in which case,  $\partial \psi$ is single-valued and $\partial \psi = \nabla \psi$) and $\psi$ is $\mu$-strongly-convex if and only if $\partial \psi$ is $\mu$-strongly monotone.
    % (if $\psi$ is differentiable, then \partial \psi = \nabla \psi$).
    % For an extended-value function $\psi: \R^q \rightarrow \R \cup \{\infty\}$, $\dom(\psi) \coloneqq  \{x \in \R^q \mid \psi(x) < \infty\}$, and  its subdifferential operator is
    % $\partial \psi: \dom(\psi) \rightrightarrows \R^q:x\mapsto  \{ v \in \R^q \mid \psi(z) \geq \psi(x) + \langle v \mid z-x \rangle , \forall  z \in {\rm dom}(\psi) \}$; if $\psi$ is differentiable and convex, $\partial \psi=\nabla \psi$. Let $\psi$ be proper, closed, convex; then, $\partial \psi$ is maximally monotone; 
    % $\psi$ is $\mu$-strongly-convex if and only if $\partial \psi$ is $\mu$-strongly monotone; 
    % $\psi$ is $L$-smooth if and only if $\partial \psi = \nabla \psi$ is $L$-Lipschitz continuous (hence,  $\partial \psi = \nabla \psi $ is single-valued). 
     % The identity operator is $\Id:x\mapsto x$.
    % ${\rm J}_{{F} }\coloneqq (\Id + {F} )^{-1}$ is the resolvent operator of ${F} $. Particularly, if $F  = \partial \psi$ for a proper, closed, convex function $\psi$, then $J_F = \operatorname{prox}_\psi: x \mapsto \operatorname{argmin}_{y \in \R^n}  \{ \psi(y) + \frac{1}{2} \norm{x - y}^2 \}$ is the proximal operator of $\psi$. 
    % The resolvent operator of ${\mathcal{A}}$ is ${\rm J}_{{\mathcal{A}} }\coloneqq (\Id + {\mathcal{A}} )^{-1}$. 
     The proximal operator of $\psi$ is  $\operatorname{prox}_\psi \coloneqq   (\Id+\partial \psi)^{-1} $, and  
    \begin{align}\label{eq:prox}
    \operatorname{prox}_\psi(\omega') = \operatorname{argmin}_{\omega \in \R^q}  \{ \psi(\omega) +  \textstyle \frac{1}{2} \norm{\omega' - \omega}^2 \}.
    \end{align} 
     
     \begin{definition}[Inverse Lipschitz]\label{def:InverseLipschitz}
        The operator $\mathcal{A} : \R^q  \rightrightarrows \R^q$ is   $\frac{1}{L}$-inverse Lipschitz  if  $\norm{u-u'} \geq \frac{1}{L}\norm{\omega-\omega'}$  for all $(\omega,u), (\omega',u') \in \gra(\mathcal{A})$, namely if $\mathcal{A}^{-1}$ is $L$-Lipschitz.
     \end{definition}
     
      The inverse Lipschitz property was recently used by \cite{Ryu2022SRG} and \cite{Gadjov2023Heavy};  earlier, a local version was introduced by \cite{Rockafellar1976}, which is related to many calmness conditions in the literature (e.g., see \cite{JiangUnifies2022}). An inverse Lipschitz operator has at most one zero, and a strongly monotone operator is inverse Lipschitz (as shown in \Cref{sec:PropertiesOfInverseLipschitz}).
      % If $F = \nabla \psi$ is the gradient of a convex  function $\psi$ with nonempty set of (unconstrained) minima $X^\star$, then     $\frac
      % {1}{L}$-inverse Lipschitz of $F$ implies   $\| \nabla \psi(x) \| = \| \nabla \psi(x)-\nabla \psi(\operatorname{proj}_{X^\star} (x)) \| \geq \frac{1}{L} \| x-\operatorname{proj}_{X^\star} (x)\|$,  which is  the standard Error Bound condition \citep{Necoara2019}. 

      % \paragraph{Basic assumption.} Throughout the paper, we postulate the following standard assumption \cite{Condat_2013}, ensuring convexity and strong duality for  \eqref{eq:bilinearSaddlePointProblem}.
      % \begin{assumption}\label{asm:basic}
      %     The functions $f: \R^n \rightarrow \overline \R$ and $g: \R^m \rightarrow \overline \R$ are proper, closed and convex. It holds that $\operatorname{ri}(A $
      % \end{assumption}

\section{UNIFIED ANALYSIS STRATEGY}
    \label{sec:mainIdea}
    
    \paragraph{Algorithm derivation.}
    By the first order optimality conditions, problem \eqref{eq:bilinearSaddlePointProblem} is equivalent to finding a zero of the  \emph{saddle-point operator}  $F:\R^{n+m} \rightrightarrows \R^{n+m}$,
       \begin{equation}
   	F(\omega) = F(x,y) \coloneqq \begin{bmatrix}
   		\partial f(x) + A^\top y\\
   		\partial g(y) - Ax
   	\end{bmatrix}
   	\label{eq:saddlePointOperator},
   \end{equation}
    i.e.,  finding a vector $\omega^\star =(x^\star,y^\star)$ such that $\0 \in F(\omega^\star)$. In turn,  one way to approach this problem is to employ a preconditioned  \emph{forward-backward splitting} \citep[Th.~26.14]{bauschkeConvexAnalysisMonotone2017}, namely the iteration
    \begin{equation}
        \label{eq:precFB}
 	\omega^{k+1} =   (\Id + \Phi^{-1}F_\textnormal{b})^{-1} \circ (\Id-\Phi^{-1} F_\textnormal{f}) (\omega^k)  
% 	  \J_{\Phi^{-1}  F_\textnormal{b}}  (\omega^k - \Phi^{-1} F_\textnormal{f} (\omega^k))
     \end{equation}  where $F$ is split as $F = F_\textnormal{f}+F_\textnormal{b}$,  $F_\textnormal{f}$ must be single-valued, $\Phi \in \R^{(n+m) \times (n+m)}$ is a  positive definite symmetric \emph{preconditioning} matrix to be designed, and 
     \begin{align}\label{eq:FBsteps}
         \mathcal{F} \coloneqq(\Id-\Phi^{-1} F_\textnormal{f}), \quad  \mathcal{B} \coloneqq(\Id + \Phi^{-1} F_\textnormal{b})^{-1}, 
     \end{align} are the forward step and implicit\footnote{Computing $v =  (\Id - \Phi^{-1} \Fb)^{-1} (u)$ means solving for $v$ the regularized inclusion $\0 \in  \Phi (v-u) + F_\textnormal{b}(v)$. Note that $\mathcal{B}$ is single-valued in all our results, see \Cref{sec:algorithmsderivation}.}  backward step, respectively. The fixed points of the iteration \eqref{eq:precFB} are the vectors $\omega^\star$ that solve $\0 \in \Phi^{-1} F(\omega^\star)$,  equivalently $\0 \in F(\omega^\star)$, equivalently the solutions of \eqref{eq:bilinearSaddlePointProblem}. For instance, if $f$ and $g$ are differentiable and we choose $F_\textnormal{f}  = F$ (hence, $F_\textnormal{b} = \0$ is the zero operator) and $\Phi = \operatorname{diag}\left(\frac{1}{\tau} I_n, \frac{1}{\sigma} I_m\right)$ for some $\tau,\sigma >0$, we obtain the iteration 
    \begin{subequations}
        \label{eq:Gradientdescentascent} 
        \begin{align} 
            x^{k+1} & = x^k - \tau (\nabla f (x^k) +A^\top y^k)\\
         y^{k+1} & = y^k - \sigma (\nabla g (y^k) - A x^k),
        \end{align}
    \end{subequations}
     namely the \gls{PDG} method. This algorithm is not guaranteed to converge without strong convexity assumptions (for instance, it does not converge if $f=g=0$ and $A=1$, which satisfies \Cthree{}). To derive iterations that converge without extra assumptions and converge linearly if \Cone{}, \Ctwo{}, or \Cthree{} hold, we instead place the skew symmetric linear operator $(A^\top y,-Ax)$ in the backward step. To obtain inexpensive updates and facilitate the computation of  $\mathcal{B}$, we choose the preconditioning matrix as  
      \begin{equation}
            \label{eq:matrixPhiEta} 
     	\Phi_{\tau, \sigma} = \begin{bmatrix}
     		\frac{1}{\tau}I_n &-A^\top \\
     		-A & \frac{1}{\sigma}I_m
     	\end{bmatrix},
     \end{equation}
    where we highlight the dependence on the step sizes $\tau$ and $\sigma$ via the subscript. Note that  $\Phi_{\tau, \sigma} \succ 0$ if $\tau \sigma \|A\|^2 <1$. \newline
    For different splittings, i.e.,  choices of $F_\textnormal{b}$ and $F_\textnormal{f}$, we obtain three different algorithms, whose complete derivation we defer to \Cref{sec:algorithmsderivation}. 
    % With $F_\textnormal{b} =F$ and $F_\textnormal{f}=\0$, we recover the Chambolle-Pock method 
    The first algorithm is the Chambolle--Pock method 
       \begin{subequations}
        \label{eq:ChambollePock1} 
        \begin{align} 
            x^{k+1} & = \operatorname{prox}_{\tau f} \left(x^k - \tau A^\top y^k \right)
            \\
         y^{k+1} & = \operatorname{prox}_{\sigma g} \left( y^k +\sigma A (2x^{k+1}-x^k) \right),
        \end{align}
    \end{subequations}
    which requires that  the proximal operators of both $f$ and $g$ are available in closed form. The second is 
    % If $g$ is smooth, with $F_\textnormal{b} = (\partial f(x)+A^\top y, -Ax)$ and $F_\textnormal{f} = (\0,\nabla g(y))$,
     the semi-implicit method
       \begin{subequations}
        \label{eq:semiimplicit1} 
        \begin{align} 
            x^{k+1} & = \operatorname{prox}_{\tau f} \left(x^k - \tau A^\top y^k \right)
            \\
         y^{k+1} & =   y^k -\sigma \left (\nabla g(y^k) - A (2x^{k+1}-x^k) \right),
        \end{align}
    \end{subequations}
   which requires $g$ to be smooth, but only  $f$ to be prox-friendly. The third is the fully explicit method
   % Finally, if both $f$ and $g$ are smooth, with $F_\textnormal{b}=(A^\top,-Ax)$ and $F_\textnormal{f}=(\nabla f(x), \nabla g(y))$, we get the fully-explicit  method
       \begin{subequations}
        \label{eq:explicit1} 
        \begin{align} 
            x^{k+1} & = x^k -\tau  \left( \nabla f(x^k) + \tau A^\top y^k \right)
            \\
         y^{k+1} & =   y^k -\sigma \left (\nabla g(y^k) - A (2x^{k+1}-x^k) \right).
        \end{align}
    \end{subequations}
where both $f$ and $g$ are only required to be smooth.  
% The algorithms in \eqref{eq:ChambollePock1}, \eqref{eq:semiimplicit1}, \eqref{eq:explicit1}, can all be retrieved as special cases of 

    \paragraph{Main idea.} The derivation technique just sketched is itself not novel \citep{Condat_2013,Vu2013,Combettes2014_ICIP,Condat2023}. For instance, it has been known since \cite{He2012SIAM} that the Chambolle--Pock algorithm can be interpreted as a preconditioned forward-backward method. 
   Instead, our contribution is a novel approach to  exploit the operator-theoretic formulation \eqref{eq:precFB}  for linear convergence arguments.  In particular, our strategy is 
    to prove contractivity properties for the forward and backward steps. Concerning the latter, we will use the notion of inverse Lipschitz operators to deal with the lack of strong convexity, as exemplified in the following proposition, proven in Appendix~\ref{sec:proofPropertiesOfResolvent}.
    \begin{proposition}
	\label{prop:PropertiesOfResolvent}
    Let $F:\R^n \rightrightarrows \R^n$ be maximally monotone and $\frac{1}{L}$ inverse Lipschitz. Then the resolvent operator $\J_F \coloneqq (\Id+F)^{-1}$ is  $\frac{L}{\sqrt{L^2+1}}$-contractive.
\end{proposition}

%     \begin{proposition}\label{prop:dummy}
%         Assume that
%         \begin{itemize}
%             \item m
%         \end{itemize}
%     \end{proposition}

%   

In fact, one of the major technical novelties of this paper is to show that \Ctwo{} and \Cthree{}  induce certain useful inverse Lipschitz properties (see \Cref{lemma:IL1,lemma:IL2} below). 
    Nonetheless, \Cref{prop:PropertiesOfResolvent} is still insufficient for the problem at hand. First, the preconditioning introduces technical complications. Second, often in our setup inverse Lipschitz and strong convexity properties do not hold on the whole space $(x,y)$, but only with respect to a subset of the variables; hence, forward and backward steps are not independently contractive.
     To address both challenges, our enabling result is the following key proposition, which leverages the interplay between the forward and backward steps and  exploits ``partial contractivity'' properties in weighted norms. The proof is found in Appendix~\ref{app:propPartialContractivity}.

    \begin{proposition}\label{prop:partialContractivity}
        Let $\Phi_{\tau,\sigma} \succ 0$ as in \eqref{eq:matrixPhiEta} and recall the forward step $\mathcal{F}$ and backward step $\mathcal{B}$ in \eqref{eq:FBsteps}. Assume that there exist a scalar $0 < \gamma \leq  1$ and matrices $\Psi_{\textnormal b} \succcurlyeq 0 $, $\Psi _{\textnormal f}\succcurlyeq 0$, such that, for all $\omega, \omega' \in \R^{n+m}$,
        \begin{itemize}
            \item[\Aone{}.] $\|\mathcal{B}(\omega) - \mathcal{B}(\omega')\|^2_{\Phi_{\tau,\sigma}+\Psi_{\textnormal{b}}} 
            \leq  \| \omega -\omega'\|^2_{\Phi_{\tau,\sigma}}$
            \item[\Atwo{}.] $\|\mathcal{F}(\omega) - \mathcal{F}(\omega')\|^2_{\Phi_{\tau,\sigma}} 
            \leq \| \omega -\omega'\|^2_{\Phi_{\tau,\sigma}-\Psi_{\textnormal{f}}}$
        \end{itemize}
        and $ \Psi_{\textnormal b}+ \Psi _{\textnormal f} \succcurlyeq \gamma (\Phi_{{\tau,\sigma}}+\Psi_{\textnormal b}) $. 
        Then, the forward-backward iteration in \eqref{eq:precFB} is contractive in $\| \cdot\|_{\Phi_{{\tau,\sigma}}+\Psi_{\textnormal{b}}} $, with rate 
        $\rho =  \sqrt{1-\gamma}$. 
    \end{proposition}
    
    By directly exploiting the assumptions on the operators $F_\textnormal{f}$ and $F_\textnormal{b}$ (namely, conditions \Cone{}, \Ctwo{}, \Cthree{}), we can guarantee the partial contractivity conditions in \Aone{} and \Atwo{} for each of the methods \eqref{eq:ChambollePock1}, \eqref{eq:semiimplicit1}, \eqref{eq:explicit1}. In turn, Proposition~\ref{prop:partialContractivity} allows us to conclude the contractivity of our algorithms, as detailed in \Cref{sec:linearConvergence}, and  informally summarized next, for readability.   

    \begin{theorem}
    Under either \Cone{}, \Ctwo{}, or \Cthree{}, there exist stepsizes $\tau,\sigma>0$ such that:
    \begin{itemize}[leftmargin=*]
        \item algorithm~\eqref{eq:ChambollePock1} is contractive;
        \item if $g$ is smooth,  algorithm~\eqref{eq:semiimplicit1} is contractive;
        \item if $f$ and  $g$ are smooth, algorithm~\eqref{eq:explicit1} is contractive. 
    \end{itemize}
    \end{theorem}

    The proofs of our results can be found in \Cref{sec:omittedproofs}.

\section{MAIN RESULTS}
   \label{sec:linearConvergence}

    \subsection{Chambolle--Pock algorithm}
        \label{sec:chambollePock}
        The first method we consider is the Chambolle--Pock algorithm \citep{chambolleFirstOrderPrimalDualAlgorithm2011}  in \eqref{eq:ChambollePock1}, which is widely used in inverse imaging tasks. 
        The method is recalled in \Cref{alg:chambollePock} below, where we highlight the hyperparameter $\theta$ for the sake of later comparison with related work, though we always set  $\theta=1$ here. 
        % It can be cast as the forward-backward iteration in \eqref{eq:precFB} by choosing $F_\textnormal{b} = F$, $F_\textnormal{f} = \0$ and  the preconditioning matrix $\Phi_{\tau, \sigma} $ defined in \eqref{eq:matrixPhiEta}. 
  
        \begin{algorithm}
        \small
            \caption{Chambolle--Pock method}
            \label{alg:chambollePock}
            \begin{algorithmic}[1]
                \Require  $x^0\in \R^n, y^0\in \R^m$, step sizes $\tau, \sigma >0$, $\theta = 1$
                \For{$k = 0, 1, 2,\dots$}
                    \State $x^{k+1} = \operatorname{prox}_{\tau f} \left (x^k - \tau A^\top y^k \right)$
                    \State $y^{k+1} = \operatorname{prox}_{\sigma g} \left (y^k + \sigma A \left ((1+\theta)x^{k+1}-\theta x^k \right) \right )$
                \EndFor
            \end{algorithmic}
        \end{algorithm}

     \Cref{alg:chambollePock} can be cast as the forward-backward iteration in \eqref{eq:precFB} by choosing $F_\textnormal{b} = F$, $F_\textnormal{f} = \0$ and $\Phi_{\tau, \sigma} $ as in \eqref{eq:matrixPhiEta}, see \Cref{sec:algderivation1}. 
    It is known that contractivity of  \Cref{alg:chambollePock} can be inferred by strong monotonicity of  $F$, when \Cone{} holds \citep{Bredies_SIAM_2022}. Here we provide a refined analysis, that results in improved rates. Further, for analysis under \Ctwo{} and \Cthree{}, we provide the following inverse Lipschitz properties, whose derivation is completely novel. The expressions for the constants $R_2$ and $R_3$ can be found in the appendix. 
    \begin{lemma} 
        \label{lemma:IL1}
        If \Ctwo{} holds, the operator $F$ in \eqref{eq:saddlePointOperator} is $\frac{1}{R_2}$-inverse Lipschitz, $R_2 >0$.
        \end{lemma}
    \begin{lemma}\label{lemma:IL2}
        If \Cthree{} holds, the operator $F$ in \eqref{eq:saddlePointOperator} is $\frac{1}{R_3}$-inverse Lipschitz, $R_3>0$.
    \end{lemma}

      Based on \Cref{lemma:IL1,lemma:IL2} we can guarantee conditions \Aone{} and \Atwo{} in \Cref{prop:partialContractivity}, with $\Psi_{\textnormal{f}} =  \0$,  $\gamma >0 $. From now on, let us
    use the shorthand 
    \begin{align}\label{eq:zet}
    	\zet  \coloneqq \max \left \{ \textstyle \frac {1}{\tau},\frac{1}{\sigma } \right\}+\|A\|.
    \end{align}
    
    \begin{theorem}[Contractivity of \Cref{alg:chambollePock}]
        \label{thm:chambollePock}
         Let $f: \R^n \rightarrow \overline \R$ and $g: \R^m \rightarrow \overline \R$ be proper, convex, closed functions. Let $\phi_{\tau, \sigma}$ be as in \eqref{eq:matrixPhiEta}. Then, the fixed points of \Cref{alg:chambollePock} coincide with the solutions of  \eqref{eq:bilinearSaddlePointProblem}. Furthermore, if $\tau,\sigma>0$ are chosen such that $\tau\sigma \|A\|^2 (1+\epsilon)^2 \leq  1 $, for some $\epsilon >0$,  then:
         \begin{itemize}
        \item[(i)] If \Cone{} holds, then \Cref{alg:chambollePock} is contractive in $\| \cdot \|_{\Phi_{\tau,\sigma}}$ with rate
                % \begin{align}
                %     \label{eq:linearPhiChambolle}
                %     \| \omega^{k} - \omega^\star\| _{\Phi_{\tau,\sigma}} \leq  \rho^k  \| \omega^0 - \omega^\star\|_{\Phi_{\tau,\sigma}},
                % \end{align}
                % where 
%             		$\rho = \left ( \frac{\lambdamax{\Phi_{\tau,\sigma}}}{\lambdamax{\Phi_{\tau,\sigma}}+ m}\right )$. 
                \begin{align*} 
                % \label{eq:rhoCP:1}
                    \rho =  \left( 1+ \min \{ \mf \tau, \mg \sigma,  \kappa   \} \right)^{-1} <1,
                \end{align*}
                where 
                    \begin{equation*}
                          \kappa \coloneqq  {	   \textstyle \frac{\mf \tau +\mg \sigma  -\sqrt{ (\mf \tau - \mg\sigma)^2 + 4 \|A\|^2 \mf\mg \tau^2 \sigma^2}}{2(1-\sigma\tau \|A\|^2)} } >0.
                    \end{equation*}
                The best rate is obtained with $\tau = \frac{1}{(1+\epsilon) \|A\| }\sqrt{\frac{\mg}{\mf}}$ and $\sigma = \frac{1}{(1+\epsilon) \|A\| }\sqrt{\frac{\mf}{\mg}}$;
            \item[(ii)] If \Ctwo{} holds, then 
            \Cref{alg:chambollePock} is contractive in $\| \cdot \|_{\Phi_{\tau,\sigma}}$ with rate
                \begin{align*}
                	\rho =\frac{ R_2 \zet}{\sqrt{ (R_2 \zet)^2+1}}<1.               %\rho =  \left( 1 + 1 / (R_2 \zet )^2 \right )^{-\frac{1}{2}}.
%                   \\
%                   \rho = \left( \sqrt{  1 + \textstyle \frac{1}{(R_2 \zet )^2} } \right)^{-1}.
% \rho = \frac{R_2 \zet }{\sqrt{(R_2\zet)^2 +1}}
% \rho = \left( \sqrt{  1 + \textstyle \frac{1}{(R_2\zet )^2} } \right)^{-1}. 
                \end{align*}
                 The best rate is obtained with $\tau =  \sigma = \frac{1}{(1+\epsilon) \|A\| }$; 
                \item [(iii)] If \Cthree{} holds, then
            \Cref{alg:chambollePock} is contractive in $\| \cdot \|_{\Phi_{\tau,\sigma}}$ with rate
                    \begin{align*}
                    		\rho =\frac{ R_3 \zet}{\sqrt{ (R_3 \zet)^2+1}} <1.
%                              \rho = \left( \sqrt{  1 + \textstyle \frac{1}{(R_3 \zet )^2} } \right)^{-1}. 
% \rho = \frac{R_3 \zet }{\sqrt{(R_2\zet)^2 +1}}
                    \end{align*}
             The best rate is obtained with $\tau =  \sigma = \frac{1}{(1+\epsilon) \|A\| }$.
         \end{itemize} 
    \end{theorem}

Let us emphasize that, by definition of contractivity, under \Cone{}, \Ctwo{}, or \Cthree{}, the previous theorem immediately implies Q-linear convergence, i.e.,
    \begin{align}
         \label{eq:linearPhiChambolle}
                    \| \omega^{k} - \omega^\star\| _{\Phi_{\tau,\sigma}} \leq  \rho^k  \| \omega^0 - \omega^\star\|_{\Phi_{\tau,\sigma}},
                \end{align}
for the sequence $(\omega^k)_{k\in\mathbb{N}} = (x^k,y^k)_{k\in\mathbb{N}} $ generated by \Cref{alg:chambollePock}, where $\omega^\star$ is the unique solution to \eqref{eq:saddlePointOperator}. 

        We note that, under \Cone{}, the work \citep{chambolleFirstOrderPrimalDualAlgorithm2011} provides the R-linear convergence rate $\tilde \rho =  \left( 1+ \frac{\sqrt{\mf\mg}}{\|A\|}  \right )^{-\frac{1}{2}}
        $. \Cref{thm:chambollePock} improves on this result, by guaranteeing both $Q$-linear convergence and  a tighter bound. In fact, the optimal step sizes choice in \Cref{thm:chambollePock}(i) 
%        for any $\epsilon >0$, letting $\tau = \frac{1}{(1+\epsilon) \|A\| }\sqrt{\frac{\mg}{\mf}}$ and $\sigma = \frac{1}{(1+\epsilon) \|A\| }\sqrt{\frac{\mf}{\mg}}$,
         results in the rate $\rho = \left( 1+ \frac{\sqrt{\mf \mg }}{ (\epsilon +2)\|A \|  } \right)^{-1} = \left( 1+ \frac{2\sqrt{\mf \mg }}{ (\epsilon +2) \|A \|  }+ \frac{\mf \mg }{ (\epsilon +2)^2 \|A \|^2 } \right)^{-\frac{1}{2}}$ and, for any $\mf$, $\mg$, $\|A\|$, we can always choose   $\epsilon$ small enough to ensure $\rho < \tilde \rho$. 
        Furthermore, when $A = \0$, the step sizes can be chosen as large as desired, and $\rho =(1+ \min \{ \mg \tau, \mg \sigma\})^{-1} $, which is the well-known bound for the proximal-point algorithm (see Appendix~\ref{sec:proofPropertiesOfResolvent}). Let us note that a different rate was provided in \citep{chambolle2016ergodic}, but for a different choice of $\theta <1$. 
        Also for  $\theta<1$, R-linear convergence of \Cref{alg:chambollePock}  was proven by  \cite{wangExploitingStrongConvexity2017} under \Ctwo{}. In contrast, \Cref{thm:chambollePock}(ii) shows contractivity  with $\theta = 1$. The choice $\theta =1 $ is particularly relevant, as in this case the algorithm can be interpreted as a proximal-point method with symmetric preconditioning \citep{He2012SIAM}, allowing one to leverage the extensive related theory; for instance, immediately enabling provably convergent accelerations \citep{Bianchi2022Fast,Iutzeler2019Acceleration}. 
        Finally, \Cref{thm:chambollePock}(iii) is the first linear convergence rate for the Chambolle--Pock algorithm under \Cthree{}, namely if $f$ and $g$ are merely convex, but $A$ is invertible. To the best of our knowledge, only \cite{kovalevAcceleratedPrimalDualGradient2022} studied linear convergence under \Cthree{}, but based on a gradient-based method, rather than proximal-based. 

		\begin{remark}[Degenerate preconditioner]
			\label{rem:degeneratePrec}The rates in \Cref{thm:chambollePock} are decreasing in $\epsilon$, thus the best rates are obtained for $\epsilon \rightarrow 0$. This should be interpreted as an asymptotic result, as for $\epsilon =0$, the matrix $\Phi_{\tau,\sigma}$ might be singular. In terms of non-weighted norm, one can only conclude the estimate $\| \omega^{k} - \omega^\star\|  \leq   \left(C \frac{1}{{\epsilon} } \right)\rho^k \| \omega^0 - \omega^\star\|$, for some $C>0$, which is vacuous if $\epsilon = 0$.  Let us note that the convergence of \Cref{alg:chambollePock} was also shown for degenerate (i.e., singular) preconditioners, with $\tau \sigma \|A\|^2 = 1$ \citep{Condat_2013}, with known linear rates \citep{chambolleFirstOrderPrimalDualAlgorithm2011}, and more recently even for $\tau \sigma \|A\|^2 < 4/3$ \citep{banert2023chambollepock}. Extending our results to these cases is an open problem.
		\end{remark}

    \subsection{Semi-implicit primal-dual method}
       
          \label{sec:semiSmoothPrimalDual}
        Next, we consider the method in \eqref{eq:semiimplicit1}, recalled in \Cref{alg:semiSmoothPrimalDual}, which requires $g$ to be smooth. 
        This method is an  instance of the Condat--V\~u splitting \citep{Condat_2013,Vu2013},
        applicable to linear regression with $\ell_1$ regularization, inverse imaging models \citep{chambolle2016Introduction}, and affine-constrained optimization \citep{Zhu2023Unified}.

        \begin{algorithm}
            \small
            \caption{Semi-implicit primal-dual method}
            \label{alg:semiSmoothPrimalDual}
            \begin{algorithmic}[1]
                \Require  $x^0\in \R^n, y^0\in \R^m$, step sizes $\tau, \sigma >0$, $\theta =1$
                \For{$k = 0, 1, \cdots$}
                \State $x^{k+1} = \operatorname{prox}_{\tau f} \left (x^k - \tau A^Ty^k \right)$
                \State $y^{k+1} = y^k-\sigma  \left (\nabla g(y^k) - A \left ((1+\theta)x^{k+1}-\theta x^k \right) \right)$
                \EndFor
            \end{algorithmic} 
        \end{algorithm}

        \Cref{alg:semiSmoothPrimalDual} corresponds to the iteration in \eqref{eq:precFB}, with $F_\textnormal{b} (\omega) = (\partial f(x) +A^\top y, -Ax)$,  $F_\textnormal{f} (\omega ) = (\0,\nabla g(y))$ and $\Phi_{\tau, \sigma} $ as in \eqref{eq:matrixPhiEta}, see \Cref{sec:algderivation2}.
        Under \Cthree{}, we can show contractivity of this algorithm by exploiting that $F_\textnormal{b}$ is inverse Lipschitz, analogously to \Cref{thm:chambollePock}(iii). 
%         given that the backward step is a contraction in the $\Phi_{\tau,\sigma}$-weighted norm (while the forward step is nonexpansive).
 Proving linear convergence for \Cref{alg:semiSmoothPrimalDual} under \Cone{} and \Ctwo{} is more complex, as neither $\mathcal F$ nor $\mathcal B$ are contractions (note further that also the sum $F= F_\textnormal{f} + F_\textnormal{b}$ is  strongly monotone, cf. Theorem~6.1 in \cite{Giselsson2021}).
 Instead, we exploit the intuition that the forward operator $F_\textnormal{f}$ is strongly monotone, but only with respect to the dual variable $y$; and the backward operator  $F_\textnormal{b}$ is inverse Lipschitz with respect to the primal variable $x$.  In turn, we can prove properties \Aone{}, \Atwo{} in \Cref{prop:partialContractivity}, in the following lemmas. 

        \begin{lemma}
        	\label{lemma:semiContractionBackward}
        	Let either  \Cone{} or \Ctwo{} hold (so either $\mf >0$ or $\mA>0$), and let  $\tau \sigma \norm{A}^2 < 1$, $F_\textnormal{b}(\omega) = (\partial f(x)+ A^\top y, -Ax)$.  
        	Then, \Aone{} in  \Cref{prop:partialContractivity} holds with   $\Psi_\textnormal{b} = \operatorname{diag}(\gamma_x I_n, \0_{m\times m})$, where $\gamma_x = \frac{\mu_A}{\zet}+2\mf >0$.
%        	 and $\Psi_b \coloneqq \operatorname{diag}(\rho_b^2 I_n, \0_{m\times m})$, where  $\rho_b^2 = \frac{1}{\lambdamax(\Phi_{\tau,\sigma})}+\mu_A >0$. Then for all $\omega, \omega' \in \R^{n+m}$, it holds that
%        	\begin{IEEEeqnarray*}{rCl}
%                \IEEEeqnarraymulticol{3}{l}{\norm{\J_{\Phi_{\tau, \sigma}^{-1}F_\textnormal{b}}(\omega)-\J_{\Phi_{\tau, \sigma}^{-1}F_\textnormal{b}}(\omega')}_{ \Phi_{\tau, \sigma}}^2}\\
%                && + \norm{\J_{\Phi_{\tau, \sigma}^{-1}F_\textnormal{b}}(\omega)-\J_{\Phi_{\tau, \sigma}^{-1}F_\textnormal{b}}(\omega')}_{\Psi_b}^2\\
%                & \leq & \norm{\omega - \omega'}_{\Phi_{\tau, \sigma}}^2
%        	\end{IEEEeqnarray*}
        \end{lemma}

    \begin{lemma}
        	\label{lemma:semiContractionForward}
        	Let either \Cone{} or \Ctwo{} hold, and 
        	let $\sigma   < \frac{2}{2\tau \|A\|^2+\Lg}$,  $F_\textnormal{f}(\omega) = (\0,\nabla g(y))$. Let  $\xi_1 =  \frac{1}{\sigma } - \tau \|A\|^2 > 0$, $\xi_2 = \frac{1}{\sigma } - \tau \mu_A>0$. 	Then, \Atwo{} in \Cref{prop:partialContractivity} holds with 
        	 $\Psi_{\textnormal{f}}= \operatorname{diag} (\0_{n\times n},\gamma_y I_m)$, where 
        	\begin{equation*}
        	0< \gamma_y = \begin{cases}
        			\frac{2 \Lg \mg}{\Lg +\mg}+\frac{\mg^2(2\xi_1-\Lg-\mg)}{(\Lg+\mg)\xi_2} & \quad \text {if } \  \xi_1 \geq \frac{\Lg + \mg}{2} 
        			\\
        			\frac{2 \Lg \mg}{\Lg +\mg}-\frac{\Lg^2(\Lg +\mg-2\xi_1)}{(\Lg+\mg)\xi_1} & \quad \text {if } \ \xi_1 <  \frac{\Lg + \mg}{2}.
        		\end{cases}
        	\end{equation*} 
%        	Then for all $\omega, \omega' \in \R^{n+m}$, we have
%        	\begin{IEEEeqnarray*}{rCl}
%        	    \IEEEeqnarraymulticol{3}{l}{\norm{\left(\Id- \Phi_{\tau, \sigma}^{-1}F_\textnormal{f}\right)(\omega)-\left(\Id- \Phi_{\tau, \sigma}^{-1}F_\textnormal{f}\right)(\omega')}_{\Phi_{\tau, \sigma}}^2}\\
%             &\leq& \norm{\omega - \omega'}_{\Phi_{\tau, \sigma}}^2-\norm{\omega - \omega'}_{\Psi_f}^2.
%        	\end{IEEEeqnarray*}
        \end{lemma}

        % \begin{lemma}
        % 	\label{lemma:semiContractionForward}
        % 	Let \Ctwo{} hold. Let $\tau \sigma \norm{A}^2 < 1$ and $\sigma < \frac{2}{\Lg + 2\|A\|}$,  $F_\textnormal{f}(\omega) = (\0,\nabla g(y))$, and $\Psi_{\textnormal{f}}= \operatorname{diag}\left(\0_{n\times n}, \rho_f^2 I_m\right)$, where 
        % 	\begin{equation*}
        % 		\rho_f^2 = \begin{cases}
        % 			\frac{2 \Lg \mg}{\Lg +\mg}+\frac{2\mg^2\xi_1}{\xi_2(\Lg+\mg)}-\frac{\mg^2}{\xi_2} & \qquad \text {if } \  \xi_1 \geq \frac{\Lg + \mg}{2} 
        % 			\\
        % 			\frac{2 \Lg \mg}{\Lg +\mg}+\frac{2L_g^2}{\Lg+\mg}-\frac{\Lg^2}{\xi_1} & \qquad \text {if } \ \xi_1 <  \frac{\Lg + \mg}{2}  
        % 		\end{cases}
        % 	\end{equation*} 
        % 	and $\xi_1 =  \frac{1}{\sigma } - \tau \|A\|^2 > 0$, $\xi_2 = \frac{1}{\sigma } - \tau \mu_A>0$. 	
        % 	Then for all $\omega, \omega' \in \R^{n+m}$, we have
        % 	\begin{IEEEeqnarray*}{rCl}
        % 	    \IEEEeqnarraymulticol{3}{l}{\norm{\left(\Id- \Phi_{\tau, \sigma}^{-1}F_\textnormal{f}\right)(\omega)-\left(\Id- \Phi_{\tau, \sigma}^{-1}F_\textnormal{f}\right)(\omega')}_{\Phi_{\tau, \sigma}}^2}\\
        %      &\leq& \norm{\omega - \omega'}_{\Phi_{\tau, \sigma}}^2-\norm{\omega - \omega'}_{\Psi_f}^2.
        % 	\end{IEEEeqnarray*}
        % \end{lemma}

       When  $A = \0$, \Cref{lemma:semiContractionForward}  simplifies to 
        \begin{IEEEeqnarray*}{rCl}
            \IEEEeqnarraymulticol{3}{l}{\norm{ y - \sigma \nabla g(y) - (y - \sigma \nabla g(y'))}^2}\\
            &\leq & \max \left\{| 1 - \sigma \mg |^2, |1-\sigma \Lg|^2  \right\} \norm{ y -y'}^2
        \end{IEEEeqnarray*}

        for all $\sigma < \frac{2}{\Lg}$, which is  the usual contractivity result for gradient descent (see \Cref{prop:propertiesOfGradientStep} in \Cref{sec:proofPropertiesOfResolvent}). The following is the main result of this subsection.

       \begin{theorem}[Contractivity of \Cref{alg:semiSmoothPrimalDual}]
            \label{thm:semiSmoothAlgorithm}
         Let $f: \R^n \rightarrow \overline \R$ and $g: \R^m \rightarrow \R$ be proper, convex, closed functions, and let $g$ be $\Lg$-smooth. Let $\phi_{\tau, \sigma}$ be as in \eqref{eq:matrixPhiEta}, and let $\epsilon>0$ be arbitrary. Then, the fixed points of \Cref{alg:semiSmoothPrimalDual} coincide with the solutions of  \eqref{eq:bilinearSaddlePointProblem}. Furthermore: 
         \begin{itemize}
         	\item[(i)] If $\tau\sigma \|A\|^2 +\frac{\Lg}{2} \sigma< 1$ and either \Cone{} or \Ctwo{} hold, then \Cref{alg:semiSmoothPrimalDual} is contractive in $\norm{\cdot}_{\Psi_{\tau,\sigma}}$  with rate $\rho$, where  $\Psi_{\tau, \sigma} = \P+\diag (\gamma_x I_n, \0) \succ 0$,        
         		\begin{align*}
         		 %    \Psi_{\tau, \sigma} & = \begin{bmatrix}
         			% \left(\frac{1}{\tau}+\gamma_x \right )I_n &-A^\top \\
         			% -A &\frac{1}{\sigma}I_m 
            %  		\end{bmatrix} \succ 0
            %     \\
                        \rho &= \sqrt{1-\frac{\min\{\gamma_x^2, \gamma_y^2\}}{\zet +\gamma_x}}<1,
         		\end{align*}
         	 and $\gamma_x$ and $\gamma_y$ are as in \Cref{lemma:semiContractionBackward} and \Cref{lemma:semiContractionForward}, respectively; and 
         	\item[(ii)]  If  $\tau \sigma \|A\| ^2 +\frac{\Lg}{2} \sigma \leq 1$, $\tau \sigma \|A\|^2(1+\epsilon)^2 < 1$ and \Cthree{} holds, then \Cref{alg:semiSmoothPrimalDual} is contractive in $\norm{\cdot}_{\Phi_{\tau,\sigma}}$ with rate
          	\begin{align*}
          			\rho =  \frac{R_3'  \zet }{\sqrt{{R_3'}^2  \zet^2 +1 } }<1,
          	\end{align*}
                where $R_3'>0$ is independent of $\sigma$ and $\tau$. The best rate is obtained with $\tau = \sigma = \min \left \{ \frac{\sqrt{\Lg^2+16\|A\|^2}-\Lg}{4\|A\|^2}, \frac{1}{(1+\epsilon) \|A\|} \right\}$.
         \end{itemize}
       \end{theorem}

 Note that $\epsilon$ in \Cref{thm:semiSmoothAlgorithm}(ii) is only used to account for the case $L_g =0$ (in which instance, analogous considerations to \Cref{rem:degeneratePrec} hold), but $\epsilon=0$ can otherwise be chosen.    Deriving an expression for the optimal step sizes in \Cref{thm:semiSmoothAlgorithm}(i) is cumbersome and not very insightful. Instead, let us focus on the novelty of \Cref{thm:semiSmoothAlgorithm}. The step sizes bound  $\tau\sigma \|A\|^2 +\frac{\Lg}{2} \sigma <  1 $ coincides with that of the Condat-V\~u method \cite[Th.~3.1]{Condat_2013}, \cite[Th.~5]{Lorenz2015}. However, the linear rate of this algorithm has not been shown under \Ctwo{} nor \Cthree{}. Under \Cthree{}, linear convergence was only studied for fully explicit methods (that use $\nabla f$ and $\nabla g$) \citep{kovalevAcceleratedPrimalDualGradient2022}, while it might be convenient to employ the proximal operator of $g$, when available in closed form. On the other hand, known methods to achieve linear convergence under \Ctwo{} (without further conditions) require both $f$ and $g$ to be prox-friendly, i.e., their $\operatorname{prox}$ operator can be efficiently computed \citep{wangExploitingStrongConvexity2017,Sadiev2022Neurips}, in contrast to \Cref{alg:semiSmoothPrimalDual}. 
For instance, let us consider again the 
optimization problem in \eqref{eq:minimizationAffineConstraints}. 
Its Lagrangian is $\mathcal{L}(x,y) =  -x^\top b+ \iota_{\R^n_{\geq 0}}(x) +   y^\top A x  - g(y)$;  its saddle-point operator is $F(x,y)= (\mathrm{N}_{\R^n_{\geq 0}}(x) +A^\top y -b,\nabla g(y) - Ax)$, with $\mathrm{N}_{\R^m_{\geq 0}}$ the normal cone of the positive orthant, and \Cref{alg:semiSmoothPrimalDual} reduces to 
\begin{align}
    x^{k+1} &  = \operatorname{proj}_{\R^n_{\geq 0}}\left(x^k -\tau ( A^\top y^k -b )\right)
    \\
    y^{k+1} &= y^k - \sigma \left( \nabla g (y^k) -A (2x^{k+1} - x^k) \right).
\end{align}
\Cref{thm:semiSmoothAlgorithm} ensures the contractivity (hence, Q-linear convergence) of this iteration under \Ctwo{}, i.e.,  if $g$ is strongly convex and $A^\top$ is full-row rank. 
Linear convergence in this setup was only studied for equality constrained problems \citep{Velarde_TAC_2022}, or by resorting to augmented Lagrangian formulations \citep{su2019contractionanalysisprimaldualgradient}, due to the difficulty of dealing with non-smoothness, and of reconciling properties in weighted spaces with the unweighted projection $\operatorname{proj}_{\R^n_{\geq 0}}$ \citep{Qu2019}. \Cref{thm:semiSmoothAlgorithm} circumvents these difficulties.

      \subsection{Preconditioned \gls{PDG} method}\label{sec:smoothPrimalDual}

    Finally, we consider the method in \eqref{eq:explicit1}, recalled in \Cref{alg:smoothForwardBackward}. The iteration requires $f$ and $g$ to be smooth, but no proximal operator evaluations. 
    This is often desirable to enable parallel and scalable computation, and to facilitate the adaptation to stochastic, mini-batch, and zeroth-order scenarios, for instance in empirical risk minimization with linear predictors \citep{Xiao2019}. 
     If $\theta =-1$, \Cref{alg:smoothForwardBackward} retrieves the \gls{PDG} method in \eqref{eq:Gradientdescentascent};  if $\theta = 0$, it retrieves the incremental \gls{PDG} method. In both cases, convergence is not guaranteed without additional assumptions: an example where both methods fail is $f=g=0$ and $A=1$, which satisfies \Cthree{}. 
%   \footnote{We provide a linear convergence proof for \eqref{eq:Gradientdescentascent} under \Cone{} and \Ctwo{} in \Cref{sec:GradientDescentAscent}, which uses our operator-theoretic arguments and improves on some guarantees in the literature}
    Here, we again focus only on the case $\theta = 1$.  
    This choice has the advantage of ensuring convergence for sufficiently small step sizes $\tau, \sigma$, without strong convexity conditions; and it is the only value of $\theta$ that casts \Cref{alg:smoothForwardBackward} as a forward-backward method.
    % ---e.g., immediately enabling the implementation of acceleration schemes \citep{Iutzeler2019Acceleration}). 
   In the following, we show that the choice $\theta =1$  also suffices to ensure contractivity, even under \Cthree{}.

        \begin{algorithm}
            \small
        	\caption{Preconditioned \gls{PDG} method}
        	\label{alg:smoothForwardBackward}
        	\begin{algorithmic}[1]
        		\Require  $x^0\in \R^n, y^0\in \R^m$, step sizes $\tau,\sigma >0$, $\theta =1$
        		\For{$k = 0, 1, \cdots$}
        		\State $x^{k+1} = x^k- \tau \left (\nabla f(x^k) + A^\top y^k \right)$
        		\State $y^{k+1} = y^k - \sigma  \left (\nabla g(y^k) - A \left ((1+\theta)x^{k+1}- \theta x^k \right) \right )$
        		\EndFor
        	\end{algorithmic}
        \end{algorithm}

    % Under the additional assumption that both $f$ and $g$ are smooth, linear convergence of \Cref{alg:smoothForwardBackward} can also be shown under \Cone{} or \Ctwo{}; however we only focus on the minimal condition \Cthree{} here.     

        In particular, \Cref{alg:smoothForwardBackward}
       is derived as in \eqref{eq:precFB} with $F_\textnormal{b} (\omega) = (A^\top y, -Ax)$, $F_{\textnormal{f}}(\omega) = (\nabla f(x), \nabla g(y))$ and $\Phi_{\tau,\sigma}$ as in \eqref{eq:matrixPhiEta}, see \Cref{sec:algderivation3}.  
By refining \Cref{lemma:semiContractionBackward}
and \Cref{lemma:IL1} to the case of \Cref{alg:smoothForwardBackward}, it is easy to show that,
under either \Cone{}, \Ctwo{}, or \Cthree{}, condition \Aone{} in \Cref{prop:partialContractivity} holds with $\Psi_\textnormal{b}=\0$, $\Psi_\textnormal{b}  \succcurlyeq 0 $, $\Psi_\textnormal{b}\succ 0$ respectively (we say that $\mathcal{B}$ is nonexpansive, partially contractive, contractive, respectively). 
 The following lemma provides instead the contractivity properties for the forward step $\F$, necessary for our main result.

   \begin{lemma}\label{lemma:PDGforward}
Let  $F_\textnormal{f}(\omega) = ( \nabla f(x), \nabla g(x) )$. Select $\nu>0$ arbitrary, $0< \tau \leq \frac{2}{\Lf+2\|A\| \nu} $, and $0<\sigma \leq \frac{2}{ \Lg +2 \|A\|\frac{1}{\nu}}$. Let
% Let $\alpha_x  = \frac{  \tau }{1- \tau \|A\|\nu}  - \frac{2}{\Lf+\mf}$, $\alpha_y  = \frac{  \sigma }{1- \sigma \|A\|\frac{1}{\epsilon}}  - \frac{2}{\Lg+\mg}$,
% and
 \begin{align*}
      	\beta_x & = 
     \begin{cases}
     	  2\mf - \frac{\tau  }{1-\tau \|A\|\nu }  \mf^2 & \quad \text {if  } \tau \leq \frac {2}{\Lf+\mf+2\|A\|\nu}
     	\\
     	2 \Lf - \frac{\tau  }{1-\tau \|A\|\nu } \Lf^2 & \quad \text {if  } \tau > \frac {2}{\Lf+\mf+2\|A\|\nu}
     \end{cases}
     \\
         	\beta_y & = 
   \begin{cases}
     	  2\mg - \frac{\sigma  }{1- \sigma \|A\|\textstyle \frac{1}{\nu}  }  \mg^2 & \quad \text {if  } \sigma \leq \frac {2}{\Lg+\mg+2\|A\|\textstyle \frac{1}{\nu} }
     	\\
     	2 \Lg - \frac{\sigma  }{1-\sigma \|A\| \textstyle \frac{1}{\nu} } \Lg^2 & \quad \text {if  } \sigma > \frac {2}{\Lg+\mg+2\|A\| \textstyle \frac{1}{\nu}}
     \end{cases} 
      \end{align*}
Then $\beta_x >0$ whenever $\mf >0 $ and $\tau < \frac{2}{\Lf+2\|A\|\nu}$,  and $\beta_f = 0$ otherwise. Likewise,   $\beta_y >0$ whenever $\mg >0 $ and $\sigma < \frac{2}{\Lg+2\|A\|\frac{1}{\nu}}$, and $\beta_y = 0$ otherwise. Furthermore,  \Atwo{} in  \Cref{prop:partialContractivity} holds with  $\Psi_\textnormal{f} \coloneqq \operatorname{diag}(\beta_x I_n, \beta_y I_m)$.
    \end{lemma}

      \begin{theorem}[Contractivity of \Cref{alg:smoothForwardBackward}]
      	\label{thm:smoothAlgorithm}
    	Let $f: \R^n \rightarrow \R$ and $g: \R^m \rightarrow \R$ be proper, convex, closed and smooth functions. Let $\P$ be as in \eqref{eq:matrixPhiEta}, and let $\epsilon, \nu>0$ be arbitrary.  
    	Then, the fixed points of \Cref{alg:smoothForwardBackward} coincide with the solutions of  \eqref{eq:bilinearSaddlePointProblem}.  Furthermore:
    	\begin{itemize}
    		\item[(i)] If \Cone{} holds and $ \tau < \frac{2}{\Lf+ 2\|A\|\nu }$, $ \sigma < \frac{2}{\Lg+ 2\|A\| \frac{1}{\nu} }$,  then \Cref{alg:smoothForwardBackward} is contractive in  $\norm{\cdot}_{\Phi_{\tau,\sigma}}$ with rate
    		\begin{align*}
    			\rho = 
    			 1 - \min \left\{ \frac{\beta_x \tau }{1+\tau\|A\|\nu},  \frac{\beta_y \sigma }{1+\sigma\|A\|\frac{1}{\nu}} \right \}<1,
    		\end{align*}
    		    with $\beta_x$, $\beta_y$ as in \Cref{lemma:PDGforward};
                %The best rate is obtained with $\tau = \frac{2}{\Lf+\mf+2\|A\|\bar \nu}$, $\sigma =\frac{2}{\Lg +\mg+2\|A\| \frac{1}{\bar \nu}}$, where $\bar \nu = 1$ if $A = \0$ and   $\bar \nu =  \frac{-b+\sqrt{b^2+4ac}}{2a}>0$ otherwise, with $a = \frac{4\|A\|\Lg\mg}{\Lg+\mg}$, $b = \frac{\Lg \mg (\Lf+\mf)^2 + \Lf \mf (\Lg+\mg)^2}{(\Lg+\mg)(\Lf+\mf)}$, $c = \frac{4\Lf\mf \|A\|}{\Lf+\mf}$;
    		\item[(ii)]
    		 If \Ctwo{} holds and $
    		 	 \tau \leq  \frac{2}{\Lf+ 2\|A\| \nu }$, $
    		 	 \sigma < \frac{2}{\Lg+ 2\|A\|  \frac{1}{
    		 			\nu}}$, 
     then \Cref{alg:smoothForwardBackward} is contractive in  $\norm{\cdot}_{\Phi_{\tau,\sigma}+\Psi_\textnormal{b}}$, $\Psi_\textnormal{b} = \diag \left(\frac{\mA}{\zet}I_n,\0 \right)$, with rate
    		\begin{align*}
    			\rho = \sqrt{1 - \frac{\min \{  \mA, \zet \beta_y \}}{\zet^2 +\mA \zet }}<1,
    		\end{align*}
             with $\beta_y$ as in \Cref{lemma:PDGforward}; and
    		\item[(iii)] If \Cthree{} holds and  $
    		 	 \tau \leq  \frac{2}{\Lf+ 2\|A\| \nu }$, $
    		 	 \sigma \leq \frac{2}{\Lg+ 2\|A\|  \frac{1}{
    		 			\nu}}$, $\tau\sigma \|A\|^2(1+\epsilon)^2 \leq 1 $, then \Cref{alg:smoothForwardBackward} is contractive in $\norm{\cdot}_{\Phi_{\tau,\sigma}}$ with rate
    		\begin{align*}
    			\rho =  \frac{  \zet }{\sqrt{\mu_A + \zet^2 }}<1 .
    		\end{align*}
    		The best rate is obtained with $\tau = \sigma = \min \left \{ \frac{1}{(1+\epsilon)\|A\|}, \frac{2}{\Lf+2\|A\| \bar \nu} \right\} $, where  $\bar{\nu} = \frac{\Lg-\Lf +\sqrt{(\Lf-\Lg)^2+ 16\|A\|^2}}{4\|A\|}  $. 
    	\end{itemize}
    	\end{theorem}

    For $A = \0$,  \Cref{thm:smoothAlgorithm} retrieves the usual step sizes bounds for the gradient method and, under \Cone{}, the usual contractivity rate $\rho = \max \{ |1-\tau\mf|, |1-\tau \Lf|, |1-\sigma\mg|, |1-\sigma \Lg| \}$ (see \Cref{prop:propertiesOfGradientStep} in Appendix~\ref{sec:proofPropertiesOfResolvent}).
    % The bounds in \Cref{thm:smoothAlgorithm} also imply $\tau \sigma \|A \|^2 \leq  (1-\frac{\tau \Lf}{2})(1-\frac{\sigma\Lg}{2})$, cf. \Cref{thm:chambollePock,thm:semiSmoothAlgorithm}.
     To the best of our knowledge, the work by \cite{kovalevAcceleratedPrimalDualGradient2022} is the only one to address linear convergence to solutions of \eqref{eq:bilinearSaddlePointProblem} under \Cthree{} (and no further assumptions). The accelerated gradient method in \citep{kovalevAcceleratedPrimalDualGradient2022} is optimal (i.e., it achieves the lower bound complexity) for the case that $f$ and $g$ are strongly convex, but not under \Ctwo{} or \Cthree{}. For example, when $f$ and $g$ are affine and $A$ is invertible, both our \Cref{thm:smoothAlgorithm} and  Theorem 1 in \cite{kovalevAcceleratedPrimalDualGradient2022} guarantee the suboptimal rate $\rho = 1- \Theta \left(\frac{\mu_A}{\|A\|^2} \right)$, where $\Theta(\cdot)$ denotes an asymptotic order, (for the optimal step sizes, and  $\epsilon=1$; as before, $\epsilon =0$ can be chosen if instead either $L_g \neq 0$ or $L_f \neq 0$).

\section{EXTENSIONS}
    \label{sec:conclusion}
    
   Our contractivity certificates pave the way for several extensions of our technical results. We conclude this paper by discussing some of these extensions. 
    % We have proposed an operator theoretic approach to study linear convergence in  bilinear saddle-point problems. 
    
    To start, acceleration schemes (e.g., overrelaxation and momentum \citep{Iutzeler2019Acceleration}, Anderson acceleration \citep{Evans2020}) can directly be applied to \Cref{alg:chambollePock,alg:semiSmoothPrimalDual,alg:smoothForwardBackward}, given their interpretation as forward-backward (or proximal-point) methods and their contractivity. In some cases, improved theoretical rates can also be guaranteed \citep{Evans2020}. 
    % For instance, overrelaxations improves the guaranteed contractivity rate in \Cref{thm:chambollePock},  from $\rho = (1+\alpha)^{-1}$ up to $\rho = (1+2\alpha)^{-1}$ (with $\alpha$ depending on the chosen assumption \Cone{}, \Ctwo{} or \Cthree{}). 

     In addition, contractivity implies important robustness properties, in case of inexact updates, delays, or problems with streaming data (see \Cref{sec:Propertiesofcontractivity}).  The operator-theoretic derivation also enables exploiting the extensive theory developed for nonexpansive and averaged operators \citep{bauschkeConvexAnalysisMonotone2017}, in particular in case of inexact relaxations \citep{Combettes2015Fejeran}.
    These observations should  facilitate the convergence analysis for  noisy and stochastic versions of our algorithms. We leave these generalizations for future work. 
    
    Our approach is not limited to the three algorithms studied nor to the forward-backward scheme in \eqref{eq:precFB}, but  can be tailored to several problem formulations and  serve as the basis for the development of fixed-point methods for many machine learning tasks. For example, we can show that the standard \gls{PDG} method in \eqref{eq:Gradientdescentascent} is contractive under \Ctwo{}; due to space limitations, we defer this result in \Cref{app:extensions}. The analysis is based on a weighted monotonicity argument, and it simplifies and extends the well-known work by \cite{duLinearConvergencePrimalDual2019}. In \Cref{app:extensions} we also design a novel algorithm to solve \eqref{eq:bilinearSaddlePointProblem}, that is contractive in the unweighted Euclidean norm.

    Finally, let us note that the novel inverse-Lipschitz properties proven in \Cref{lemma:IL1} and \Cref{lemma:IL2} are  algorithm independent, and  of general interest. We believe that the inverse-Lipschitz condition has not received sufficient attention in the saddle-point literature; it would be interesting   to study whether our techniques related to this property extend to  general (convex-concave) saddle-point problems with nonlinear coupling.

\bibliographystyle{apalike}
\bibliography{references}

%%%%%%%%%%%%%%%%%%%%%%%%%%%%%%%%%%%%%%%%%%%%%%%%%%%%%%%%%%%%
\section*{Checklist}

% %%% BEGIN INSTRUCTIONS %%%
% The checklist follows the references. For each question, choose your answer from the three possible options: Yes, No, Not Applicable.  You are encouraged to include a justification to your answer, either by referencing the appropriate section of your paper or providing a brief inline description (1-2 sentences). 
% Please do not modify the questions.  Note that the Checklist section does not count towards the page limit. Not including the checklist in the first submission won't result in desk rejection, although in such case we will ask you to upload it during the author response period and include it in camera ready (if accepted).

% \textbf{In your paper, please delete this instructions block and only keep the Checklist section heading above along with the questions/answers below.}
% %%% END INSTRUCTIONS %%%

\begin{enumerate}

 \item For all models and algorithms presented, check if you include:
 \begin{enumerate}
   \item A clear description of the mathematical setting, assumptions, algorithm, and/or model. [Yes]
   \item An analysis of the properties and complexity (time, space, sample size) of any algorithm. [Yes]
   \item (Optional) Anonymized source code, with specification of all dependencies, including external libraries. [Yes]
 \end{enumerate}

 \item For any theoretical claim, check if you include:
 \begin{enumerate}
   \item Statements of the full set of assumptions of all theoretical results. [Yes] Conditions \Cone{} to \Cthree{} and \Cref{thm:chambollePock,thm:semiSmoothAlgorithm,thm:smoothAlgorithm}.
   \item Complete proofs of all theoretical results. [Yes] See Appendix.
   \item Clear explanations of any assumptions. [Yes]     
 \end{enumerate}

 \item For all figures and tables that present empirical results, check if you include:
 \begin{enumerate}
   \item The code, data, and instructions needed to reproduce the main experimental results (either in the supplemental material or as a URL). [Not Applicable]
   \item All the training details (e.g., data splits, hyperparameters, how they were chosen). [Not Applicable]
         \item A clear definition of the specific measure or statistics and error bars (e.g., with respect to the random seed after running experiments multiple times). [Not Applicable]
         \item A description of the computing infrastructure used. (e.g., type of GPUs, internal cluster, or cloud provider). [Not Applicable]
 \end{enumerate}

 \item If you are using existing assets (e.g., code, data, models) or curating/releasing new assets, check if you include:
 \begin{enumerate}
   \item Citations of the creator If your work uses existing assets. [Not Applicable]
   \item The license information of the assets, if applicable. [Not Applicable]
   \item New assets either in the supplemental material or as a URL, if applicable. [Not Applicable]
   \item Information about consent from data providers/curators. [Not Applicable]
   \item Discussion of sensible content if applicable, e.g., personally identifiable information or offensive content. [Not Applicable]
 \end{enumerate}

 \item If you used crowdsourcing or conducted research with human subjects, check if you include:
 \begin{enumerate}
   \item The full text of instructions given to participants and screenshots. [Not Applicable]
   \item Descriptions of potential participant risks, with links to Institutional Review Board (IRB) approvals if applicable. [Not Applicable]
   \item The estimated hourly wage paid to participants and the total amount spent on participant compensation. [Not Applicable]
 \end{enumerate}

 \end{enumerate}

\onecolumn
\newpage
\appendix
\section{Background Material}

% \subsection{Some more operator theoretic background?}\label{app:operatorbackground}

\subsection{Uniqueness of primal-dual solutions}
\label{sec:uniquenessofprimaldual}

    \begin{lemma}
    	Let $f$ and $g$ be proper, convex, and closed functions. If any of   \Cone{}, \Ctwo{}, or \Cthree{} hold, then the bilinear saddle point problem  \eqref{eq:bilinearSaddlePointProblem} has a unique solution $(x^\star,y^\star)$. 
    \end{lemma}
    
    \begin{proof}
        We recall some properties of the  Fenchel conjugate  $g^*(v) = \sup \{ \langle v,y \rangle -g(y): y \in \R^m \}$ of a proper convex closed function $g$ \cite[Ch.13, Th.~18.15]{bauschkeConvexAnalysisMonotone2017}:  $g^*$ is proper convex closed; if $g$ is strongly convex, $g^*$ is smooth (hence, $\dom(g) = \R^m$); if $g$ is smooth, $g^*$ is strongly convex; $\partial g^* = (\partial g)^{-1}$.  Then, under any of the conditions \Cone{}, \Ctwo{}, or \Cthree{}, the primal problem \eqref{eq:primal} has nonempty domain and is strongly convex (under \Cone{}, $f$ is strongly convex; under \Ctwo{} and \Cthree{}, $g^*\circ A$ is), hence it admits a unique solution $x^\star \in \R^n$. Furthermore, under  either condition \Cone{}, \Ctwo{}, or \Cthree{}, we have that $\dom (g^*) \cap A(\dom (f)) \neq \varnothing$ (under \Cone{} or \Ctwo{},   $\dom(g^*) = \R^m$ and $f$ is proper; under \Cthree{}, $A(\dom f) = \R^n$ and $g^*$ is proper). Therefore, we can apply Theorem~19.1 in \cite {bauschkeConvexAnalysisMonotone2017}; in particular, strong duality holds, and   the solution to \eqref{eq:primal} is the unique vector $x^\star$ for which there exists $y^\star\in \R^m$ such that \cite[Theorem~19.1(ii)]{bauschkeConvexAnalysisMonotone2017}
    	\begin{align}
        \label{eq:equations}
    		 \0_{n+m}  \in F(x^\star,y^\star) \coloneqq \begin{bmatrix}
    		 	\partial f(x^\star) + A^\top y^\star \\
    		 	\partial g(y^\star ) - Ax^\star
    		 	\end{bmatrix},
    	\end{align}
         namely, such that $(x^\star,y^\star)$ solves \eqref{eq:bilinearSaddlePointProblem}. Hence, a solution $(x^\star,y^\star )$ to \eqref{eq:bilinearSaddlePointProblem} must exist. Furthermore, the dual solution $y^\star$ is also unique, as, given the solution $x^\star$ of \eqref{eq:primal}: if \Cone{} or \Ctwo{} hold, then $\partial g$ is strongly monotone, hence the second inclusion in \eqref{eq:equations}, $\0 \in \partial g(y^\star) - Ax^\star$, has a unique solution  $y^\star$ \cite[Ex.~22.12]{bauschkeConvexAnalysisMonotone2017}; if \Cthree{} holds, then the first inclusion in  \eqref{eq:equations}, $\0 \in \nabla f(x^\star) + A^\top y^\star$, has a unique solution $y^\star$, since $A$ is invertible.
    \end{proof}

That the conditions \Cone{}, 
\Ctwo{}, \Cthree{} are tight (i.e., that uniqueness of primal-dual solutions cannot be ensured if any of the subconditions is relaxed) can be verified via simple counterexamples. Consider the following  cases:
\begin{enumerate}
    \item[(I)] $n =1$, $m=2$,  $f(x)=0$, $g(y) = 0$, $A = \begin{bsmallmatrix}
        0 \\ 1 
    \end{bsmallmatrix} $; 
    \item[(II)] $n=m=1$, $f(x) = 0$, $g = \iota_{[0,1]}(y)+y^2$, $A = 1 $ (where $\iota_{[0,1]}$ is the indicator function of the interval $[0,1]$); \item[(III)] $n=m=1$, $f(x) = 0$, $g(y)= y^2$, $A  =0$.
    \end{enumerate}
Recall the subconditions: \Cone{}(a): $f$ is strongly convex (symmetrically, \Cone{}(b): $g$ is strongly convex); \Ctwo{}(a) $g$ is strongly convex; \Ctwo(b): $g$ is smooth; \Ctwo(c) $A$ is full column rank; \Cthree{}(a):  $g$ is smooth (symmetrically, \Cthree{}(b) $f$ is smooth);  \Cthree{}(c): $A$ is full row rank (symmetrically, \Cthree(d): $A$ is full column rank). 

Case (III) violates \Cone{}(a) (but satisfies the other subconditions in \Cone{}). Case (I) violates  \Ctwo{}(a). Case (II) violates  \Ctwo{}(b). Case (III) violates  \Ctwo(c). Case (II) violates \Cthree(a). Case (I)  violates \Cthree(c).

It is easy to see that in all cases (I), (II), and (III), multiple primal-dual solutions exist for the inclusion in \eqref{eq:equations}, hence for the saddle point problem in \eqref{eq:bilinearSaddlePointProblem} (in which case, contractivity for \emph{any} algorithm is excluded). Therefore, conditions \Cone{}, \Ctwo{}, \Cthree{} are tight.

\subsection{Contractivity versus  Q-linear convergence}\label{sec:Propertiesofcontractivity}

We illustrate the difference between contractivity and Q-linear convergence. Let us consider an iteration $\omega^{k+1} = \mathcal{A}(\omega^k)$, where $\mathcal{A}:\R^q\rightarrow \R^q$ has a unique fixed point, i.e., there is a unique $\omega^\star\in\R^q$ such that $\mathcal{A}(\omega^\star) = \omega^\star$.  
 The iteration is contractive (in $\| \cdot\|)$  with rate $0 < \rho <1$ if, for all $\omega,\omega' \in \R^q$, it holds that $\|\mathcal A(\omega)- \mathcal A(w') \| \leq \rho  \|\omega - \omega ' \| $,  namely any two trajectories of the algorithm converge linearly to each other. The iteration is globally Q-linearly convergent with rate $0 < \rho <1$ if, for all $\omega \in \R^q$, it holds that $ \|\mathcal A(\omega)- \mathcal A(w^\star) \| \leq \rho  \|\omega - \omega^\star \|$, therefore the iterates converge geometrically to $\omega^\star$.
 
 Clearly, contractivity implies Q-linear convergence, by taking $\omega' = \omega^\star$. 
 Conversely, there are many iterations that are Q-linearly convergent, but not contractive. For instance, the forward method for restricted strongly monotone operators \cite[Theorem~7]{Tatarenko2021}, or the scalar iteration
\begin{align}\label{ex:contractivity}
    \omega^{k+1} = \mathcal{A} (\omega^k) =  \rho \operatorname{sin}(\omega) \omega,
\end{align}
with $0<\rho<1$,which converges Q-linearly to $\omega^\star = 0$ with rate $\rho$, but  is not contractive. The gap between contractivity and Q-linear convergence is significant in terms of properties, proof strategy, and practical relevance. 

\emph{Properties}: Contractive algorithms enjoy superior robustness properties compared to Q-linear convergent iterations. For instance, contractivity ensures algorithm stability, in the sense that small changes in the initialization result in small changes in the execution path. This is not the case for Q-linearly convergent algorithms, as it can be checked on \eqref{ex:contractivity}. As another example, consider the perturbed iteration $\omega^{k+1}= \mathcal{A}(\omega^k) +d(\omega^k)$. If $\mathcal A$ is $\rho$ contractive,  the error $d(\omega)$ is $L_d$-Lipschitz, and $\rho +L_d<1$, then the perturbed iteration is still contractive, so it converges to a (perturbed) fixed point. This robustness cannot be guaranteed for Q-linearly convergent iterations. For instance, consider the perturbed version of \eqref{ex:contractivity},
$\omega^{k+1} = \mathcal{A} (\omega^k) +\bar d =  \rho \operatorname{sin}(\omega^k) \omega^k + \bar d$, where $\bar d$ is a constant independent of $\omega^k$ (so $L_d =0)$. 
 For large enough $\bar d>0$, the iteration has multiple fixed-points and it fails to converge, oscillating indefinitely. This would not happen for a contractive iteration.
Finally, an important property of contractive iterations is modularity: the composition of contractive operators is contractive.  At the contrary, if the iterations $\omega^{k+1} = \mathcal{A} (\omega^k)$ and $\omega^{k+1} = \mathcal{B} (\omega^k)$ are only Q-linearly convergent, the iteration $\omega^{k+1} = \mathcal A \circ \mathcal B (\omega^k)$ needs not converge. 

\emph{Proof strategy}: Proving contractivity is generally more involved than proving  Q-linear convergence. One reason is that to show $|| \mathcal{A}(w) - \mathcal{A}(w^\star) ||\leq \rho || \omega - \omega^\star || $,  one can exploit the fact that $\omega^\star$ is a fixed point of $\mathcal A$. Of course, it is not always possible to extend Q-linear convergence arguments to  contractivity, as there are algorithms that are Q-linearly convergent but not contractive, such as \eqref{ex:contractivity}. More importantly, even if an iteration is contractive, a Lyapunov function that proves
Q-linear convergence for an algorithm may not be suitable for demonstrating its
contractivity, potentially requiring a different Lyapunov function. 

\emph{Practical relevance}: 
Because of its superior properties, contractivity plays an important role when an algorithm is used as a basis 
for the development of more complex methods. The robustness properties of contractive algorithms  can be pivotal in studying perturbed iterations, or time-varying problems, such as those  in  learning with data streams; see the recent work  by  \cite{davydov2025timevaryingconvexoptimizationcontraction} for a continuous-time perspective. As another example, \cite{Scutari_Unified} exploit the modularity of contractive iterations (specifically, of gradient descent) for the analysis of composite methods in distributed optimization \cite[7, Prop. 9]{Scutari_Unified}; this derivation would not be possible based on Q-linear convergence guarantees alone. 
Finally, contractivity is required to guarantee convergence of  general modification schemes, such as Anderson acceleration \cite[Asm.~3.2]{Evans2020}.

    \subsection{Properties of $\Phi_{\tau, \sigma}$}
        \label{sec:propertiesPhi}
          \begin{proposition}
                \label{prop:propertiesPhi}
               Let $\Phi_{\tau, \sigma} = \begin{bmatrix}
     		     \frac{1}{\tau}I_n &-A^\top \\
     		     -A & \frac{1}{\sigma}I_m
                \end{bmatrix} $ as in \eqref{eq:matrixPhiEta}. Then:
                \begin{itemize}
                        \item[i)] $\Phi_{\tau, \sigma}$ is symmetric and positive definite if $\tau \sigma \norm{A}^2< 1$.
                        \item[ii)]  If $\tau \sigma \norm{A}^2< 1$, the maximum eigenvalue of $\P$ can then be upper bounded as follows
                        \begin{equation*}
                                \lambdamax{\Phi_{\tau, \sigma}} = \norm{\Phi_{\tau, \sigma}} \leq \max \left \{ \frac{1}{\tau}, \frac{1}{\sigma}\right \}+ \norm{A} \eqqcolon \zet{}.
                            \end{equation*}
                    \end{itemize}
            \end{proposition}
        \begin{proof}
             \textit{Proof of item i)}. 
             By Schur complement, $\P$ is positive definite if and only if $\frac{1}{\tau}I_n - \sigma A^\top A \succ 0 $, which is implied by $\tau \sigma \|A\|^2 <1.$
            % Symmetry follows directly from the definition. Let $x\in \R^n$ and $y \in \R^m$ be nonzero vectors but otherwise arbitrary and let $\tau \sigma \norm{A}^2< 1$, then we have
            % \begin{align*}
            %     \begin{bmatrix}
            %         x^\top  &y^\top 
            %     \end{bmatrix} \begin{bmatrix}
            %         \frac{1}{\tau}I_n &- A^\top \\
            %         -A &\frac{1}{\sigma}I_m
            %     \end{bmatrix} \begin{bmatrix}
            %         x\\
            %         y
            %     \end{bmatrix} &= \innerproduct{\begin{bmatrix}
            %         \frac{1}{\tau}x-A^\top y\\
            %         \frac{1}{\sigma}y- Ax
            %     \end{bmatrix}}{\begin{bmatrix}
            %         x\\
            %         y
            %     \end{bmatrix}} \\
            %     &=\frac{1}{\tau}\norm{x}^2- 2 \innerproduct{A^\top y}{x}+\frac{1}{\sigma}\norm{y}^2\\
            %     &\overset{(i)}{\geq} \frac{1}{\tau}\norm{x}^2- 2 \norm{A}\norm{x}\norm{y}+\frac{1}{\sigma}\norm{y}^2\\
            %     &= \left(\frac{1}{\sqrt{\tau}}\norm{x} - \sqrt{\tau}\norm{A} \norm{y} \right )^2 + \left (\frac{1}{\sigma}-\tau\norm{A}^2 \right)\norm{y}^2\\
            %     &> 0
            % \end{align*}
            % where in (i) we used Cauchy-Schwarz and the definition of the operator norm. \par
            
            \textit{Proof of item ii)}. Since $\Phi_{\tau, \sigma}$ is symmetric and positive semidefinite for $\tau \sigma \norm{A}^2 \leq 1$, we have
            \begin{IEEEeqnarray*}{rCl}
                \lambdamax{\Phi_{\tau,\sigma}} &=& \norm{\Phi_{\tau,\sigma}}\\
                & = & \norm{\begin{bmatrix}
                    \frac{1}{\tau}I_n &0\\
                    0 &\frac{1}{\sigma}I_m
                \end{bmatrix} + \begin{bmatrix}
                    0 & -A^\top \\
                    -A &0
                \end{bmatrix}}\\
                &\leq& \max \left \{ \frac{1}{\tau}, \frac{1}{\sigma}\right \}+\sqrt{\lambdamax{\begin{bmatrix}
                    A^\top A &0\\
                    0 &AA^\top 
                \end{bmatrix}}}\\
                & = & \max \left \{ \frac{1}{\tau}, \frac{1}{\sigma}\right \}+\sqrt{\lambdamax{AA^\top }}\\
                & = & \max \left \{ \frac{1}{\tau}, \frac{1}{\sigma}\right \}+\norm{A}.
            \end{IEEEeqnarray*} 
        \end{proof}

\subsection{Properties of inverse Lipschitz operators}\label{sec:PropertiesOfInverseLipschitz}

\begin{lemma}
    Let $\mathcal{A} : \R^q  \rightrightarrows \R^q$ be $\frac{1}{L}$-inverse Lipschitz, i.e., $\norm{u-u'} \geq \frac{1}{L}\norm{\omega-\omega'}$  for all $(\omega,u), (\omega',u') \in \gra(\mathcal{A})$. Then $\mathcal{A}$ has at most one zero, i.e., there exists at most one $\omega^\star \in \R^q$ such that $\0 \in \mathcal{A}(\omega^\star)$. 
% The operator $\mathcal{A} : \R^q  \rightrightarrows \R^q$ is   $\frac{1}{L}$-inverse Lipschitz  if  $\norm{u-u'} \geq \frac{1}{L}\norm{\omega-\omega'}$  for all $(\omega,u), (\omega',u') \in \gra(\mathcal{A})$, namely if $\mathcal{A}^{-1}$ is $L$-Lipschitz.
\end{lemma}

\begin{proof}
    The  definition of inverse Lipschitz implies that, if    
    $(\omega,\0), (\omega',\0) \in \gra(\mathcal{A})$,  then $ \| \0 -\0 \|  \geq \frac{1}{L}\norm{\omega-\omega'}$, and hence $\omega = \omega '$.
\end{proof}

\begin{lemma}
    Let $\mathcal{A} : \R^q  \rightrightarrows \R^q$ be $\mu$-strongly monotone. Then, $\mathcal{A}$ is $\mu$-inverse Lipschitz. 
\end{lemma}

\begin{proof}
    By definition of strong monotonicity and the Cauchy--Schwartz inequality, for any $(\omega,u),(\omega',u')\in\gra(\mathcal{A})$,  $\mu\|\omega-\omega'\|^{2} \leq  \langle u-u' \mid \omega-\omega'\rangle  \leq  \| u-u' \| \| \omega-\omega'\ \|$, and hence $\|u -u'\| \geq \mu \| \omega - \omega'\|$.
\end{proof}

    \section{Algorithms derivation}
        \label{sec:algorithmsderivation}
        The  three algorithms \eqref{eq:ChambollePock1}, \eqref{eq:semiimplicit1} and \eqref{eq:explicit1}, are derived from forward-backward splitting  in \eqref{eq:precFB}. The preconditioning matrix $\Phi_{\tau, \sigma}$ is defined in \eqref{eq:matrixPhiEta}, and we assume throughout that $\tau \sigma \|A\|^2 < 1$, so that $\P \succ 0$. The forward-backward iteration in \eqref{eq:precFB} can thus be reformulated as follows
        \begin{IEEEeqnarray}{LrCl}
            \label{eq:precFBReformulated}
            & \omega^{k+1} & = & (\Id + \Phi_{\tau, \sigma}^{-1}\Fb)^{-1} \circ (\Id-\Phi_{\tau, \sigma}^{-1} \Ff) (\omega^k) \nonumber\\
            \overset{(i)}{\Leftrightarrow} & \left (\Id + \Phi_{\tau, \sigma}^{-1}\Fb \right )(\omega^{k+1}) & \ni & \left ( \Id - \Phi_{\tau, \sigma}^{-1} \Ff \right ) (\omega^k) \nonumber\\
            \overset{(ii)}{\Leftrightarrow} & \left (\Phi_{\tau, \sigma}+ \Fb \right ) (\omega^{k+1}) & \ni & \left ( \Phi_{\tau, \sigma}- \Ff \right )(\omega^k),
        \end{IEEEeqnarray}
        where (i) follows from the definition of  inverse operator, and (ii) holds since $\Phi_{\tau, \sigma}$ is invertible. 

        \begin{remark}
    For either of \Cref{alg:chambollePock,alg:semiSmoothPrimalDual,alg:smoothForwardBackward}, the operators $F_\textnormal{f}$ and $F_\textnormal{b}$ are maximally monotone \citep[Cor.~25.5]{bauschkeConvexAnalysisMonotone2017}, as they all can be written as the sum of the subdifferential of a convex function (which is maximally monotone) and a linear skew symmetric operator (which is maximally monotone and with full domain). In particular, this implies that the backward step $\mathcal B = (\operatorname{Id}+\Phi_{\sigma,\tau}^{-1} F_\textnormal{b})^{-1}$ is always single-valued and has full domain  \cite[Prop.~20.14, Prop.~23.8]{bauschkeConvexAnalysisMonotone2017}, even though $F_\textnormal{b}$ might be set-valued (hence the ``$\ni$'' rather than ``$=$'' sign in the first equivalence above).
\end{remark}

        \subsection{\Cref{alg:chambollePock} (Equation \eqref{eq:ChambollePock1})}
        \label{sec:algderivation1}
           We choose $\Ff = \0$, $\Fb = F$ as described in \Cref{sec:chambollePock}. Explicitly writing \eqref{eq:precFBReformulated} yields the following:
           \begin{IEEEeqnarray*}{LrCl}
                & \left (\Phi_{\tau, \sigma}+ \Fb \right ) (\omega^{k+1}) & \ni & \left ( \Phi_{\tau, \sigma}- \Ff \right )(\omega^k)\\
                \Leftrightarrow & \begin{bmatrix}
                    \frac{1}{\tau}x^{k+1}+\partial f(x^{k+1}) + \cancel{A^\top y^{k+1}}- \cancel{A^\top y^{k+1}}\\
                    \frac{1}{\sigma}y^{k+1} -Ax^{k+1}-Ax^{k+1}\\
                \end{bmatrix} &\ni&  \begin{bmatrix}
                    \frac{1}{\tau}x^k-A^\top  y^k\\
                    \frac{1}{\sigma}y^k - Ax^k
                \end{bmatrix} 
                \label{eq:inclusioncancel}
                \\
                \Leftrightarrow & \begin{bmatrix}
                     x^{k+1} + \tau \partial f(x^{k+1})\\
                      y^{k+1} + \sigma \partial g(y^{k+1})
                \end{bmatrix} & \ni & \begin{bmatrix}
                    x^k - \tau A^\top y^k\\
                    y^k + \sigma A\left ( 2x^{k+1}- x^k \right)
                \end{bmatrix}\\
                \Leftrightarrow & \begin{bmatrix}
                    x^{k+1}\\
                    y^{k+1}
                \end{bmatrix} & = & \begin{bmatrix}
                    \operatorname{prox}_{\tau f}\left( x^k - \tau A^\top y^k \right)\\
                    \operatorname{prox}_{\sigma g}\left ( y^k + \sigma A\left ( 2x^{k+1}- x^k \right) \right),
                \end{bmatrix}
           \end{IEEEeqnarray*}
            where the last equivalence follows by definition of $\operatorname{prox}$ operator (we recall that for a proper closed convex function $\psi$, $\operatorname{prox}_\phi = (\operatorname{Id}+ \partial \psi)^{-1}$ is single valued,  see \eqref{eq:prox}, although $\partial \psi$ might not be; hence ``$=$'' rather than ``$\in$'' is used in the last line).

        \begin{remark}[Design rationale]
        Informally speaking, convergence of the forward-backward method \eqref{eq:precFB} typically requires only monotonicity for the operator $F_\textnormal{b}$, but a stronger ``cocoercivity'' condition for the operator $F_\textnormal{f}$ \cite[Th.~26.14]{bauschkeConvexAnalysisMonotone2017} (cocoercivity is for instance satisfied by gradients of smooth functions). For this reason, to ensure convergence even without strong convexity assumptions, we always place the skew symmetric operator $(A^\top y, -Ax)$, which is not cocoercive, in the backward step. On the downside, this complicates the computation of  $\mathcal B$. To remedy,  the preconditioning matrix $\P$ is designed to make the system of inclusions above block triangular, i.e., to remove the term $A^\top y^{k+1}$ in the first inclusion. In this way, $x^{k+1}$ does not depend on $y^{k+1}$. This ensures that the resulting iteration can be computed efficiently (provided that the proximal operators for the functions $f$ and $g$ are available). 
        \end{remark}
        \subsection{\Cref{alg:semiSmoothPrimalDual} (Equation \eqref{eq:semiimplicit1}) } \label{sec:algderivation2}
            We choose $\Ff(\omega) = (\0, \nabla g(y))$ and $\Fb(\omega) = (\partial f(x)+A^\top y, -Ax)$ as described in \Cref{sec:semiSmoothPrimalDual}. Then, we have from \eqref{eq:precFBReformulated}:
            \begin{IEEEeqnarray*}{LrCl}
                 & \left (\Phi_{\tau, \sigma}+ \Fb \right ) (\omega^{k+1}) & \ni & \left ( \Phi_{\tau, \sigma}- \Ff \right )(\omega^k)\\
                 \Leftrightarrow & \begin{bmatrix}
                    \frac{1}{\tau}x^{k+1}+\partial f(x^{k+1}) + \cancel{A^\top y^{k+1}}- \cancel{A^\top y^{k+1}}\\
                    \frac{1}{\sigma}y^{k+1} -Ax^{k+1}-Ax^{k+1}\\
                \end{bmatrix} &\ni&  \begin{bmatrix}
                    \frac{1}{\tau}x^k-A^\top  y^k\\
                    \frac{1}{\sigma}y^k - Ax^k-\nabla g(y^k)
                \end{bmatrix}\\
                \Leftrightarrow & \begin{bmatrix}
                      x^{k+1} + \tau \partial f(x^{k+1})\\
                    y^{k+1}
                \end{bmatrix} & \ni & \begin{bmatrix}
                      x^k - \tau A^\top y^k\\
                      y^k - \sigma \left (\nabla g(y^k) - A \left (2x^{k+1}-x^{k} \right ) \right )
                \end{bmatrix}\\
                 \Leftrightarrow & \begin{bmatrix}
                    x^{k+1}\\
                    y^{k+1}
                \end{bmatrix} & = & \begin{bmatrix}
                     \operatorname{prox}_{\tau f}\left(x^k - \tau A^\top y^k \right)\\
                     y^k - \sigma \left (\nabla g(y^k) - A \left (2x^{k+1}-x^{k} \right ) \right ),
                \end{bmatrix}
            \end{IEEEeqnarray*}
            which concludes the derivation of \Cref{alg:semiSmoothPrimalDual}.

        \subsection{\Cref{alg:smoothForwardBackward} (Equation \eqref{eq:explicit1}) } \label{sec:algderivation3}
            We choose $\Fb(\omega) = (A^\top y, -Ax)$ and $\Ff(\omega) = (\nabla f(x), \nabla g(y))$ as given in \Cref{sec:smoothPrimalDual}. From \eqref{eq:precFBReformulated}, we have
              \begin{IEEEeqnarray*}{LrCl}
                 & \left (\Phi_{\tau, \sigma}+ \Fb \right ) (\omega^{k+1}) & \ni & \left ( \Phi_{\tau, \sigma}- \Ff \right )(\omega^k)\\
                 \Leftrightarrow  & \begin{bmatrix}
                    \frac{1}{\tau}x^{k+1} + \cancel{A^\top y^{k+1}}- \cancel{A^\top y^{k+1}}\\
                    \frac{1}{\sigma}y^{k+1} -Ax^{k+1}-Ax^{k+1}\\
                \end{bmatrix} &=&  \begin{bmatrix}
                    \frac{1}{\tau}x^k-A^\top  y^k-\nabla f(x^k)\\
                    \frac{1}{\sigma}y^k - Ax^k-\nabla g(y^k)
                \end{bmatrix}\\
                \Leftrightarrow & \begin{bmatrix}
                      x^{k+1}\\
                    y^{k+1}
                \end{bmatrix} & = & \begin{bmatrix}
                      x^k - \tau \left ( \nabla f(x^k) + A^\top y^k \right )\\
                      y^k - \sigma \left (\nabla g(y^k) - A \left (2x^{k+1}-x^{k} \right ) \right ),
                \end{bmatrix}
            \end{IEEEeqnarray*}
            where in the first equivalence we have equality since all operators are single-valued.

\section{Omitted Proofs}
    \label{sec:omittedproofs}

 \subsection{Proof of \Cref{prop:PropertiesOfResolvent}} \label{sec:proofPropertiesOfResolvent}

We will also recall the case of strongly  monotone operators. 

\begin{proposition}[Contractivity of Resolvents]
    Let $F:\R^n \rightrightarrows \R^n$ be  a maximally monotone operator. Then its resolvent $\J_F = (\operatorname{Id} +F)^{-1}$ is everywhere defined and single-valued. Moreover:
	\begin{itemize}
		\item[(i)] If $F$ is $\mu$-strongly monotone, then the resolvent $\J_F$ is  $\frac{1}{1+\mu}$-Lipschitz. 
		\item[(ii)] If $F$ is $\frac{1}{L}$ inverse Lipschitz, then the resolvent $\J_F$ is  $\frac{L}{\sqrt{L^2+1}}$-Lipschitz.
	\end{itemize}
\end{proposition}

            \begin{proof}
                That the resolvent of a maximally monotone operator is full-domain and single valued is the well-known Minty's Theorem \cite[Prop.~23.8]{bauschkeConvexAnalysisMonotone2017}.
                
                \textit{Proof of item i)}. The operator $\operatorname{Id+F}$ is $(1+\mu)$ -strongly monotone, so its inverse is $\frac{1}{1+\mu}$-Lipschitz. 
                
                \textit{Proof of item ii)}. For any $x, x' \in \R^n$, let $u = \J_{F}(x) $ and $u' = \J_{F}(x')$, which implies $u+F(u) \ni x'$ and $u'+F(u') \ni x'$, by definition of inverse operator. Hence, for some $v \in F(u)$ and $v' \in F(u')$, we have
            \begin{IEEEeqnarray*}{rCl}
                \norm{x-x'}^2 & = & \norm{u+v-u' - v'}^2\\
                & = & \norm{u-u'}^2 +  \norm{v-v'}^2 + 2 \innerproduct{u-u'}{v-v'}\\
                & \overset{(i)}{\geq} & \norm{u-u'}^2 + \frac{1}{L^2}\norm{u-u'}^2\\
                & = & \left (1+\frac{1}{L^2}\right ) \norm{u-u'}^2,
            \end{IEEEeqnarray*}
            where monotonicity and the inverse Lipschitz property of $A$ were used in (i). Taking a square root on both sides proves that $(\Id+F)^{-1}$ is Lipschitz with constant $\frac{L}{\sqrt{L^2+1}}<1$.
            \end{proof}

For future reference, let us also recall the following, well-known, result (we will prove a more general result in \Cref{thm:smoothAlgorithm}, hence we do not include a proof here). 

    \begin{proposition}[Contractivity of Gradient Step]
        \label{prop:propertiesOfGradientStep}   
        Let $f:\R^n \rightarrow R$ be 
         $\mu$-strongly convex and $L$-smooth. Then, for any $0<\alpha< \frac{2}{L}$, the operator $(\Id - \alpha \nabla f)$ is $\rho$-Lipschitz, with $\rho = \max \{|1-\alpha L|,|1-\alpha \mu|\}$.
    \end{proposition}

        \subsection{Proof of Proposition~\ref{prop:partialContractivity}}
            \label{app:propPartialContractivity}

        For simplicity of notation, we use a slightly refined (but equivalent) version of \Cref{prop:partialContractivity} in some of our proofs. Clearly,  \Cref{prop:partialContractivity} is retrieved from \Cref{prop:partialContractivityB} by choosing $\rho_\textnormal{f} =\rho _\textnormal{b} =1$.

        \begin{proposition}\label{prop:partialContractivityB}
        Let $\Phi_{\tau,\sigma} \succ 0$ as in \eqref{eq:matrixPhiEta}. Assume that there exists scalars $0 \leq  \rho_\textnormal{f}, \rho_\textnormal{b}, \gamma \leq  1$ and matrices $\Psi_{\textnormal b} \succcurlyeq 0 $, $\Psi _{\textnormal f}\succcurlyeq 0$, such that, for all $\omega, \omega' \in \R^{n+m}$,
        \begin{itemize}
            \item[\Aone{}.] $\|\mathcal{B}(\omega) - \mathcal{B}(\omega')\|^2_{\Phi_{\tau,\sigma}+\Psi_{\textnormal{b}}} 
            \leq  \rho_\textnormal{b}^2 \| \omega -\omega'\|^2_{\Phi_{\tau,\sigma}}$
            \item[\Atwo{}.] $\|\mathcal{F}(\omega) - \mathcal{F}(\omega')\|^2_{\Phi_{\tau,\sigma}} 
            \leq \rho_\textnormal{f}^2 \| \omega -\omega'\|^2_{\Phi_{\tau,\sigma}-\Psi_{\textnormal{f}}}$
        \end{itemize}
        and $ \Psi_{\textnormal b}+ \Psi _{\textnormal f} \succcurlyeq \gamma (\Phi_{{\tau,\sigma}}+\Psi_{\textnormal b}) $. 
        Then, the forward-backward iteration in \eqref{eq:precFB} is contractive in $\| \cdot\|_{\Phi_{{\tau,\sigma}}+\Psi_{\textnormal{b}}} $, with rate 
        $\rho = \rho_\textnormal{b}\rho_\textnormal{f} \sqrt{1-\gamma}$. 
    \end{proposition}

            \begin{proof}
                Let $\omega, \omega' \in \R^{n+m}$. We have
                % \begin{IEEEeqnarray*}{rCl}
                %     \IEEEeqnarraymulticol{3}{l}{ \norm{\mathcal{B}\circ \mathcal{F}(\omega)- \mathcal{B}\circ \mathcal{F}(\omega')}_{\Phi_{\tau, \sigma}+\Psi_b}^2} \nonumber\\
                %     & \overset{(i)}{\leq} & \rho_b^2 \norm{\mathcal{F}(\omega)-\mathcal{F}(\omega')}_{\Phi_{\tau, \sigma}}^2\\
                %     & \overset{(ii)}{\leq} & \rho_f^2 \rho_b^2 \norm{\omega - \omega'}_{\Phi_{\tau, \sigma}-\Psi_f}^2\\
                %     & = & \rho_f^2 \rho_b^2 \left (\omega-\omega' \right )^\top \left (\Phi_{\tau, \sigma}-\Psi_f\right ) \left(\omega - \omega' \right)\\
                %     & \overset{(iii)}{\leq} & \rho_f^2 \rho_b^2 \left ( \omega-\omega') \right)^\top \left (\Phi_{\tau, \sigma}+\Psi_b - \gamma (\Phi_{\tau, \sigma}+\Psi_b) \right) \left ( \omega-\omega' \right)\\
                %     & = & \rho_f^2 \rho_b^2 (1-\gamma) \left(\omega-\omega' \right )^\top \left (\Phi_{\tau, \sigma}+\Psi_b\right ) \left(\omega - \omega' \right)\\
                %     & = & \rho_f^2 \rho_b^2 (1-\gamma) \norm{\omega-\omega'}_{\Phi_{\tau, \sigma}+\Psi_b}^2
                % \end{IEEEeqnarray*}
                \begin{IEEEeqnarray*}{rCl}
                    \IEEEeqnarraymulticol{3}{l}{ \norm{\mathcal{B}\circ \mathcal{F}(\omega)- \mathcal{B}\circ \mathcal{F}(\omega')}_{\Phi_{\tau, \sigma}+\Psi_b}^2} \nonumber\\
                    & \overset{(i)}{\leq} & \rho_\textnormal{b} ^2\norm{\mathcal{F}(\omega)-\mathcal{F}(\omega')}_{\Phi_{\tau, \sigma}}^2\\
                    & \overset{(ii)}{\leq} &  
                    \rho_\textnormal{b}^2\rho_\textnormal{f}^2
                    \norm{\omega - \omega'}_{\Phi_{\tau, \sigma}-\Psi_f}^2\\
                    & = &  \rho_\textnormal{b}^2\rho_\textnormal{f}^2 \left (\omega-\omega' \right )^\top \left (\Phi_{\tau, \sigma}-\Psi_f\right ) \left(\omega - \omega' \right)\\
                    & \overset{(iii)}{\leq} &  \rho_\textnormal{b}^2\rho_\textnormal{f}^2 \left ( \omega-\omega' \right)^\top \left (\Phi_{\tau, \sigma}+\Psi_b - \gamma (\Phi_{\tau, \sigma}+\Psi_b) \right) \left ( \omega-\omega' \right)\\
                    & = &  \rho_\textnormal{b}^2\rho_\textnormal{f}^2(1-\gamma) \left(\omega-\omega' \right )^\top \left (\Phi_{\tau, \sigma}+\Psi_b\right ) \left(\omega - \omega' \right)\\
                    & = & \rho_\textnormal{b}^2\rho_\textnormal{f}^2  (1-\gamma) \norm{\omega-\omega'}_{\Phi_{\tau, \sigma}+\Psi_b}^2,
                \end{IEEEeqnarray*}
                 where (i) follows from \Aone{}, (ii) from \Atwo{} and (iii) from $\Psi_{\textnormal b}+ \Psi _{\textnormal f} \succcurlyeq \gamma (\Phi_{{\tau,\sigma}}+\Psi_{\textnormal b})$. Taking the square root on both sides concludes the proof.
            \end{proof}
          
        \subsection{\texorpdfstring{Proof of \Cref{lemma:IL1}}{chambollePockC2InverseLipschitz}}
                The lemma assumes that \Ctwo{} holds, hence $g$ is strongly convex and smooth and the matrix $A$ has full column rank. Define the following matrix $\Psi_{\alpha} \in \R^{(n+m) \times (n+m)}$, for $\alpha >0$ small enough to be chosen,
                \begin{equation*}
                    \Psi_{\alpha} = \begin{bmatrix}
                        I_n &0\\
                        -\alpha A &I_m
                    \end{bmatrix}.
                \end{equation*}
               For any $\omega =(x,y), \omega'=(x',y') \in \R^{n+m}$, let $v \in F(\omega)$, $v' \in F(\omega')$, so that, for some  $u \in \partial f (x)$ and $u' \in \partial f(x')$ 
\allowdisplaybreaks{\begin{IEEEeqnarray*}{rCl}
                        \IEEEeqnarraymulticol{3}{l}{
                             \norm{v-v'}\norm{\Psi_{\alpha}} \norm{\omega - \omega'}} \nonumber \\
                        &=& \norm{\begin{bmatrix}
                                u + A^\top y-(u'+A^\top y')\\
                                \nabla g(y)-Ax-(\nabla g(y')-Ax')
                        \end{bmatrix}} \norm{\begin{bmatrix}
                                I_n &0\\
                                -\alpha A &I_m
                        \end{bmatrix}} \norm{\begin{bmatrix}
                                x-x'\\
                                y-y'
                        \end{bmatrix}}\\
                        & \geq & \norm{\begin{bmatrix}
                                u + A^\top y-(u'+A^\top y')\\
                                \nabla g(y)-Ax-(\nabla g(y')-Ax')
                        \end{bmatrix}} \norm{\begin{bmatrix}
                                x-x'\\
                                -\alpha A(x-x')+y-y'
                        \end{bmatrix}}\\
                        & \overset{(i)}{\geq} & \innerproduct{\begin{bmatrix}
                                u + A^\top y-(u'+A^\top y')\\
                                \nabla g(y)-Ax-(\nabla g(y')-Ax')
                        \end{bmatrix}}{\begin{bmatrix}
                                x-x'\\
                                -\alpha A(x-x')+y-y'
                        \end{bmatrix}}\\
                        & = & \innerproduct{u-u'}{x-x'}+\cancel{\innerproduct{A^\top (y-y')}{x-x'}}+\innerproduct{\nabla g(y)-\nabla g(y')}{y-y'}\\
                        && -\alpha \innerproduct{\nabla g(y)-\nabla g(y')}{A(x-x')}+\alpha \norm{A(x-x')}^2-\cancel{\innerproduct{A(x-x')}{y-y}}\\
                        & \overset{(ii)}{\geq} & \innerproduct{\nabla g(y)-\nabla g(y')}{y-y'}-\alpha \innerproduct{\nabla g(y)-\nabla g(y')}{A(x-x')}\\
                        && + \mf \norm{x-x'}^2+\alpha \norm{A(x-x')}^2\\
                        & \overset{(iii)}{\geq} &\mg\norm{y-y'}^2- \alpha \Lg\norm{A}\norm{x-x'}\norm{y-y'}\\
                        &&\mf \norm{x-x'}^2+\alpha \lambdamin{A^\top A} \norm{x-x'}\\
                        & = & \begin{bmatrix}
                            \norm{x-x'} &\norm{y-y'}
                        \end{bmatrix}
                        \underbrace{\begin{bmatrix}
                            \mf + \alpha \lambdamin{A^\top A} &-\frac{\alpha \Lg\norm{A}}{2}\\
                            -\frac{\alpha \Lg\norm{A}}{2} &\mg
                        \end{bmatrix}}_{\coloneqq M_\alpha} \begin{bmatrix}
                            \norm{x-x'}\\
                            \norm{y-y'}
                        \end{bmatrix}\\
                        & \geq &\lambdamin{M_{\alpha}} \norm{\omega-\omega'}^2,
                \end{IEEEeqnarray*}}
                where (i) follows from the Cauchy--Schwarz inequality, (ii) from (strong) convexity of $f$ (with $\mf=0$ possibly) and (iii) from strong convexity and smoothness of $g$. By Sylvester's criterion, the matrix $M_{\alpha}$ is positive definite (hence $\lambdamin{M_\alpha}>0$) for $\alpha$ small enough, in particular  if we choose 
                \begin{equation*}
                    \alpha < \frac{2 \lambdamin{A^\top A}\mg+2\sqrt{\lambda_{\min}^2(A^\top A)\mg^2+\Lg^2\norm{A}^2\mg\mf}}{\Lg^2\norm{A}^2}.
                \end{equation*}
                Using   that $ \lambdamax{  \Psi_{\alpha}} = \|\Psi_{\alpha}\| \leq 1+\alpha \norm{A}$ 
                yields  $\| v-v'\| \geq \frac{\lambdamin{M_\alpha}}{1+\alpha \|A\|} \|\omega-\omega'\|$, 
                which concludes the proof with $R_2 = \frac{1+\alpha \|A\|}{\lambdamin{M_\alpha}}>0$. 
            \hfill $\square$
            
        \subsection{\texorpdfstring{Proof of \Cref{lemma:IL2}}{chambollePockC3InverseLipschitz}}
                The idea is very similar to the proof of \Cref{lemma:IL1}. Under \Cthree{},  both functions $f$ and $g$ are smooth and the matrix $A$ is invertible. Define the following matrix,$\Psi_{\varepsilon}\in \R^{2n \times 2n}$, for $\varepsilon>0$ small enough to be chosen,
                \begin{equation*}
                    \Psi_{\varepsilon} = \begin{bmatrix}
                        I_n &\varepsilon A^\top \\
                        -\varepsilon A &I_n
                    \end{bmatrix}.
                \end{equation*}
                Then we have, for any $\omega=(x,y), \omega'=(x',y') \in \R^{2n}$,
                \allowdisplaybreaks{\begin{IEEEeqnarray*}{rCl}
                        \IEEEeqnarraymulticol{3}{l}{\norm{F(\omega)-F(\omega')}\norm{\Psi_{\varepsilon}}\norm{\omega - \omega'}}\\
                        & \geq & \norm{\begin{bmatrix}
                                \nabla f(x)- \nabla f(x')+A^\top (y-y')\\
                                \nabla g(y)-\nabla g(y')-A(x-x')
                        \end{bmatrix}}\norm{\begin{bmatrix}
                                x-x'+\varepsilon A^\top (y-y')\\
                                y-y'-\varepsilon A(x-x')
                        \end{bmatrix}}\\
                        & \overset{(i)}{\geq} & \innerproduct{\begin{bmatrix}
                                \nabla f(x)- \nabla f(x')+A^\top (y-y')\\
                                \nabla g(y)-\nabla g(y')-A(x-x')
                        \end{bmatrix}}{\begin{bmatrix}
                                x-x'+\varepsilon A^\top (y-y')\\
                                y-y'-\varepsilon A(x-x')
                        \end{bmatrix}}\\
                        & = & \innerproduct{\nabla f(x)-\nabla f(x')}{x-x'}+\varepsilon \innerproduct{\nabla f(x)-\nabla f(x')}{A^\top (y-y')}+\varepsilon \norm{A^\top (y-y')}^2\\
                        && +\innerproduct{\nabla g(y)-\nabla g(y')}{y-y'}-\varepsilon \innerproduct{\nabla g(y)-\nabla g(y')}{A(x-x')}+\varepsilon \norm{A(x-x')}^2\\
                        & \overset{(ii)}{\geq} & \frac{1}{\Lf}\norm{\nabla f(x)-\nabla f(x')}^2+\varepsilon \innerproduct{\nabla f(x)-\nabla f(x')}{A^\top (y-y')}+\varepsilon \norm{A^\top (y-y')}^2\\
                        && +\frac{1}{\Lg} \norm{\nabla g(y)-\nabla g(y')}^2-\varepsilon \innerproduct{\nabla g(y)-\nabla g(y')}{A(x-x')}+\varepsilon \norm{A(x-x')}^2\\
                        & \overset{(iii)}{\geq} & \left (  \frac{1}{\Lf}-\frac{\norm{A}\varepsilon \delta}{2}\right)\norm{\nabla f(x)-\nabla f(x')}^2 + \varepsilon \left ( \lambdamin{AA^\top }-\frac{\norm{A}}{2 \delta}\right )\norm{y-y'}^2\\
                        && +\left ( \frac{1}{\Lg}-\frac{\norm{A}\varepsilon \delta}{2} \right )\norm{\nabla g(y)-\nabla g(y')}^2 + \varepsilon \left ( \lambdamin{A^\top A}- \frac{\norm{A}}{2\delta}\right )\norm{x-x'}^2
                \end{IEEEeqnarray*}}
                where (i) follows from Cauchy-Schwarz, (ii) from cocoercivity of $\nabla f$ and $\nabla g$, (iii) from Young's inequality, for any $\delta>0$. In particular, if we choose $\delta > \frac{\norm{A}}{2\mu_A}$, $\varepsilon< \frac{2}{L\norm{A}\delta}$, $L= \max\{\Lf,\Lg\}$ we obtain $\|F(\omega)-F(\omega') \|\norm{\Psi_{\varepsilon}}\norm{\omega - \omega'} \geq C(\epsilon,\delta) \| \omega - \omega'\|^2$, where 
                \begin{equation*}
                    C(\varepsilon, \delta) = \varepsilon \left (\mu_A -\frac{\norm{A}}{2\delta}\right ) >0.
                \end{equation*}
               Using $\norm{\Psi_{\varepsilon}} \leq 1+\varepsilon \norm{A}$ concludes the proof, with  $R_3 = \frac{1+\varepsilon \norm{A}}{C(\varepsilon, \delta)}>0.$
            \hfill $\square$
            
        \subsection{\texorpdfstring{Proof of \Cref{thm:chambollePock}}{chambollePock}}
          We first recall  an additional result.
            \begin{lemma}
                \label{lemma:monotonyPreconditionedInverse}
                Let $F:\R^{n+m}\rightrightarrows \R^{n+m}$ be $\mu$-strongly monotone in $\innerproduct{\cdot}{\cdot}$, with $\mu \geq 0$, and $\P \succ 0$. Then $\Phi_{\tau,\sigma}^{-1}F$ is 
                strongly monotone in $\innerproduct{\cdot}{\cdot}_{\Phi_{\tau,\sigma}}$ with parameter $\frac{\mu }{\lambdamax{\Phi_{\tau,\sigma}}}$, i.e., for any $(\omega,u),(\omega',u') \in \gra(\Phi_{\tau,\sigma}^{-1} F)$, it holds that $\innerproduct{u-u'}{\omega-\omega'}_{\Phi_{\tau,\sigma}} \geq \frac{\mu}{\lambdamax{\P}}\|\omega - \omega'\|^2_\P$. 
            \end{lemma}
    
            \begin{proof}
                 Note that  $ (\omega,u) \in\gra(F)$ if and only if $(\omega, \Phi_{\tau,\sigma}^{-1}u) \in \gra(\Phi_{\tau,\sigma}^{-1}F)$. For any 
                 $(\omega,u),( \omega',u') \in \gra F$, we have 
                \begin{IEEEeqnarray*}{rCl}
                        \innerproduct{ \Phi_{\tau,\sigma}^{-1}u- \Phi_{\tau,\sigma}^{-1}u'}{\omega -\omega'}_{\Phi_{\tau,\sigma}} 
                        & = & \innerproduct{u-u'}{\omega-\omega'}\\
                        & \geq & \mu \norm{\omega-\omega'}^2\\
                        & \geq & \frac{\mu }{\lambdamax{\Phi_{\tau,\sigma}}} \norm{ \omega-\omega'}_{\Phi_{\tau,\sigma}}^2,
                    \end{IEEEeqnarray*}
                where the last inequality follows from  equivalence of norms.
            \end{proof}

            For any $\tau>0,\sigma>0$, the fixed points of \Cref{alg:chambollePock} are the vectors $(x^\star,\omega^\star)$ that satisfy 
            \begin{align*} 
            x^\star & = \operatorname{prox}_{\tau f} \left(x^\star - \tau A^\top y^\star \right)
            \\
            y^\star & = \operatorname{prox}_{\sigma g} \left( y^\star +\sigma A (2x^\star-x^\star) \right),
        \end{align*}
            or, equivalently, by definition of proximal operator, 
            \begin{align*} 
            x^\star + \tau \partial f(z^\star)   & \ni  x^\star - \tau A^\top y^\star 
            \\
            y^\star + \sigma \partial g & \ni y^\star +\sigma A x^\star, 
        \end{align*}
            equivalently, $\0 \in F(x^\star,y^\star)$, equivalently the solutions of \eqref{eq:bilinearSaddlePointProblem}.

            Furthermore, whenever $\sigma \tau \|A\|^2 <1$ (so that $\P \succ 0$), \Cref{alg:chambollePock} can be recast as the forward-backward iteration in \eqref{eq:precFB} (see \Cref{sec:algderivation1}), with $\Ff =\0$, $\Fb = F$. Since $\Ff = \0$, then the forward-step $\mathcal{F} = \operatorname{Id}$ satisfies trivially \Atwo{} in \Cref{prop:partialContractivityB} with $\Psi_\textnormal{f} = \0$, $\rho_\textnormal{f}=1$ (namely, $\F$ is nonexpansive in $\| \cdot\|_\P$). 
             It is left to show that $\mathcal B$ is contractive,  which is done in the following, for each of the conditions \Cone{}, \Ctwo{}, \Cthree. 
             
            \subsubsection{\texorpdfstring{Item (i): \Cref{alg:chambollePock} under \Cone{}}{chambollePockC1}}
    
                We start by an auxiliary result. Let $\tilde{\rho} := \min{\{\mf \tau, \mg \sigma, \kappa\}}$. We  prove that
                \begin{equation}
                    \label{eq:psdOfCase1}
                    \begin{bmatrix}
                        \mf I &0\\
                        0 & \mg I
                    \end{bmatrix}-\tilde{\rho} \Phi_{\tau, \sigma} = \begin{bmatrix}
                       \left (\mf - \tilde{\rho}\frac{1}{\tau}\right)I &\tilde{\rho}A^\top \\
                        \tilde{\rho}A &\left(\mg - \tilde{\rho}\frac{1}{\sigma}\right)I
                    \end{bmatrix} \succcurlyeq 0.
                \end{equation}
                If $A=\0$, this is obvious. If $\norm{A} \neq 0$, then either $\tilde{\rho}<\mf \tau$ or $\tilde{\rho}<\mg \sigma$, because $\kappa = \frac{\mf \tau}{1 + \sqrt{\tau \sigma \norm{A}^2}}< \mg \tau$ if $\mf \tau = \mg \sigma$. Let us assume without loss of generality that $\tilde{\rho}<\mf \tau$. Then, by Schur complement, \eqref{eq:psdOfCase1} is equivalent to
                \begin{equation*}
                    \left(\mg - \tilde{\rho} \frac{1}{\sigma}\right ) I - \tilde{\rho}^2A \left [\left (\mf - \tilde{\rho} \frac{1}{\tau}\right ) I \right]^{-1}A^\top  \succcurlyeq 0,
                \end{equation*}
                which is implied by (since $AA^\top  \succcurlyeq 0$)
                \begin{IEEEeqnarray*}{rC}
                     \left (\mf - \tilde{\rho} \frac{1}{\tau}\right ) \left ( \mg - \tilde{\rho}\frac{1}{\sigma} \right ) - \tilde{\rho}^2 \norm{A}^2 &\geq 0 \\
                     \Longleftrightarrow \quad \tilde{\rho}^2 ( 1- \tau \sigma \norm{A}^2) - \rho (\mf \tau + \mg \sigma) + \mf \mg \tau \sigma  &\geq 0.
                \end{IEEEeqnarray*}
                In turn, by solving the second-order inequality, we obtain that \eqref{eq:psdOfCase1} holds if $\tilde{\rho} \leq \kappa$ (note that $1-\tau \sigma \norm{A}^2>0$  and the determinant of the second order inequality is   $\left ( \mf \tau + \mg \sigma \right )^2 - 4 (1-\tau \sigma \norm{A}^2)(\mf \mg \tau \sigma)\geq (\mf \tau - \mg \sigma)^2 \geq 0 $),         
                which proves \eqref{eq:psdOfCase1}. We further note   that $\kappa>0$, because 
                $(\mf \tau - \mg\sigma)^2 + 4 \|A\|^2 \mf\mg \tau^2 \sigma^2 < (\mf \tau + \mg\sigma)^2$ whenever 
                $\tau\sigma \|A\|^2<1$.

                Now, for $\omega, \omega' \in \R^{n+m}$, let $v=(I+\Phi_{\tau, \sigma}^{-1}F)^{-1}\omega$ and $v'=(I+\Phi_{\tau, \sigma}^{-1}F)^{-1}\omega'$, which implies $\omega = v + \Phi_{\tau, \sigma}^{-1} z$ and $\omega' = v' + \Phi_{\tau, \sigma}^{-1} z'$ for some $z \in Fv$ and $z' \in Fv'$, by definition of inverse operator.
                Therefore, by strong convexity, definition of $F$, and letting $v = (v_x,v_y)$, $v' = (v_x',v_y')$, we have
                \begin{IEEEeqnarray*}{rCl}
                    \IEEEeqnarraymulticol{3}{l}{
                        \norm{\omega-\omega'}_{\Phi_{\tau, \sigma}} \norm{v-v'}_{\Phi_{\tau, \sigma}}
                    } \nonumber \\
                    & = & \norm{v+\Phi_{\tau, \sigma}^{-1}z-\left(v'+\Phi_{\tau, \sigma}^{-1}z'\right )}_{\Phi_{\tau, \sigma}}\norm{v-v'}_{\Phi_{\tau, \sigma}}\\
                    & \overset{(i)}{\geq} & \innerproduct{v+\Phi_{\tau, \sigma}^{-1}z-\left(v'+\Phi_{\tau, \sigma}^{-1}z'\right )}{v-v'}_{\Phi_{\tau, \sigma}}\\
                    & \overset{(ii)}{\geq} & \norm{v-v'}_{\Phi_{\tau, \sigma}}^2 + \mf \norm{v_x-v_x'}^2+\mg \norm{v_y-v_y'}^2\\
                    &=& (v-v')^\top \left ( \Phi_{\tau, \sigma}+ \begin{bmatrix}
                        \mf I &0\\
                        0 &\mg I
                    \end{bmatrix}\right) (v-v')\\
                    & \overset{(iii)}{\geq} & (v-v')^\top  \left ((1+\tilde{\rho})  \Phi_{\tau, \sigma} \right)(v-v')\\
                    & = & (1+\tilde{\rho}) \norm{v-v'}_{\Phi_{\tau, \sigma}}^2
                \end{IEEEeqnarray*}
                where (i) follows from Cauchy-Schwarz, (ii) from the definition of the weighted inner product and (iii) from \eqref{eq:psdOfCase1}. Dividing both sides by $(1+\tilde \rho)\|v-v' \|$, we conclude that $\mathcal B$ is contractive in $\|\cdot\|_\P$ with constant $\rho$ (i.e., \Aone{} in \Cref{prop:partialContractivityB},  
                holds with $\Psi_\textnormal{b} = \0$, $\gamma =0$, $\rho_b = \rho$). Then the contractivity result follows by   \Cref{prop:partialContractivityB} (we recall that \Atwo{} in \Cref{prop:partialContractivityB} holds with $\Psi_\textnormal{f} = \0$, $\rho_\textnormal{f}=1$).

                Concerning the optimal step sizes, first fix $ \tau \sigma \|A\|^2 (1+\alpha)^2 =1$, for some $\alpha \geq \epsilon$.  We now want to maximize $\kappa$ by choosing opportunely $\sigma$ and $\tau$, where we recall that
                \begin{align*}
                    \kappa = 		\left(   \frac{\mf \tau +\mg \sigma  -\sqrt{ (\mf \tau - \mg\sigma)^2 + 4 \|A\|^2 \mf\mg \tau^2 \sigma^2}}{2(1-\sigma\tau \|A\|^2)}\right) >0.
                \end{align*}
                Because $\sigma \tau = \frac{1}{\|A\|^2(1+\alpha)^2}$ is a constant independent of $\tau$ and $\sigma$, maximizing $\kappa$ is equivalent to maximizing
                \begin{align*}
                   ( \mf \tau +\mg \sigma) - \sqrt{( \mf \tau +\mg \sigma)^2 - K },
                \end{align*}
                where $K$ is a constant independent of $\sigma $ and $\tau$. 
                In turn, the latter expression is decreasing in $( \mf \tau +\mg \sigma)$. Minimizing $\mf \tau +\mg \sigma = \mf \tau +\frac{\mg}{\|A\|^2(1+\alpha)^2 \tau}$ is a convex optimization problem in $\tau$, whose solution    gives the step size $\tau = \frac{1}{(1+\alpha)\|A\|}\sqrt{\frac{\mg}{\mf}} $, and thus $\sigma = \frac{1}{(1+\alpha)\|A\|}\sqrt{\frac{\mf}{\mg}}$. For this values of the step sizes, we get 
                \begin{align}
                    \kappa = \frac{ 
                    2 \frac{\sqrt{\mf\mg}}{(1+\alpha)\|A\|} - \sqrt{\frac{4\mf\mg}{\|A\|^2(1+\alpha)^4 }} }{ 2\left(1- \frac{1}{(1+\alpha)^2} \right)} = \frac{\sqrt{\mf \mg}}{(\alpha+2)\|A\|} < \frac{\sqrt{\mf \mg}}{(\alpha+1)\|A\|} = \mf \tau = \mf g. \label{eq:kappaopt}
                \end{align}
                We conclude that the choice $\tau = \frac{1}{(1+\alpha)\|A\|}\sqrt{\frac{\mg}{\mf}} $,  $\sigma = \frac{1}{(1+\alpha)\|A\|}\sqrt{\frac{\mf}{\mg}}$ maximizes $\rho$, provided that $\tau \sigma \|A\|^2(1+\alpha)^2 =1$. Because the optimal rate is increasing in $\alpha$ (i.e., the optimal $\kappa$ in \eqref{eq:kappaopt} is decreasing in $\alpha$), we conclude that the optimal choice is $\alpha =\epsilon$, which is the desired result.

            \subsubsection{\texorpdfstring{Item (ii): \Cref{alg:chambollePock} under \Ctwo{}}{chambollePockC2}}\label{proof:algo1C2}
                
                    For all $\omega, \omega' \in \R^{n+m}$ 
                    and
                    $v\in  F_\textnormal{b}(\omega)$, $v' \in   F_\textnormal{b}(\omega')$
                    we have 
                    \begin{IEEEeqnarray*}{rCl}
                        \norm{\Phi_{\tau, \sigma}^{-1}v-\Phi_{\tau, \sigma}^{-1}v'}_{\Phi_{\tau, \sigma}}^2 & = & \innerproduct{\Phi_{\tau, \sigma}^{-1}(v-v')}{v-v'}\\
                        & \geq & \lambdamin{\Phi_{\tau, \sigma}^{-1}} \norm{v-v'}^2\\
                        & \overset{(i)}{\geq} & \frac{1}{\|\P\| R_2^2} \| \omega - \omega'\|^2
                        \\
                        & \geq & \frac{1}{\|\P\|^2 R_2^2} \| \omega - \omega'\|^2_\P,
                    \end{IEEEeqnarray*}
                    where in (i) \Cref{lemma:IL1} was used. The operator $\Phiinv F_\textnormal{b}$ is therefore inverse Lipschitz with constant $ S_2 = \frac{1}{\|\P\| R_2}$ with respect to $\|\cdot \|_
                    \P$. Furthermore, since $F_\textnormal{b}$ is monotone (in $\langle \cdot, \cdot \rangle $), then   $\Phi_{\tau, \sigma}^{-1}F_\textnormal{b}$ is monotone in $\innerproduct{\cdot}{\cdot}_{\Phi_{\tau, \sigma}}$ by
                    \Cref{lemma:monotonyPreconditionedInverse}. Therefore we can repeat the proof of \Cref{prop:PropertiesOfResolvent}, with the only caution of replacing the unweighted norm and inner product with $\|\cdot\|_\P$ and $\langle \cdot, \cdot \rangle _\P$, to  conclude that $\mathcal B  =(\operatorname{Id}+\Phiinv F\textnormal{b})$ is contractive   in $\| \cdot\|_\P$ with constant \begin{align}
                    \frac{\| \P \| R_2 }{\sqrt{(\| \P \| R_2)^2+1}} \leq \frac{ \zet R_2 }{\sqrt{(\zet R_2)^2+1}} = \rho,
                    \end{align}
                    where we used the bound $\|\P\| \leq \zet $ in \Cref{prop:propertiesPhi}. 
                    In other terms, \Aone{} in \Cref{prop:partialContractivityB} 
                    holds with $\Psi_\textnormal{b} = \0$, $\gamma =0$, $\rho_b = \rho$. Then, the contractivity result follows by \Cref{prop:partialContractivityB} (we recall that \Atwo{} in \Cref{prop:partialContractivityB} holds with $\Psi_\textnormal{f} = \0$, $\rho_\textnormal{f}=1$).

                    To find the best step sizes, we note that $\rho$ is decreasing in $\zet$. Thus, the best step sizes are found by minimizing $\zet$ in \eqref{eq:zet}, i.e., maximizing $\min\{\tau,\sigma \}$ subject to $\tau\sigma \|A\|^2(1+\epsilon)^2 \leq 1$.

            \subsubsection{\texorpdfstring{Item (iii): \Cref{alg:chambollePock} under \Cthree{}}{chambollePockC3}}
                 
                 %    We have $\forall \omega, \omega' \in \R^{2n}$ 
                 %    \begin{IEEEeqnarray*}{rCl}
                 %        \norm{\Phi_{\tau, \sigma}^{-1}F_\textnormal{b}(\omega)-\Phi_{\tau, \sigma}^{-1}F_\textnormal{b}(\omega)}_{\Phi_{\tau, \sigma}}^2 & = & \innerproduct{\Phi_{\tau, \sigma}^{-1}(F_\textnormal{b}(\omega)-F_\textnormal{b}(\omega)}{F_\textnormal{b}(\omega)-F_\textnormal{b}(\omega')}\\
                 %        & \geq & \lambdamin{\Phi_{\tau, \sigma}^{-1}} \norm{F_\textnormal{b}(\omega)-F_\textnormal{b}(\omega')}^2\\
                 %        & \overset{(i)}{\geq} & \frac{\lambdamin{\Phi_{\tau, \sigma}^{-1}}C^2(\varepsilon, \delta)}{(1+\varepsilon \norm{A})^2} \norm{\omega-\omega'}^2\\
                 %        & \geq & \frac{\lambdamin{\Phi_{\tau, \sigma}^{-1}}C^2(\varepsilon, \delta)}{\lambdamax{\Phi_{\tau, \sigma}}(1+\varepsilon\norm{A})^2} \norm{\omega-\omega'}_{\Phi_{\tau, \sigma}}^2
                 %    \end{IEEEeqnarray*}
                 %    where in (i) \Cref{lemma:IL2} was used. Since $\Phi_{\tau,\sigma}^{-1}F_\textnormal{b}$ is monotone in $\innerproduct{\cdot}{\cdot}_{\Phi_{\tau,\sigma}}$ and inverse Lipschitz, $G$ is a contraction with factor $\frac{\norm{\Phi_{\tau,\sigma}}(1+\varepsilon \norm{A})}{\sqrt{\norm{\Phi_{\tau,\sigma}}^2(1+\varepsilon \norm{A})^2+C^2(\varepsilon, \delta ) }}$ (see \Cref{prop:PropertiesOfResolvent}(ii)).
                 Identical to the proof in \Cref{proof:algo1C2}, with the only caution of replacing the constant $R_2$ from \Cref{lemma:IL1} with $R_3$ from \Cref{lemma:IL2}. 
                \hfill $\square$

        \subsection{\texorpdfstring{Proof of \Cref{lemma:semiContractionBackward}}{semiContractionBackward}}
            \label{sec:proofSemiContractionBackward}
             
                 Let $\tau \sigma \norm{A}^2 < 1$ and, for any $\omega, \omega' \in \R^{n+m}$, define $u = \J_{\Phi_{\tau, \sigma}^{-1}F_\textnormal{b}}(\omega)$ and $u' = \J_{\Phi_{\tau, \sigma}^{-1}F_\textnormal{b}}(\omega')$ which implies $u+ \Phi_{\tau, \sigma}^{-1}F_\textnormal{b}(u) \ni \omega$ and $u'+ \Phi_{\tau, \sigma}^{-1}F_\textnormal{b}(u') \ni \omega'$. Let $u = (u_x,u_y)$ and $u' = (u_x',u_y')$.  Then, for some $v_x\in \partial f(u_x)$ and $v_x' \in \partial f(u_x')$,  we get
                 \allowdisplaybreaks{\begin{IEEEeqnarray*}{rCl}
                    \IEEEeqnarraymulticol{3}{l}{
                     \norm{\omega-\omega'}_{\Phi_{\tau, \sigma}}^2} \nonumber\\
                     & = & \norm{u+ \Phi_{\tau, \sigma}^{-1}F_\textnormal{b}(u)-u' -  \Phi_{\tau, \sigma}^{-1}F_\textnormal{b}(u')}_{\Phi_{\tau, \sigma}}^2\\
                     & = & \norm{u-u'}_{\Phi_{\tau, \sigma}}^2+ \norm{\Phi_{\tau, \sigma}^{-1}(F_\textnormal{b}(u)-F_\textnormal{b}(u'))}_{\Phi_{\tau, \sigma}}^2 \nonumber  + 2\innerproduct{u-u'}{\Phi_{\tau, \sigma}^{-1}(F_\textnormal{b}(u)-F_\textnormal{b}(u'))}_{\Phi_{\tau, \sigma}}\\
                     & = & \norm{u-u'}_{\Phi_{\tau, \sigma}}^2+\innerproduct{\Phi_{\tau, \sigma}^{-1}(F_\textnormal{b}(u)-F_\textnormal{b}(u'))}{F_\textnormal{b}(u)-F_\textnormal{b}(u')} \nonumber  + 2\innerproduct{u-u'}{F_\textnormal{b}(u)-F_\textnormal{b}(u')}\\
                     & \overset{(i)}{\geq} & \norm{u-u'}_{\Phi_{\tau, \sigma}}^2 + \lambdamin{\Phi_{\tau, \sigma}^{-1}}\norm{\begin{bmatrix}
                         v_x+A^\top u_y -v_x'-A^\top u_y'\\
                         -A(u_x-u_x')
                     \end{bmatrix}}^2 \nonumber\\
                     && +2\innerproduct{u_x-u_x'}{v_x-v_x'}+\cancel{2\innerproduct{u_x-u_x'}{A^\top (u_y-u_y')}}-\cancel{2 \innerproduct{u_y-u_y'}{A(u_x-u_x')}}\\
                     & \geq & \norm{u-u'}_{\Phi_{\tau, \sigma}}^2 + \lambdamin{\Phi_{\tau, \sigma}^{-1}}\norm{A(u_x-u_x')}^2 \nonumber +2\innerproduct{u_x-u_x'}{v_x-v_x'}\\
                     & \overset{(ii)}{\geq} &  \norm{u-u'}_{\Phi_{\tau, \sigma}}^2 + \lambdamin{\Phi_{\tau, \sigma}^{-1}}\lambdamin{A^\top A}\norm{u_x-u_x'}^2 \nonumber +2\innerproduct{u_x-u_x'}{v_x-v_x'}\\
                      & \overset{(iii)}{\geq} &  \norm{u-u'}_{\Phi_{\tau, \sigma}}^2 + \lambdamin{\Phi_{\tau, \sigma}^{-1}}\lambdamin{A^\top A}\norm{u_x-u_x'}^2 \nonumber +2 \mf \norm{u_x-u_x'}^2\\
                     & = & \norm{u-u'}_{\Phi_{\tau, \sigma}}^2 +\left (\lambdamin{\Phi_{\tau, \sigma}^{-1}}\lambdamin{A^\top A} +2\mf\right )\norm{\begin{bmatrix}
                         u_x-u_x'\\
                         0
                     \end{bmatrix}}^2\\
                     & \overset{(iv)}{=} & \norm{u-u'}_{\Phi_{\tau, \sigma}}^2 + \norm{u-v}_{\Psi_b}^2\\
                     & = &  \norm{\J_{ \Phi_{\tau, \sigma}^{-1}F_\textnormal{b}}(\omega)-\J_{ \Phi_{\tau, \sigma}^{-1}F_\textnormal{b}}(\omega')}_{\Phi_{\tau, \sigma}}^2+\norm{\J_{\Phi_{\tau, \sigma}^{-1}F_\textnormal{b}}(\omega)-\J_{\Phi_{\tau, \sigma}^{-1}F_\textnormal{b}}(\omega')}_{\Psi_b}^2,
                \end{IEEEeqnarray*}}
                where in (i) the fact that $\Phi_{\tau, \sigma}^{-1}$ is invertible was used, (ii) holds since the smallest singular value is non-negative,  (iii) follows by (strong) convexity of $f$ and (iv) follows from the definition of $\Psi_b$.
            \hfill $\square$

        \subsection{\texorpdfstring{Proof of \Cref{lemma:semiContractionForward}}{semiContractionForward}}
            \label{sec:proofSemiContractionForward}
            We need the following additional result.
            \begin{lemma}
                \label{lemma:semiContractionForwardNonSmoothAdditional}
               Let $\xi_1 = \frac{1}{\sigma} -\tau \|A\|^2$. Then, for all $\omega =(x,y), \omega'=(x',y') \in \R^{n+m}$,   
                \begin{equation*}
                    \innerproduct{F_\textnormal{f}(\omega) - F_\textnormal{f}(\omega'))}{\omega-\omega'} \geq \frac{\Lg \mg}{\Lg + \mg} \norm{\begin{bmatrix}
                        0\\
                        y-y'
                    \end{bmatrix}}^2+\frac{ \xi_1}{\Lg+\mg} \norm{F_\textnormal{f}(\omega)-F_\textnormal{f}(\omega')}_{\Phi_{\tau, \sigma}^{-1}}^2.
                \end{equation*}
            \end{lemma}
            \begin{proof}
                Let $h(\omega) = g(y)-\frac{\mg}{2}\norm{y}^2$. Since  $h$ is $(\Lg-\mg)$-smooth, we have the cocoercivity bound \cite[Th.~18.15(v)]{bauschkeConvexAnalysisMonotone2017}
                \begin{equation*}
                    \langle \ng(y) -\ng(y'), y- y' \rangle = \innerproduct{F_\textnormal{f}(\omega)-F_\textnormal{f}(\omega')}{\omega - \omega'} \geq \frac{1}{\Lg-\mg}\norm{\nabla h(\omega)-\nabla h(\omega')}^2.
                \end{equation*}
                Rearranging the terms gives
                \begin{align*}                \innerproduct{F_\textnormal{f}(\omega) - F_\textnormal{f}(\omega'))}{\omega-\omega'} \geq \frac{\Lg \mg}{\Lg + \mg} \norm{\begin{bmatrix}
                        0\\
                        y-y'
                    \end{bmatrix}}^2+\frac{1}{\Lg+\mg} \norm{F_\textnormal{f}(\omega)-F_\textnormal{f}(\omega')}^2.
                \end{align*}
                The conclusion follows because $F_\textnormal{f}= (\0,\nabla g)$ and, by computing the block inverse of $\Phi$, the right lower block of $\Phiinv$ is $(\frac{1}{\sigma}I - \tau AA^\top)$, so that 
                \begin{align*}
                    \norm{F_\textnormal{f}(\omega)-F_\textnormal{f}(\omega')} \leq \xi_1 \norm{F_\textnormal{f}(\omega)-F_\textnormal{f}(\omega')}_\Phiinv.
                \end{align*}
            \end{proof}
            We are ready to prove \Cref{lemma:semiContractionForward}.
            \begin{proof}
                For any $\omega, \omega' \in \R^{n+m}$, with $\omega = \begin{bmatrix}
                    x\\
                    y
                \end{bmatrix}$ and $\omega' = \begin{bmatrix}
                    x'\\
                    y'
                \end{bmatrix}$, we have
                \begin{IEEEeqnarray*}{rCl}
                    \IEEEeqnarraymulticol{3}{l}{
                     \norm{\left(I-\Phi_{\tau, \sigma}^{-1}F_\textnormal{f}\right)(\omega)-\left(I- \Phi_{\tau, \sigma}^{-1}F_\textnormal{f}\right)(\omega')}_{\Phi_{\tau, \sigma}}^2} \nonumber\\
                     & = & \norm{\omega - \omega'}_{\Phi_{\tau, \sigma}}^2 + \norm{\Phi_{\tau, \sigma}^{-1}(F_\textnormal{f}(\omega)-F_\textnormal{f}(\omega'))}_{\Phi_{\tau, \sigma}}^2  \nonumber- 2\innerproduct{\omega-\omega'}{F_\textnormal{f}(\omega)-F_\textnormal{f}(\omega')}\\
                     & \leq & \norm{\omega - \omega'}_{\Phi_{\tau, \sigma}}^2-\frac{2L_g \mg}{\Lg+\mg}\norm{\begin{bmatrix}
                         0\\
                         y-y'
                     \end{bmatrix}}^2+\left (1-\frac{2 \xi_1}{\Lg+\mg} \right ) \norm{F_\textnormal{f}(\omega)-F_\textnormal{f}(\omega')}_{\Phi_{\tau, \sigma}^{-1}},
                \end{IEEEeqnarray*}
                where in the last line \Cref{lemma:semiContractionForwardNonSmoothAdditional} and the fact that $\norm{\Phi_{\tau, \sigma}^{-1}(F_\textnormal{f}(\omega)-F_\textnormal{f}(\omega'))}_{\Phi_{\tau, \sigma}}^2 = \norm{F_\textnormal{f}(\omega)-F_\textnormal{f}(\omega')}_{\Phi_{\tau, \sigma}^{-1}}^2$ was used. We have to distinguish between two cases:
                \begin{itemize}
                    \item $1-\frac{2\xi_1 }{\Lg+\mg}\leq 0$:\\
                    In this case we can use the bound $\norm{F_\textnormal{f}(\omega)-F_\textnormal{f}(\omega')}_{\Phi_{\tau, \sigma}^{-1}}^2\geq \frac{\mg^2}{\xi_2}\norm{\begin{bmatrix}
                        0\\
                        y-y'
                    \end{bmatrix}}^2$, with $\xi_ 2= \frac{1}{\sigma}I - \tau \mu_A \geq 
                \lambda_{\max}(\frac{1}{\sigma}I - \tau A A^\top) $, which follows again by definition of $F_\textnormal{f}$ and block inverse of $\P$. This implies
                    \begin{IEEEeqnarray*}{rCl}
                        \IEEEeqnarraymulticol{3}{l}{
                         \norm{\left(I-\Phi_{\tau, \sigma}^{-1}F_\textnormal{f}\right)(\omega)-\left(I- \Phi_{\tau, \sigma}^{-1}F_\textnormal{f}\right)(\omega')}_{\Phi_{\tau, \sigma}}^2} \nonumber\\
                         & \leq & \norm{\omega - \omega'}_{\Phi_{\tau, \sigma}}^2 -\left ( \frac{2 \Lg \mg}{\Lg +\mg}+\frac{2\mg^2 \xi_1} { \xi_2(\Lg+\mg)}-\frac{\mg^2}{\xi_2}\right )\norm{\begin{bmatrix}
                             0\\
                             y-y'
                         \end{bmatrix}}^2.\nonumber
                    \end{IEEEeqnarray*}
                    \item $1-\frac{2\xi_1}{\Lg+\mg}> 0$:\\
                   Using $\norm{F_\textnormal{f}(\omega)-F_\textnormal{f}(\omega')}_{\Phi_{\tau, \sigma}^{-1}}^2 \leq \frac{\Lg^2}{\xi_1} \norm{\begin{bmatrix}
                        0\\
                        y-y'
                    \end{bmatrix}}^2$ we get
                    \begin{align}
                         & \hphantom{{}={}} \norm{\left(I-\Phi_{\tau, \sigma}^{-1}F_\textnormal{f}\right)(\omega)-\left(I- \Phi_{\tau, \sigma}^{-1}F_\textnormal{f}\right)(\omega')}_{\Phi_{\tau, \sigma}}^2 \nonumber
                         \\
                         & \leq  \norm{\omega - \omega'}_{\Phi_{\tau, \sigma}}^2 -\left ( \frac{2 \Lg \mg}{\Lg +\mg}+\frac{2L_g^2}{\Lg+\mg}-\frac{\Lg^2}{\xi_1}\right )\norm{\begin{bmatrix}
                             0\\
                             y-y'
                         \end{bmatrix}}^2.\label{eq:usefulbound}
                    \end{align}
                \end{itemize}
               The last addend is negative only if $\xi_1> \frac{\Lg}{2}$, or equivalently \begin{align*}\sigma < \frac{2}{2\tau \|A\|^2+\Lg}.\end{align*}
            
                 % and using \Cref{prop:propertiesPhi} we get the following bound on the step sizes
                % \begin{equation*}
                %     \max \{\tau, \sigma\}<\frac{2}{\Lg+2 \norm{A}}
                % \end{equation*}
                % which also implies $\tau \sigma \norm{A}^2 < 1$. 
                Therefore, if the latter bounds on the step sizes holds, the forward step fulfills
                \begin{equation*}
                    \norm{\left(I-\Phi_{\tau, \sigma}^{-1}F_\textnormal{f}\right)(\omega)-\left(I- \Phi_{\tau, \sigma}^{-1}F_\textnormal{f}\right)(\omega')}_{\Phi_{\tau, \sigma}}^2 \leq \norm{\omega - \omega'}_{\Phi_{\tau, \sigma}}^2-\norm{\omega - \omega'}_{\Psi_{\textnormal f}}^2,
                \end{equation*}
                where ${\Psi_{\textnormal f}} = \diag (\0, \gamma_y I_m)$ and 
                % \begin{equation*}
                %     \rho_f^2 = \begin{cases}
                %         \frac{2 \Lg \mg}{\Lg +\mg}+\frac{2\mg^2}{\kappa (\Phi_{\tau, \sigma})(\Lg+\mg)}-\frac{\mg^2}{\lambdamax{\Phi_{\tau, \sigma}}} & \frac{\Lg + \mg}{2} \leq \lambdamin{\Phi_{\tau, \sigma}}\\
                %          \frac{2 \Lg \mg}{\Lg +\mg}+\frac{2L_g^2}{\Lg+\mg}-\frac{\Lg^2}{\lambdamin{\Phi_{\tau, \sigma}}} & \frac{\Lg + \mg}{2} > \lambdamin{\Phi_{\tau, \sigma}}
                %     \end{cases}
                % \end{equation*}
                \begin{equation*}
        	0< \gamma_y = \begin{cases}
        			\frac{2 \Lg \mg}{\Lg +\mg}+\frac{\mg^2(2\xi_1-\Lg-\mg)}{\xi_2(\Lg+\mg)} & \quad \text {if } \  \xi_1 \geq \frac{\Lg + \mg}{2} 
        			\\
        			\frac{2 \Lg \mg}{\Lg +\mg}-\frac{\Lg^2(\Lg +\mg-2\xi_1)}{(\Lg+\mg)\xi_1} & \quad \text {if } \ \xi_1 <  \frac{\Lg + \mg}{2},
        		\end{cases}
        	\end{equation*} 
                which concludes the proof.
            \end{proof}
            
        \subsection{\texorpdfstring{Proof of \Cref{thm:semiSmoothAlgorithm}}{semiSmooth}}
         As in the proof of \Cref{thm:chambollePock}, for any $\tau, \sigma >0$, the fixed points of \Cref{alg:semiSmoothPrimalDual} coincide with the saddle-points of \eqref{eq:bilinearSaddlePointProblem}. 
         Furthermore, whenever $\sigma \tau \|A\|^2 <1$ (which is ensured by the assumptions on the step sizes in \Cref{thm:semiSmoothAlgorithm}), $\P \succ 0$ and \Cref{alg:semiSmoothPrimalDual} can be recast as the forward-backward iteration in \eqref{eq:precFB}, with $F_\textnormal{f}(\omega) = (\0,\ng(y))$, $F_\textnormal{b}(\omega) =(\nf(x)+A^\top y, -Ax)$.

             \subsubsection{\texorpdfstring{Item(i): \Cref{alg:semiSmoothPrimalDual} under \Cone{} or \Ctwo}{semiSmoothC1}}

             Under either \Cone{} or \Ctwo{}, $\mathcal B$ satisfies \Aone{} in \Cref{prop:partialContractivity} with $\Phi_\textnormal{b}$ as in \Cref{lemma:semiContractionBackward}, and $\mathcal{F}$ satisfies \Atwo{} in \Cref{prop:partialContractivity} with $\Phi_\textnormal{f}$ as in \Cref{lemma:semiContractionForward}. Furthermore, 
             \begin{align*}
                 \Psi_\textnormal{b}+\Psi_\textnormal{f} = \diag(\gamma_x I_n, \gamma_y I_m) & \geq \frac{\min \{\gamma_x,\gamma_y\}}{\| \P \| + \gamma_x} \Psi_{\tau, \gamma} 
                \\
                & \geq \frac{\min \{\gamma_x,\gamma_y\}}{ \zet + \gamma_x} (\P+ \Psi_\textnormal{b}),
             \end{align*}
             and the conclusion readily follows by \Cref{prop:partialContractivity}.

            \subsubsection{\texorpdfstring{Item (ii): \Cref{alg:semiSmoothPrimalDual} under \Cthree}{semiSmoothC3}}
             
                     First, we  show that the backward step $\J_{\Phi_{\tau, \sigma}^{-1}F_\textnormal{b}}$ is contractive. The proof is analogous  to \Cref{lemma:IL2} and \Cref{thm:chambollePock}(iii). Define the following matrix $\Psi_{\varepsilon}\in \R^{2n\times 2n}$, for $\varepsilon>0$ small enough to be chosen,
                    \begin{equation*}
                        \Psi_{\varepsilon} = \begin{bmatrix}
                            I_n &\varepsilon A^\top \\
                            -\varepsilon A &I_n
                        \end{bmatrix}.
                    \end{equation*}
                    Then we have, for $\omega =(x,y)$, $\omega' = (x',y') \in \R^{2n}$,
                    \allowdisplaybreaks{\begin{IEEEeqnarray*}{rCl}
                            \IEEEeqnarraymulticol{3}{l}{\norm{\Fb(\omega)-\Fb(\omega')}\norm{\Psi_{\varepsilon}}\norm{\omega - \omega'}}\\
                            & \geq & \norm{\begin{bmatrix}
                                    \nabla f(x)- \nabla f(x')+A^\top (y-y')\\
                                    -A(x-x')
                            \end{bmatrix}}\norm{\begin{bmatrix}
                                    x-x'+\varepsilon A^\top (y-y')\\
                                    y-y'-\varepsilon A(x-x')
                            \end{bmatrix}}\\
                            & \overset{(i)}{\geq} & \innerproduct{\begin{bmatrix}
                                    \nabla f(x)- \nabla f(x')+A^\top (y-y')\\
                                    -A(x-x')
                            \end{bmatrix}}{\begin{bmatrix}
                                    x-x'+\varepsilon A^\top (y-y')\\
                                    y-y'-\varepsilon A(x-x')
                            \end{bmatrix}}\\
                            & = & \innerproduct{\nabla f(x)-\nabla f(x')}{x-x'}+\varepsilon \innerproduct{\nabla f(x)-\nabla f(x')}{A^\top (y-y')} +\varepsilon \norm{A^\top (y-y')}^2 +\varepsilon \norm{A(x-x')}^2\\
                            & \overset{(ii)}{\geq} & \frac{1}{\Lf}\norm{\nabla f(x)-\nabla f(x')}^2+\varepsilon \innerproduct{\nabla f(x)-\nabla f(x')}{A^\top (y-y')} +\varepsilon \norm{A^\top (y-y')}^2+\varepsilon \norm{A(x-x')}^2\\
                            & \overset{(iii)}{\geq} & \left (  \frac{1}{\Lf}-\frac{\norm{A}\varepsilon \delta}{2}\right)\norm{\nabla f(x)-\nabla f(x')}^2 + \varepsilon \left ( \lambdamin{AA^\top }-\frac{\norm{A}}{2 \delta}\right )\norm{y-y'}^2\\
                            && + \varepsilon \left ( \lambdamin{A^\top A}- \frac{\norm{A}}{2\delta}\right )\norm{x-x'}^2\\
                            & \overset{(iv)}{\geq} & C(\varepsilon, \delta)\norm{\begin{bmatrix}
                                    x-x'\\
                                    y-y'
                            \end{bmatrix}}^2,
                    \end{IEEEeqnarray*}}
                    where (i) follows from Cauchy-Schwarz, (ii) from cocoercivity of $\nabla f$, (iii) from Young's inequality and (iv) holds if $\delta > \frac{\norm{A}\mu_A}{2}$, $\varepsilon< \frac{2}{\Lf \norm{A}\delta}$ and 
                    \begin{equation*}
                        C(\varepsilon, \delta) = \varepsilon \left (\mu_A -\frac{\norm{A}}{2\delta}\right ) >0.
                    \end{equation*}
                  
                    Then the proof follows as for   \Cref{proof:algo1C2}. In particular, using $\norm{\Psi_{\varepsilon}} \leq 1+\varepsilon \norm{A}$ we have that $\Fb$ is inverse Lipschitz with constant $\frac{1}{R_3'}$, $R_3' \coloneqq\frac{1+\varepsilon \norm{A}}{C(\varepsilon, \delta)}$, thus $\Phi_{\tau,\sigma}^{-1}\Fb$ is inverse Lipschitz in $\norm{\cdot}_{\Phi_{\tau,\sigma}}$ with constant $\frac{R_3'}{\norm{\Phi_{\tau,\sigma}}}$. Since $\Phi_{\tau,\sigma}^{-1}\Fb$ is maximally monotone in $\norm{\cdot}_{\Phi_{\tau,\sigma}}$, the backward step $\mathcal B$ is contractive in $\|\cdot \|_\P$ with parameter$\frac{\zet R_3'}{\sqrt{(R_3'\zet)^2+1}}$ (see \Cref{prop:PropertiesOfResolvent}(ii)), namely \Aone{} in \Cref{prop:partialContractivityB} holds with $\rho_\textnormal{b} =\rho$, $\Phi_\textnormal{b} =\0$.
                   
                   \par It is left to show, that the forward step $\mathcal F = I-\Phi_{\tau, \sigma}^{-1}F_\textnormal{f}$ is nonexpansive, namely that \Atwo{} in \Cref{prop:partialContractivityB} holds with $\rho_\textnormal{f}=1$, $\Psi_\textnormal{f} = \0$. If $\Lg =0$, this holds  trivially (we recall that $F_\textnormal{f} = (\0,\ng)$). If $\Lg \neq 0$, then this holds whenever $\sigma < \frac{2}{\Lg+2\|A\|}$ by  \eqref{eq:usefulbound}.
                   Then, the contractivity result in \Cref{thm:semiSmoothAlgorithm}(ii) follows by \Cref{prop:partialContractivity}.

                   To find the best step sizes, we note that $\rho$ is decreasing in $\zet$. Thus, the best step sizes are found by minimizing $\zet$ in \eqref{eq:zet}, i.e., maximizing $\min\{\tau,\sigma \}$ subject to $ \tau \sigma \|A\|^2+ \frac{\Lg}{2} \sigma \leq 1$ and $\tau \sigma \|A\|^2(1+\epsilon)^2\leq 1$. Clearly $\sigma= \tau$ at the solution. Hence the best step size $\tau = \sigma$ is obtained as the maximum value that satisfies $  \sigma^2 \|A\|^2+ \frac{\Lg}{2} \sigma \leq 1$ and $ \sigma^2 \|A\|^2(1+\epsilon)^2\leq 1$, or equivalently $\sigma \leq  \frac{\sqrt{\Lg^2+16\|A\|^2}-\Lg}{4\|A\|^2}$ and $ \sigma \leq  \frac{1}{(1+\epsilon) \|A\|} $, which is the claim. \hfill $\square$

        \subsection{Proof of \Cref{lemma:PDGforward}}

        First, let us assume $\Lf, \Lg \neq 0$. Let us recall that, as in the proof of \Cref{lemma:semiContractionForward}, for any $\omega= (x,y),\omega'=(x',y')\in\R^{n+m}$,
			\begin{align*}
				& \hphantom{{}\geq{}} \langle \w - \w', F(\w)- F(\w') 
				\rangle  
				\\& \geq \frac{\Lf \mf }{\Lf +\mf} \| x- x'\|^2 + \frac{1}{\Lf+\mf} \| \nf(x) -\nf(x')\| ^2 + \frac{\Lg \mg }{\Lg +\mg} \| y- y'\|^2 + \frac{1}{\Lg+\mg} \| \ng(y) -\ng(y')\| ^2.
			\end{align*}
        
      Therefore,
      \begin{align*}
      & \hphantom{{}\geq{}} 	\| \F (\w) - \F(\w')\|_\P
      \\      
      & \leq \| \w{}- \w ' \|_{\P}- 2\frac{\Lf \mf }{\Lf +\mf} \| x- x'\|^2  - 2  \frac{\Lg \mg }{\Lg +\mg} \| y- y'\|^2  -2  \frac{1}{\Lf+\mf} \| \nf(x) -\nf(x')\| ^2
       \\ 
      & \qquad - 2 \frac{1}{\Lg+\mg} \| \ng(y) -\ng(y')\| ^2  + \| \Ff (\w) - \Ff(\w' ) \| ^2_{\Phiinv}.
       \end{align*}
       
       To bound the last three addends, we look for the smallest (possibly negative) constant $\alpha_x$, $\alpha_y$ , such that
       \begin{align}
       	\Phiinv - \diag \left ( 2 \textstyle  \frac{1}{\Lf+\mf}I_n,  2 \frac{1}{\Lg+\mg} I_m \right ) \preccurlyeq \diag(\alpha_x I_n, \alpha_y I_m)
       \end{align}
       (since $\Phiinv \succ 0$, this implies $2 \textstyle  \frac{1}{\Lf+\mf} + \alpha_x > 0$, $2 \frac{1}{\Lg+\mg} +\alpha_y >0$).
       By double application of the Schur complement, the above condition is equivalent to 
       \begin{align}\label{eq:useful4}
       	\P - \diag \left (  \textstyle  \frac{\Lf+\mf }{2+\alpha_x(\Lf+\mf)}I_n,  \frac{\Lg+\mg}{2+\alpha_y(\Lg+\mg)} I_m \right )   \succcurlyeq 0.
       \end{align}
       By computing the block-matrix determinant in \eqref{eq:useful4}, we obtain that a sufficient condition for this to hold is that, for any $\nu>0$,
       \begin{align*}
       	 \frac{1}{\tau} -  \textstyle  \frac{\Lf+\mf }{2+\alpha_x(\Lf+\mf)} & \geq \|A\| \nu
       	 \\
       	 \frac{1}{\sigma } - \textstyle  \frac{\Lg+\mg }{2+\alpha_y(\Lg+\mg)} &  \geq \|A\|\frac{1}{\nu}.
       \end{align*}
      	Solving for $\alpha_x$ and $\alpha_y$ (noting that $1-\tau \|A\| >0$), we obtain the smallest values
      	\begin{align*}
      		\alpha_x  = \frac{  \tau }{1- \tau \|A\|\nu}  - \frac{2}{\Lf+\mf}, \qquad
      		\alpha_y  = \frac{  \sigma }{1- \sigma \|A\|\textstyle\frac{1}{\nu}}  - \frac{2}{\Lg+\mg},
      	\end{align*}
      	and get 
      	      \begin{align*}
      		& \hphantom{{}\geq{}} 	\| \F (\w) - \F(\w')\|_\P
      		\\      
      		& \leq \| \w{}- \w ' \|_{\P}- 2\frac{\Lf \mf }{\Lf +\mf} \| x- x'\|^2  - 2  \frac{\Lg \mg }{\Lg +\mg} \| y- y'\|^2  +\alpha_x  \| \nf(x) -\nf(x')\| ^2
      		+\alpha_y \| \ng(y) -\ng(y')\| ^2 .
      	\end{align*}
      	Note that $\alpha_x$, $\alpha_y$  might be positive or negative, depending on the step sizes $\tau$, $\sigma$.
      In particular, if $ \tau \leq \frac {2}{\Lf+\mf+2\|A\|\nu }$, then $\alpha_x$ is nonpositive, and we can use the bound $ \| \nf(x) - \nf(x') \| \geq \mf \|x - x '\|$, derived from strong convexity of $f$ (with $\mf$ possibly zero); and if $\tau  > \frac {2}{\Lf+\mf+2\|A\|\nu} $, then $\alpha_x$ is positive, and we use the bound $\| \nf(x) - \nf(x') \| \leq \Lf \| x - x'\| $, and similarly for $\alpha_y$. Finally, we get
      \begin{align}
	\| \F (\w) - \F(\w')\|_\P
      	& \leq \| \w{}- \w ' \|_{\P} - \beta_x \|x-x'\|^2 - \beta_y \|y-y'\|^2,
      \end{align}
      where 
      \begin{align*}
      	\beta_x & = 
     \begin{cases}
     	  2 \frac{ \Lf \mf }{\Lf+\mf} - \alpha_x \mf^2 & \quad \text {if  } \tau \leq \frac {2}{\Lf+\mf+2\|A\|\nu}
     	\\
     	2 \frac{ \Lf \mf }{\Lf+\mf} - \alpha_x \Lf^2 & \quad \text {if  } \tau > \frac {2}{\Lf+\mf+2\|A\|\nu}
     \end{cases}
     \\
         	\beta_y & = 
   \begin{cases}
   	2 \frac{ \Lg \mg }{\Lg+\mg} - \alpha_y \mg^2 & \quad \text {if  } \sigma \leq \frac {2}{\Lg+\mg+2\|A\|\textstyle\frac{1}{\nu}}
   	\\
   	2 \frac{ \Lg \mg }{\Lg+\mg} - \alpha_y \Lg^2 & \quad \text {if  } \sigma > \frac {2}{\Lg+\mg+2\|A\|\textstyle\frac{1}{\nu}},
   \end{cases}  
      \end{align*}
      and it is easily verified that $\beta_x, \beta_y > 0 $ whenever $\mf,\mg >0$ and $\tau <  \frac{2}{\Lf + 2\|A\|\nu} $, $\sigma < \frac{2}{\Lg + 2\|A\|\textstyle\frac{1}{\nu}}$ (and $\beta_x, \beta_y = 0 $ if $\tau = \frac{2}{\Lf + 2\|A\|\nu} $, $\sigma = \frac{2}{\Lg + 2\|A\|\textstyle\frac{1}{\nu}}$).
      
      Finally, let us note that the statement still holds with $\beta_x = 0$ whenever $\Lf = 0$, and with $\beta_y = 0 $ whenever $\Lg = 0$ (repeat the same proof by omitting the terms corresponding to the zero Lipschitz constants).  \hfill $\square$

        \subsection{\texorpdfstring{Proof of \Cref{thm:smoothAlgorithm}}{smooth}}

            For any $\tau, \sigma>0$, the fixed points of \Cref{alg:smoothForwardBackward} coincide with the saddle points of \eqref{eq:bilinearSaddlePointProblem}, as in the proof of \Cref{thm:chambollePock}. 
             Furthermore, whenever $\sigma \tau \|A\|^2 <1$ (which is ensured by the assumptions on the step sizes in \Cref{thm:smoothAlgorithm}),    $\P \succ 0$ and \Cref{alg:smoothForwardBackward} can be recast as the forward-backward iteration in \eqref{eq:precFB}, with $F_\textnormal{f} = (\nf,\ng)$, $F_\textnormal{b} = (A^\top y, -Ax)$.

            \subsubsection{\texorpdfstring{ Item(i): \Cref{alg:smoothForwardBackward} under \Cone{}}{smoothC1}}

             Under \Cone{}, the forward step $\mathcal F$ is contractive, as a direct consequence of \Cref{lemma:PDGforward}.   In particular, its  contractivity rate $\rho>0$ in $\| \cdot\|_\P$ can be found by solving
      \begin{align*}
         \P - \diag( \beta_x I_n, \beta_y I_m) \preccurlyeq \rho \Phi,
      \end{align*}
      which, via a Schur complement argument,  can be relaxed as
      \begin{align*}
        \beta_x - \frac{1}{\tau}   (1-\rho) \geq (1-\rho) \|A\| \nu, \qquad 
         \beta_y - (1- \rho )\frac{1}{\sigma}   \geq (1-\rho)\|A\| \textstyle \frac{1}{\nu },
      \end{align*}
      equivalently 
      \begin{align*}
      \rho \geq 1 - \min \left\{ \frac{\beta_x \tau }{1+\tau\|A\|\nu},  \frac{\beta_y \sigma }{1+\sigma\|A\|\textstyle \frac{1}{\nu }} \right \}.
      	      \end{align*}
    Hence, \Atwo{} in \Cref{prop:partialContractivityB} holds with $\Psi_\textnormal{f} =\0$ and $\rho_\textnormal{f} = \rho$. 

    It is only left to show that the backward step $\mathcal B$ is nonexpansive, namely that \Aone{} in \Cref{prop:partialContractivityB} holds with $\Psi_\textnormal{b} =\0$ and $\rho_\textnormal{b} = 1$. This is a consequence of the fact that the operator $\Phiinv F_\textnormal{b}$ is monotone in $\langle \cdot, \cdot \rangle_\Phi$ (since $F_\textnormal{b}$ is monotone in $\langle \cdot, \cdot \rangle$ and by \Cref{lemma:monotonyPreconditionedInverse}), and hence its resolvent $\mathcal B$ is nonexpansive in $\| \cdot \|_\P $ \cite[Prop.~23.8]{bauschkeConvexAnalysisMonotone2017}. Then the contractivity result follows by \Cref{prop:partialContractivityB}. 

   %  Concerning the optimal step sizes, for $\tau \leq \frac {2}{\Lf+\mf+2\|A\|\nu}$, we have
   %  $
   %      \frac{\beta_x \tau}{1 + \tau \|A\|\nu} = 2\mf \frac{\tau}{1 + \tau \|A\|\nu} - \mf ^2 \frac{\tau^2 }{(1 + \tau \|A\|\nu)(1 - \tau \|A\|\nu)}
   % $. Differentiating  this expression with respect to $\tau$, we get 
   % $
   %         \frac{2\mf }{(1-\tau \|A\|\nu)^2(1+\tau \|A\|\nu)^2} ( 1 - \tau \|A\|\nu  -  \tau \mf)
   %        $ 
   %  which is $>0$ for all $\tau < \frac{2}{\Lf+ \mf +2\|A\|\nu} \leq \frac{2}{2\mf + 2 \|A\|\nu} = \frac{1}{\mf +\|A\|\nu}$, meaning that $\frac{\beta_x \tau}{1 + \tau \|A\|\nu}$ is increasing in $\tau$. Similarly, for $\tau > \frac {2}{\Lf+\mf+2\|A\|\nu} \geq \frac {1}{\Lf+\|A\|\nu}$, we have that $\frac{\beta_x \tau}{1 + \tau \|A\|\nu}$ is decreasing. We conclude that, for fixed $\nu$, 
   %  $\frac{\beta_x \tau}{1 + \tau \|A\|\nu}$ is maximized for $\tau = \frac{2}{\Lf+\mf+2\|A\| \nu}$. Analogously,  $\frac{\beta_y \sigma}{1 + \sigma \|A\|\frac{1}{\nu}}$ is maximized for $\sigma = \frac{2}{\Lg+\mg+2\|A\| \frac{1}{\nu}}$. For these step sizes, the term $\frac{\beta_x \tau}{1 + \tau \|A\|\nu}$ is decreasing in $\nu$, while $\frac{\beta_y \sigma}{1 + \sigma \|A\|\frac{1}{\nu}}$ is increasing in $\nu$; we conclude that their minimum is maximized by taking $\nu $ such that  $\frac{\beta_x \tau}{1 + \tau \|A\|\nu} = \frac{\beta_y \sigma}{1 + \sigma \|A\|\frac{1}{\nu}}$. Solving for $\nu$ yields the desired result. 
    
            \subsubsection{\texorpdfstring{Item(ii):  \Cref{alg:smoothForwardBackward} under \Ctwo{}}{smoothC2}}

            By specializing \Cref{lemma:semiContractionBackward} to the case $f=0$, we obtain that \Aone{} in \Cref{prop:partialContractivity} holds with $\Psi_\textnormal{b} = \diag \left(\frac{\mA}{\zet}I_n,\0 \right)$. 

            On the other hand, \Atwo{} holds with $\Psi_\textnormal{f}$ as in  \Cref{lemma:PDGforward} (where $\beta_x = 0$ in general). 
            
            Finally, 
            \begin{align*}
                \Psi_\textnormal{b}+ \Psi_\textnormal{f} = \diag \left(\frac{\mA}{\zet}I_n, \beta_y I_m \right) & \succeq \frac{\min \left\{ \frac{\mA}{\zet}, \beta_y \right\} }{\| \Psi_\textnormal{b} +\P \|  } (\Psi_\textnormal{b} +\P)
                \\
                & \geq \frac{\min \left\{ \frac{\mA}{\zet}, \beta_y \right\} }{\zet+\mA } (\Psi_\textnormal{b} +\P),
            \end{align*}
    and the conclusion readily follows by \Cref{prop:partialContractivity}.

            \subsubsection{\texorpdfstring{Item(iii): \Cref{alg:smoothForwardBackward} under \Cthree{}}{smoothC3}}

            We start by an auxiliary result, which refines \Cref{lemma:IL1}. 
    
\begin{lemma}
        \label{lemma:inverseLipschitzBackward}
        Let \Cthree{} hold. Then, the operator $F_{\text{b}} (\omega)= (A^\top y,-Ax)$ is $\ \sqrt{\mu_A}$-inverse Lipschitz.
    \end{lemma}

            %   \begin{proof}
            %     We have for $\omega, \omega' \in \R^{2n}$
            %     \begin{IEEEeqnarray*}{rCl}
            %         \norm{\Phi_{\tau, \sigma}^{-1}F_\textnormal{b}( \omega-\omega')}_{\Phi_{\tau, \sigma}}^2 & = & \innerproduct{\Phi_{\tau, \sigma}^{-1}F_\textnormal{b} ( \omega - \omega')}{F_\textnormal{b} ( \omega - \omega')}\\
            %         & \geq & \lambdamin{\Phi_{\tau, \sigma}^{-1}}\lambdamin{F_\textnormal{b}^\top F_\textnormal{b}}\innerproduct{\omega - \omega'}{\omega - \omega'}\\
            %         & \geq & \frac{\lambdamin{\Phi_{\tau, \sigma}^{-1}}\lambdamin{A^\top A}}{\lambdamax{\Phi_{\tau, \sigma}}}\norm{\omega - \omega'}_{\Phi_{\tau, \sigma}}^2
            %     \end{IEEEeqnarray*}
            %     using $\lambdamin{\Phi_{\tau, \sigma}^{-1}}= \frac{1}{\lambdamax{\Phi_{\tau, \sigma}}}$ and $\lambdamax{\Phi_{\tau, \sigma}}= \norm{\Phi_{\tau, \sigma}}$ concludes the proof.
            % \end{proof}  

             Note that  $\Phi_{\tau,\sigma}^{-1}\Fb$ is monotone in $\innerproduct{\cdot}{\cdot}_{\Phi_{\tau,\sigma}}$ by \Cref{lemma:monotonyPreconditionedInverse}. Furthermore, $F_\textnormal{b}$ is $\mA$-inverse Lipschitz according to \Cref{lemma:inverseLipschitzBackward}, and hence $\Phi_{\tau, \sigma}^{-1}F_\textnormal{b}$ is inverse Lipschitz in $\innerproduct{\cdot}{\cdot}_{\Phi_{\tau,\sigma}}$ with constant $\frac{\sqrt{\mA}}{\norm{\Phi_{\tau, \sigma}}}$. Hence, the backward step $\J_{\Phi_{\tau,\sigma}^{-1}F_\textnormal{b}}$ is contractive in $\|\cdot\|_\Phi$ with factor $\frac{\norm{\Phi_{\tau,\sigma}}}{\sqrt{\norm{\Phi_{\tau,\sigma}}^2+\mA}} \leq \frac{\zet} {\sqrt{ \zet^2+\mA}}$ (the proof is identical to that of  \Cref{prop:PropertiesOfResolvent}, with the only caution of replacing $\langle \cdot, \cdot \rangle$ and $\|\cdot\|$ with $\langle \cdot, \cdot \rangle_\P$ and $\|\cdot\|_\P$). Therefore, \Aone{} in \Cref{prop:partialContractivityB} is satisfied with $\Psi_\textnormal{b}= \0$, $\rho_\textnormal{b} = \rho$. 
             
             On the other hand, by \Cref{lemma:PDGforward} (with $\beta_x = \beta_y =0$), \Atwo{} in 
            \Cref{prop:partialContractivityB} is satisfied with $\Psi_\textnormal{f}=\0$, $\rho_\textnormal{f} = 1$ (i.e., the forward step is nonexpansive). Then, the contractivity result  follows by \Cref{prop:partialContractivity}. 
            
            To find the best step sizes, we note that $\rho$ is decreasing in $\zet$. Thus, the best step sizes are found by minimizing $\zet$ in \eqref{eq:zet}, i.e., maximizing $\min\{\tau,\sigma \}$ subject to 
            $\tau \leq \frac{2}{\Lf+ +2\|A\|\nu}$, $\sigma \leq \frac{2}{\Lg +2\|A\|\frac{1}{\nu}}$, $\tau \sigma \|A\|^2 (1+\epsilon)^2 \leq 1$. Clearly  $\sigma = \tau$ at a solution (otherwise, there exist $\tau' = \sigma' $ satisfying the step size bounds for some $\nu$, such that $\min \{ \sigma',\tau'\} > \min \{\sigma,\tau\}$). Therefore, the the optimal $\bar  \nu$ is obtained by solving for $\nu$ the equation $\tau = \frac{2}{\Lf +2\|A\| \nu } =   \frac{2}{\Lg +2\|A\|\frac{1}{\nu}} = \sigma$.
            
            Note that the bound $\tau \sigma \|A\|^2 (1+\epsilon)^2 \leq 1$ is used to ensure $\P \succ 0$ when both $\Lf$ and $\Lg$ are zero. For this particular case, similar considerations to \Cref{rem:degeneratePrec} hold for the choice of $\epsilon$; otherwise, $\epsilon = 0 $ can be chosen (since any $\epsilon$ small enough would be inconsequential for the bounds and optimal choice of step sizes). \hfill $\square$

\section{Extensions}
    \label{app:extensions}
        \subsection{Gradient Descent-Ascent}
            \label{sec:GradientDescentAscent}
            \begin{algorithm}
                \caption{Primal-dual gradient method}
                \label{alg:primalDualGradient}
                \begin{algorithmic}[1]
                    \Require  $x^0\in \R^n, y^0\in \R^m$, step size $\alpha >0$
                    \For{$k = 0, 1, \cdots$}
                        \State $x^{k+1} = x^k - \alpha \left (\nabla f(x^k)+A^\top y^k \right)$
                        \State $y^{k+1} = y^k - \alpha \left (\nabla g(y^k)-Ax^k \right)$
                    \EndFor
                \end{algorithmic}
            \end{algorithm}

            We assume that \Ctwo{} holds (namely,  that $g$ is $\Lg$-smooth and $\mg$-strongly convex, and $A$ has full column rank) and that $f$ is $\Lf$-smooth. Under these assumptions, the R-linear convergence of the  \gls{PDG} method for small-enough step sizes was studied first proved by \cite{duLinearConvergencePrimalDual2019}. Here, we show that the \gls{PDG} method is actually \emph{contractive} (hence, converges Q-linearly). The proof is very simple: we show that the saddle-point operator in \eqref{eq:saddlePointOperator} is strongly monotone and Lipschitz continuous  in a weighted space. Then the contractivity of the method follows by known results for the forward step, e.g., \cite[Prop.~26.16]{bauschkeConvexAnalysisMonotone2017}.
            
            % Many applications in \Cref{sec:introduction} correspond to this case, e.g. linear regression with smoothed L1 regularization \citep{duLinearConvergencePrimalDual2019}. We will present two different algorithms for this specific case, basic primal-dual gradient descent/ascent and a preconditioned version of it. The second algorithm was developed based on an analysis of the proofs of the first algorithm, showing the advantage of a generalized approach. 
            
            To streamline the presentation,  we consider the algorithm in \Cref{alg:primalDualGradient}, which only uses one step size $\alpha>0$ (instead of two step sizes $\tau$, $\sigma$, for which the analysis would be analogous).

            Let us define 
            \begin{equation}\label{eq:Phieta1}
                \Phi_{\eta} = \begin{bmatrix}
                    I_n &-\eta A^\top \\
                    -\eta A &I_m,
                \end{bmatrix}
            \end{equation}
            where $\eta >0$ is a parameter to be chosen, and $\Phi_\eta \succ 0 $ if $\eta < \frac{1}{\|A\|}$.

            \begin{lemma}
                \label{lemma:stronglyMonotoneForwardStep}
                If   $\eta >0$ is chosen small enough,  then $\Phi_\eta \succ 0$ and there is $\mu_\eta>0$ such that $F$ in \eqref{eq:saddlePointOperator} is $\mu_\eta$-strongly monotone in the space weighted by $\Phi_\eta$, i.e., for all $\omega,\omega' \in \R^{n+m}$
                \begin{equation*}
                    \innerproduct{F(\omega)-F(\omega')}{\omega-\omega'}_{\Phi_{\eta}} \geq \mu_\eta\norm{\omega-\omega'}^2_{\Phi_\eta}.
                \end{equation*}
            \end{lemma}
            
            \begin{proof}
                Let $\omega =(x,y),\omega' = (x',y') \in\R^{n+m}$ and let $\eta< \frac{1}{\|A\|}$, so that $\Phi_\eta \succ 0$. We have                \allowdisplaybreaks{\begin{IEEEeqnarray*}{rCl}
                    \IEEEeqnarraymulticol{3}{l}{
                    \langle F(\omega) - F(\omega'), \omega-\omega'\rangle_\P
                    % \innerproduct{\begin{bmatrix}
                    % x-x'\\
                    % y-y'
                    % \end{bmatrix}}{\begin{bmatrix}
                    %     \nabla f(x)-\nabla f(x') + A^\top (y-y')\\
                    %     \nabla g(y)-\nabla g(y')-A(x-x') \end{bmatrix}}_{\Phi_{\eta}}
                    }\nonumber \\
                    & = &
                    \innerproduct{\begin{bmatrix}
                    x-x'\\
                    y-y'
                    \end{bmatrix}}{\begin{bmatrix}
                        \nabla f(x)-\nabla f(x') + A^\top (y-y')\\
                        \nabla g(y)-\nabla g(y')-A(x-x') \end{bmatrix}}_{\Phi_{\eta}}
                    \\
                    & = &
                    \innerproduct{\begin{bmatrix}
                        x-x' - \eta A^\top (y-y')\\
                        y-y' - \eta A(x- x')
                    \end{bmatrix}}{\begin{bmatrix}
                        \nabla f(x)-\nabla f(x') + A^\top (y-y')\\
                        \nabla g(y)-\nabla g(y')-A(x-x')
                    \end{bmatrix}}\\
                    & = & \innerproduct{x-x'- \eta A^\top (y-y')}{ \nabla f(x)-\nabla f(x') + A^\top (y-y')}\nonumber\\
                    &&+ \innerproduct{ y-y' - \eta A(x- x')}{\nabla g(y)-\nabla g(y')-A(x-x')}\\
                    % & = & \innerproduct{x-x'}{A^\top (y-y')} - \innerproduct{y-y'}{A(x-x')} + \innerproduct{x-x'}{\nabla f(x)-\nabla f(x')} \nonumber\\
                    % && +\innerproduct{y-y'}{\nabla g(y)-\nabla g(y')} - \eta \innerproduct{A^\top (y-y')}{\nabla f(x)-\nabla f(x') + A^\top (y-y')} \nonumber\\
                    % && -\eta \innerproduct{A(x-x')}{\nabla g(y)-\nabla g(y')-A(x-x')}\\
                    & = & \innerproduct{x-x'}{\nabla f(x)-\nabla f(x')} + \innerproduct{y-y'}{\nabla g(y)-\nabla g(y')} \nonumber\\
                    && - \eta \innerproduct{A^\top (y-y')}{A^\top (y-y')} + \eta \innerproduct{A(x-x')}{A(x-x')} \nonumber\\
                    && - \eta \innerproduct{A^\top (y-y')}{\nabla f(x)-\nabla f(x')} - \eta \innerproduct{A(x-x')}{\nabla g(y)-\nabla g(y')}\\
                    & \overset{(i)}{\geq} & \mg\norm{y-y'}^2- \eta \innerproduct{AA^\top (y-y')}{y-y'} + \eta \innerproduct{A^\top A(x-x')}{x-x'}\nonumber\\
                    && - \eta \innerproduct{A^\top (y-y')}{\nabla f(x)-\nabla f(x')} - \eta \innerproduct{A(x-x')}{\nabla g(y)-\nabla g(y')}\\
                    & \geq & \mg\norm{y-y'}^2 - \eta \lambda_{\max}(AA^\top ) \norm{y-y'}^2+\eta \lambda_{\min}(A^\top A)\norm{x-x'}^2 \nonumber\\
                    && - \eta \innerproduct{A^\top (y-y')}{\nabla f(x)-\nabla f(x')} - \eta \innerproduct{A(x-x')}{\nabla g(y)-\nabla g(y')}\\
                    & \overset{(ii)}{\geq} & \mg\norm{y-y'}^2 - \eta \lambda_{\max}(AA^\top ) \norm{y-y'}^2+\eta \lambda_{\min}(A^\top A)\norm{x-x'}^2 \nonumber\\
                    && -\eta (\Lf + \Lg)\norm{A}\norm{x-x'}\norm{y-y'}\\
                    & = & \begin{bmatrix}
                        \norm{x-x'} \\
                        \norm{y-y'}
                    \end{bmatrix} \begin{bmatrix}
                         \eta \lambda_{\min}(A^\top A) &-\frac{1}{2}\eta (\Lf + \Lg)\norm{A}\\
                        -\frac{1}{2}\eta (\Lf + \Lg)\norm{A} &\mg-\eta \lambda_{\max}(AA^\top )
                    \end{bmatrix} \begin{bmatrix}
                        \norm{x-x'} \\
                        \norm{y-y'}
                    \end{bmatrix}\\
                    & \overset{(iii)}{\geq} & \lambda_{\min}(M_{\eta}) \norm{\begin{bmatrix}
                        x-x'\\
                        y-y'
                    \end{bmatrix}}^2,
                \end{IEEEeqnarray*}}
                where we used monotononicity of the gradient and strong convexity of $g$ in (i), in (ii) we used Cauchy-Schwarz and strong smoothness of $f$ and $g$, and  (iii) follows by defining
                \begin{equation}
                    M_{\eta} \coloneqq \begin{bmatrix}
                        \eta \lambda_{\min}(A^\top A) &-\frac{1}{2}\eta (\Lf + \Lg)\norm{A}\\
                        -\frac{1}{2}\eta (\Lf + \Lg)\norm{A} &\mg-\eta \lambda_{\max}(AA^\top )
                    \end{bmatrix}.
                    \label{eq:matrixMEta}
                \end{equation}
                
                We next show that $M_{\eta} \succ 0$ for small enough $\eta$. 
                We have
                \begin{align*}
                     \operatorname{det} (M_{\eta} ) &=  (\eta \lambda_{\min}(A^\top A))(\mg-\eta \lambda_{\max}(AA^\top ))-\frac{1}{4} \eta^2 (\Lf + \Lg)^2\norm{A}^2,
                \end{align*}
                which is positive if and only if
                \begin{equation}
                    \eta < \frac{\mg\lambda_{\min}(A^\top A)}{\lambda_{\min}(A^\top A)\lambda_{\max}(AA^\top )+\frac{1}{4} (\Lf + \Lg)^2\norm{A}^2} \eqqcolon C_M.
                    \label{eq:constantMEta}
                \end{equation}
                Since $\eta \lambda_{\min}(A^\top A)>0$, from Sylvester's criterion  it follows that $M_{\eta} \succ 0$ for all $\eta$ small enough.
                In particular,  for any $\eta < \min\{\frac{1}{\|A\|},C_M \}$, the conclusion follows with
                \begin{align}
                    \mu_\eta = \frac{\lambdamin{M_\eta}}{\| \Phi_\eta\|} >0.
                \end{align}
                
            \end{proof}
            
            \begin{lemma}
                \label{lemma:lipschitzForwardStep}
               Let $\eta < \frac{1}{\norm{A}}$. Then, the saddle-point operator $F$ in \eqref{eq:saddlePointOperator} is Lipschitz in $\norm{\cdot}_{\Phi_\eta}$ with constant  $ L_\eta  = \sqrt{\frac{\lambdamax{\Phi_\eta}}{\lambdamin{\Phi_\eta}}}\sqrt{ \max \{\Lf, \Lg\} ^2 + \norm{A}^2} >0 $.
            \end{lemma}

            \begin{proof}
                Let $x,x' \in \R^n$ and $y, y' \in \R^m$ be arbitrary. We have
                \begin{IEEEeqnarray*}{rCl} 
            \frac{1}{\lambdamax{\Phi_\eta} }\norm{F \left (\begin{bmatrix}
                    x\\
                    y
                \end{bmatrix}\right ) - F \left (\begin{bmatrix}
                    x'\\
                    y'
                \end{bmatrix}\right )}^2_{\Phi_\eta} & \leq & 
                    \norm{F \left (\begin{bmatrix}
                    x\\
                    y
                \end{bmatrix}\right ) - F \left (\begin{bmatrix}
                    x'\\
                    y'
                \end{bmatrix}\right )}^2 
                \\
                & = & \norm{\begin{bmatrix}
                     \nabla f(x)-\nabla f(x')\\
                      \nabla g(y)-\nabla g(y')
                \end{bmatrix} + \begin{bmatrix}
                    A^\top (y-y')\\
                    -A(x-x')
                \end{bmatrix}}^2\\
                & \leq & \norm{\begin{bmatrix}
                     \nabla f(x)-\nabla f(x')\\
                      \nabla g(y)-\nabla g(y')
                \end{bmatrix}}^2 + \norm{\begin{bmatrix}
                     A^\top (y-y')\\
                    -A(x-x')
                \end{bmatrix}}^2\\
                & = & \norm{\nabla f(x)-\nabla f(x')}^2 + \norm{ \nabla g(y)-\nabla g(y')}^2 \nonumber\\
                && \norm{ A^\top (y-y')}^2 + \norm{A(x-x')}^2\\
                & \overset{(i)}{\leq} & \left( \Lf^2+ \norm{A}^2\right) \norm{x-x'}^2+\left( \Lg^2+ \norm{A}^2\right) \norm{y-y'}^2 \nonumber \\
                & \leq & 
                \left(\max \{\Lf, \Lg\} ^2 + \norm{A}^2 \right)
                \norm{\begin{bmatrix}
                    x - x'\\
                    y - y'
                \end{bmatrix}}^2
                \\
                & \leq & 
                \frac{1}{\lambdamin{\Phi_\eta }}\left(\max \{\Lf, \Lg\} ^2 + \norm{A}^2 \right)
                \norm{\begin{bmatrix}
                    x - x'\\
                    y - y'
                \end{bmatrix}}^2_{\Phi_\eta},
                \end{IEEEeqnarray*}
                where we used the fact that $f$ and $g$ are smooth in (i).
            \end{proof}

            % Since both functions $f$ and $g$ are smooth, the gradients are singular and we choose the following splitting
            % \begin{equation*}
            %     F_\textnormal{b} = 0 \qquad F_\textnormal{f} = \alpha F
            % \end{equation*}
            % Because the backward step corresponds to the identity operator, \Cref{alg:primalDualGradient} is feasible without preconditioning. To have a contraction, we know from \citep{ryuPRIMERMONOTONEOPERATOR2016} that $F$ must be Lipschitz, strongly monotone and the stepsize $\alpha$ must be small enough. $F$ is not strongly monotone in the Euclidean inner product, however, it will be in $\innerproduct{\cdot}{\cdot}_{\Phi_{\eta}}$ for $\eta > 0$ small enough. First, we need the following result

            % \
            
            %  \begin{lemma}
            %     \label{lemma:matrixM}
            %     The matrix $M_{\eta} \in \R^{2\times 2}$
            %     \begin{equation}
            %         M_{\eta} = \begin{bmatrix}
            %             \eta \lambda_{\min}(A^\top A) &-\frac{1}{2}\eta (\Lf + \Lg)\norm{A}\\
            %             -\frac{1}{2}\eta (\Lf + \Lg)\norm{A} &\mg-\eta \lambda_{\max}(AA^\top )
            %         \end{bmatrix}
            %         \label{eq:matrixMEta}
            %     \end{equation}
            %     is positive definite for $\eta>0$ small enough.
            % \end{lemma}
    
            % \begin{proof}
               
            % \end{proof}

            % The two properties of $F$ can be formalized in the following Lemmas
            
           Using the previous lemmas, the following contractivity  result is straightforward. 
            \begin{theorem}
                \label{thm:forwardStepMethod}
                Let $f: \R^n \rightarrow \R$  be convex and $\Lf$ smooth.
                Let $g: \R^m \rightarrow \R$ be $\Lg$ smooth and $\mg$ strongly convex. Let $\mA = \lambdamin{A^\top A}>0$. Let $\eta < \min \left\{\frac{1}{\norm{A}}, C_M \right\}$ with $C_M$ as in \eqref{eq:constantMEta},  $\Phi_\eta \succ 0$ as in \eqref{eq:Phieta1}, $\mu_\eta$ and $L_\eta$  as in \Cref{lemma:stronglyMonotoneForwardStep} and \Cref{lemma:lipschitzForwardStep}, respectively. Then, for any  $0< \alpha < \frac{2 \mu_\eta}{L_\eta^2}$, \Cref{alg:primalDualGradient} is contractive in $\|\cdot\|_{\Phi_\eta}$ with rate
                \begin{align}
                    \rho = \sqrt{1 - 2 \alpha \mu_\eta+ \alpha^2 L_\eta^2 } <1.
                \end{align}
                Thus, for all $k\in \mathbb{N}$, 
                \begin{equation*}
                    \norm{\begin{bmatrix}
                        x^{k}-x^{\star}\\
                        y^{k}-y^{\star}\end{bmatrix}}_{\Phi_{\eta}} \leq \rho^{k}\norm{\begin{bmatrix}
                        x^0 - x^{\star}\\
                        y^0 - y^{\star}\end{bmatrix}}_{\Phi_{\eta}},
                \end{equation*}
                where $(x^\star,y^\star)$ is the unique solution to problem \eqref{eq:bilinearSaddlePointProblem}.
                % where 
                % \begin{equation*}
                %     \rho_{\alpha} = \sqrt{1- \frac{2\alpha\lambdamin{M_{\eta}}}{\lambdamax{\Phi_{\eta}}}+\alpha^2 \kappa(\Phi_{\eta}) \left( \norm{A}^2+L^2 \right)}
                % \end{equation*}
                % and $\rho_{\alpha}< 1$ if $\alpha < \frac{2 \lambdamin{M_{\eta}}}{(L^2+\norm{A}^2) \kappa(\Phi_{\eta})\lambdamax{\Phi_{\eta}}}$.
            \end{theorem}

            \begin{proof}
                Clearly, any fixed point of \Cref{alg:primalDualGradient} is a solution to \eqref{eq:bilinearSaddlePointProblem}. The theorem then readily follows in view of \Cref{lemma:stronglyMonotoneForwardStep} and \Cref{lemma:lipschitzForwardStep}, see for instance \cite[Prop.~26.16(ii)]{bauschkeConvexAnalysisMonotone2017}.             
            \end{proof}

        \subsection{Preconditioned Gradient Descent-Ascent}

\begin{algorithm}
        \caption{Preconditioned primal-dual gradient method}
        \label{alg:precPrimalDualGradient}
        \begin{algorithmic}[1]
            \Require $x^0\in \R^n, y^0\in \R^m$, step sizes $\alpha >0, \eta > 0$
            \For{$k = 0, 1, \cdots$}
                \State $a^k  =  \nabla f(x^k)+A^\top  y^k$
                \State $b^k  =  \nabla g(y^k) - A x^k$
                \State $x^{k+1}  =  x^k - \alpha \left (a^k - \eta A^\top b^k\right )$
                \State $y^{k+1}  =  y^k - \alpha \left (b^k - \eta Aa^k \right )$
            \EndFor
        \end{algorithmic}
    \end{algorithm} \Cref{lemma:stronglyMonotoneForwardStep} shows that  saddle point operator $F$ is strongly monotone in $\|\cdot \|_{\Phi_\eta}$, when \Ctwo{} holds and $f$ is smooth.  This is an interesting result, and it paves the way for several extensions. For instance, from the proof of \Cref{lemma:stronglyMonotoneForwardStep}, it is immediate to see that, for all $\omega, \omega' \in \R^{m+n}$
    \begin{align*}
        \langle \Phi_\eta F(\omega) - \Phi_\eta F(\omega'), \omega -\omega' \rangle & =  \langle  F(\omega) -  F(\omega'), \omega -\omega' \rangle_{\Phi_\eta} \\
        & \geq \lambda_{\min}(M_{\eta}) \|\omega -\omega'\|^2,
    \end{align*}
         namely, $\Phi _\eta F$   is strongly monotone in the (unweighted) inner product $\innerproduct{\cdot}{\cdot}$. This suggests to use $\Phi _\eta F$ in the forward step, resulting in \Cref{alg:precPrimalDualGradient}. This algorithm is novel to the best of our knowledge; as for \eqref{alg:smoothForwardBackward}, it requires two sequential updates for each iteration, and it is related to the (accelerated) algorithm in \cite{kovalevAcceleratedPrimalDualGradient2022}. Its peculiarity is  to be contractive in the unweighted norm $\| \cdot\|$.
         
         % hence choosing $F_\textnormal{f} = \Phi_{\eta}F$ as the forward step is beneficial for the rate of convergence. To our knowledge, this algorithm has never been presented and can be viewed as an accelerated version of \Cref{alg:primalDualGradient}. The following Theorem summarizes the findings
            
            \begin{theorem}
                \label{thm:preconditionedForwardStep}
                Let $f: \R^n \rightarrow \R$  be convex and $\Lf$ smooth.
                Let $g: \R^m \rightarrow \R$ be $\Lg$ smooth and $\mg$ strongly convex. Let $\mA = \lambdamin{A^\top A}>0$. Let $\eta < \min \left\{\frac{1}{\norm{A}}, C_M \right\}$ with $C_M$ as in \eqref{eq:constantMEta},  $\Phi_\eta \succ 0$ as in \eqref{eq:Phieta1}, $\mu_\eta$ and $L_\eta$  as in \Cref{lemma:stronglyMonotoneForwardStep} and \Cref{lemma:lipschitzForwardStep}, respectively, $M_\eta$ as in \eqref{eq:matrixMEta}.
                Let $\alpha < \frac{2 \lambdamin{M_{\eta}}}{(L^2+\norm{A}^2)\lambda_{\max}^2(\Phi_{\eta})}$. Then \Cref{alg:precPrimalDualGradient} is contractive in $\|\cdot\|$ with rate
                \begin{equation*}
                    \rho = \sqrt{1- 2\alpha\lambdamin{M_{\eta}}+\alpha^2 \lambda_{\max}^2(\Phi_{\eta}) \left( \max\{\Lf,\Lg\}^2+\norm{A}^2 \right)} <1.
                \end{equation*}
                Thus, for all $k \in \mathbb N$,
                \begin{equation*}
                    \norm{\begin{bmatrix}
                        x^{k}-x^{\star}\\
                        y^{k}-y^{\star}
                    \end{bmatrix}} \leq \rho^{k} \norm{\begin{bmatrix}
                        x^0 - x^{\star}\\
                        y^0 - y^{\star}
                    \end{bmatrix}},
                \end{equation*}
                where $(x^\star,y^\star)$ is the unique solution to \eqref{eq:bilinearSaddlePointProblem}.
            \end{theorem}

            \begin{proof}
            \Cref{alg:precPrimalDualGradient} can be written as $\omega^{k+1} = (\operatorname{Id} - \alpha \Phi_\eta F) (\omega^k)$. Since $\Phi_\eta \succ 0$, the fixed points of \Cref{alg:precPrimalDualGradient} coincide with the zeros of $F$, equivalently withe the solutions to \eqref{eq:bilinearSaddlePointProblem}. The result follows because the operator $\Phi_\eta F$ is strongly monotone and Lipschitz continuous (in $\| \cdot\|$), see for instance \cite[Prop.~26.16(ii)]{bauschkeConvexAnalysisMonotone2017}. 
            \end{proof}
            
\end{document}